\documentclass[a4paper,10pt]{amsart}
\usepackage{graphicx}
\usepackage[UKenglish]{babel}
\usepackage[margin=3cm]{geometry}
\usepackage[utf8]{inputenc}
\usepackage{amsthm}
\usepackage{amsmath}
\usepackage{amsfonts}
\usepackage{amssymb}
\usepackage{bbm}
\usepackage{float}
\usepackage{xcolor}
\usepackage{diagbox}
\usepackage[colorlinks,
urlcolor=blue!50!black,
citecolor=red!50!black,
linkcolor=blue!50!black]{hyperref}
\usepackage{caption}
\usepackage{subcaption}
\usepackage{csquotes}
\usepackage{multicol}
\usepackage{multirow}
\usepackage{enumerate}
\usepackage{mathrsfs}
\numberwithin{equation}{section}

\theoremstyle{plain}
\newtheorem{thm}{Theorem}[section]

\newtheorem{lemma}[thm]{Lemma}
\newtheorem{prop}[thm]{Proposition}

\newtheorem{remark}[thm]{Remark}

\theoremstyle{definition}
\newtheorem{defn}[thm]{Definition}

\newcommand{\fd}{\mathfrak{d}}
\newcommand{\ve}{\varepsilon}
\newcommand{\SLE}{\mathrm{SLE}}
\newcommand{\h}{\mathbb{H}}
\newcommand{\p}{\mathbb{P}}
\newcommand{\R}{\mathbb{R}}

\title[Metric gluing and conformal welding in $\gamma$-Liouville quantum gravity]{Equivalence of metric gluing and conformal welding in $\gamma$-Liouville quantum gravity for $\gamma \in (0,2)$}
\author{Liam Hughes and Jason Miller}
\date{\today}

\begin{document}
\maketitle
\begin{abstract}
	We consider the $\gamma$-Liouville quantum gravity (LQG) model for $\gamma \in (0,2)$, formally described by $e^{\gamma h}$ where $h$ is a Gaussian free field on a planar domain $D$. Sheffield showed that when a certain type of LQG surface, called a quantum wedge, is decorated by an appropriate independent SLE curve, the wedge is cut into two independent surfaces which are themselves quantum wedges, and that the original surface can be recovered as a unique conformal welding. We prove that the original surface can also be obtained as a metric space quotient of the two wedges, extending results of Gwynne and Miller in the special case $\gamma = \sqrt{8/3}$ to the whole subcritical regime $\gamma \in (0,2)$.
	
	Since the proof for $\gamma = \sqrt{8/3}$ used estimates for Brownian surfaces, which are equivalent to $\gamma$-LQG surfaces only when $\gamma=\sqrt{8/3}$, we instead use GFF techniques to establish estimates relating distances, areas and boundary lengths, as well as bi-H\"older continuity of the LQG metric w.r.t.\ the Euclidean metric at the boundary, which may be of independent interest. 
\end{abstract}

	\setcounter{tocdepth}{1}
	\tableofcontents
	\section{Introduction}
	\subsection{Liouville quantum gravity surfaces}
	A \emph{$\gamma$-Liouville quantum gravity} (LQG) surface is a random surface parametrized by a domain $D \subseteq \mathbb{C}$ given, in a formal sense, by the random metric tensor 
	\begin{equation}\label{eq:lqg}
		e^{\gamma h(z)} \, (dx^2+dy^2)
	\end{equation}
	where $\gamma$ is a parameter in $(0,2)$, $h$ is some form of the Gaussian free field (GFF) on the domain $D$ and $dx^2+dy^2$ is the Euclidean metric. These surfaces have been shown to arise as scaling limits of several random planar map models (\cite{inventory,kmsw,schnyder,gm_saw,perc_uq,gkmw}; see also~\cite{ghs} and the references therein).
	
	Since $h$ is not sufficiently regular to be a random function on $D$ (it is only a distribution on $D$, in the sense of Schwartz), the expression (\ref{eq:lqg}) for the LQG metric tensor does not make literal sense; in order to rigorously define an LQG surface one must take a limit of regularized versions of $e^{\gamma h(z)} dz$. This was done for the volume form in~\cite{ds}, resulting in the \emph{$\gamma$-LQG measure} $\mu_h$, a random measure on $D$, and the \emph{$\gamma$-LQG boundary length} $\nu_h$, a random measure on $\partial D$, each of which fall under the general framework of Kahane's \emph{Gaussian multiplicative chaos}, as introduced in \cite{kahane}. These measures are conformally covariant in the following sense: given a conformal map $\psi\colon\widetilde{D}\to D$, if we set
	\begin{equation}\label{eq:reparam}
		Q = \frac{2}{\gamma} + \frac{\gamma}{2}, \quad \widetilde{h} = h \circ \psi + Q\log{|\psi'|},
	\end{equation}
	then by~\cite[Prop.~2.1]{ds}, almost surely we have $\mu_{\widetilde{h}} = \mu_h \circ \psi$ and (provided $\psi$ extends to a homeomorphism between the closures of $\widetilde{D}$ and $D$ in the Riemann sphere) $\nu_{\widetilde{h}} = \nu_h \circ \psi$.
	
	We can then consider various types of \emph{quantum surfaces}, random surfaces that can be described by $(D,h)$ where $D$ is a domain in $\mathbb{C}$ and $h$ some form of the GFF on $D$, with (random) quantum area and boundary length measures given respectively by $\mu_h$ and $\nu_h$, and which are defined as equivalence classes of pairs $(D,h)$ that are related by conformal reparametrizations as described by (\ref{eq:reparam}). A particular one-parameter family of such surfaces are the \emph{$\alpha$-quantum wedges} for $\alpha \le Q$. An $\alpha$-quantum wedge is parametrized by $\h$ with marked points at~0 and~$\infty$, and is given by $h-\alpha \log{|\cdot|}$ where $h$ is a variant of the free-boundary GFF on $\h$ chosen so that the law of the resulting surface is invariant under the operation of replacing $h$ with $h + c$ for $c\in \R$. For any $\alpha \in (-\infty,Q)$, this surface is homeomorphic to $\h$, and is referred to as a \emph{thick} quantum wedge, as in~\cite[\S4.2]{dms}. The starting point for an alternative but equivalent definition~\cite[Def.~4.15]{dms} is a Bessel process of dimension
	\begin{equation*}
		\delta := 2 + \frac{2(Q-\alpha)}{\gamma}; 
	\end{equation*}
	this can be used to extend the definition to include $\alpha \in (Q,Q+\gamma/2)$. For such $\alpha$, the Bessel process has dimension in $(1,2)$ and thus hits zero, and one no longer obtains a single surface homeomorphic to~$\h$; for each excursion of the Bessel process away from 0 one obtains a surface with the topology of the disc, and concatenating all these surfaces (the \emph{beads} of the wedge) gives a \emph{thin} quantum wedge, as seen in~\cite[\S4.4]{dms}. Instead of using the parameters $\alpha$ or $\delta$, it is often more convenient to consider the value
	\begin{equation*}
		\mathfrak{w} = \gamma \left(\frac{\gamma}{2} + Q-\alpha \right),
	\end{equation*}
	called the \emph{weight} of the wedge.
	
	Different kinds of quantum surfaces include \emph{quantum cones}, which are homeomorphic to $\mathbb{C}$, and \emph{quantum spheres}, which are homeomorphic to the Riemann sphere (and can thus be parametrized by the bi-infinite cylinder $\mathscr{C}$ given by $\R\times[0,2\pi]$ with $\R\times\{0\}$ and $\R\times\{2\pi\}$ identified and the points~$-\infty$ and~$+\infty$ added). By~\cite[Thm~1.5]{dms}, a quantum cone of weight $\mathfrak{w}$ is the surface that results when the two sides $(-\infty,0)$ and $(0,\infty)$ of a quantum wedge of weight $\mathfrak{w}$ are conformally welded together. As with wedges, there are choices of parameter other than the weight parameter~$\mathfrak{w}$. A quantum cone of weight $\mathfrak{w}$ can be referred to as an $\alpha$-quantum cone, where the parameter $\alpha$ corresponding to the log singularity of the field and the weight $\mathfrak{w}$ are related by $\alpha = Q-\mathfrak{w}/(2\gamma)$. A quantum sphere of weight $\mathfrak{w}$ is a compact finite-volume surface constructed so as to look like a quantum cone of weight $\mathfrak{w}$ near each of its endpoints $-\infty$ and $+\infty$.
	
	It was proven in \cite{lqg1,lqg2,lqg3} for $\gamma = \sqrt{8/3}$, and later in \cite{gm} for $\gamma \in (0,2)$, there is a unique random metric $\fd_h$, measurable w.r.t.\ the GFF $h$, that satisfies a certain list of axioms associated with LQG ($\fd_h$ is required to induce the Euclidean topology and to transform appropriately under affine coordinate changes and adding a continuous function to $h$, and must also be a length metric locally determined by~$h$). This metric arises as a subsequential limit of \emph{Liouville first passage percolation} (LFPP), a family of random metrics obtained from a regularized version of the GFF; existence of such subsequential limits was established in~\cite{dddf}, and building on~\cite{dfgps}, the article~\cite{gm} then showed that the limit is unique and satisfies the requisite axioms. (More recently in~\cite{dg_supercrit} the \emph{critical} LQG metric corresponding to $\gamma = 2$ was constructed and proven to be unique, as were \emph{supercritical} LQG metrics corresponding to complex values of $\gamma$ with $|\gamma|=2$.)
	
	The result~\cite[Thm~1.8]{zip} (later generalized by~\cite[Thm~1.2]{dms}) says that when a certain quantum wedge $\mathcal{W}$ is cut by an appropriate independent random curve $\eta$, the regions to the left and right of $\eta$ (call them~$\mathcal{W}^-$,~$\mathcal{W}^+$ respectively) are independent quantum wedges; moreover, the original wedge~$\mathcal{W}$ and curve~$\eta$ may be reconstructed by \emph{conformally welding} the right side of $\mathcal{W}^-$ to the left side of $\mathcal{W}^+$ according to $\gamma$-LQG boundary length. The curve~$\eta$ is a variant of Schramm's \cite{schramm} SLE---more specifically it is an $\mathrm{SLE}_\kappa(\rho_1;\rho_2)$, as first defined in~\cite[\S8.3]{lsw}.
	
	Though we will not need it in this paper, we briefly discuss what is meant here by ``conformal welding''. Given a homeomorphism between boundary arcs of two topological surfaces, one can obtain a new surface by gluing along the boundary arcs; if the two original surfaces are each endowed with a conformal structure, the problem of \emph{conformally welding} them is that of obtaining a conformal structure on the glued surface compatible with those on the original surfaces. In the setting of the previous paragraph, it turns out~\cite[Thm~1.3]{zip} that the LQG boundary length measures on the boundaries of $\mathcal{W}^-$ and $\mathcal{W}^+$ agree for segments of $\eta$. This allows the glued surface~$\mathcal{W}$ to be recovered from the two pieces $\mathcal{W}^-$ and $\mathcal{W}^+$. Indeed, if $\mathcal{W}^-$ and $\mathcal{W}^+$ are reparametrized by~$\h$ with corresponding fields $h^-$ and $h^+$, then we can define a homeomorphism
	$$ \psi: [0,\infty) \to (-\infty,0] $$
	from the right-hand boundary arc of $\mathcal{W}^-$ to the left-hand boundary arc of $\mathcal{W}^-$ via the equation
	\begin{equation*}
		\nu_{h^-} ([0,x]) = \nu_{h^+}([\psi(x),0]), \quad x\in(0,\infty).
	\end{equation*}
	Crucially, $\psi$ is uniquely determined by $\mathcal{W}^-$ and $\mathcal{W}^+$ as surfaces (i.e., modulo reparametrization as in (\ref{eq:reparam})). We can glue the surfaces together by identifying each point $x\in[0,\infty) \subset \partial \mathcal{W}^-$ with its corresponding point $\psi(x) \in (-\infty,0] \subset \partial \mathcal{W}^+$; then by a \textbf{conformal welding} of $\mathcal{W}^-$ and $\mathcal{W}^+$ along $\psi$ we mean a map from the resulting space into $\h$ that is conformal on the interiors of~$\mathcal{W}^-$ and $\mathcal{W}^+$. In this case the glued space is the original surface $\mathcal{W}$, so such a map is given by a parametrization of $\mathcal{W}$ by~$\h$. In this case, this map is in fact (up to conformal automorphisms of~$\h$) the \emph{unique} conformal welding of $\mathcal{W}^-$ and $\mathcal{W}^+$ along $\psi$, so that both the original surface $\mathcal{W}$ and the SLE-type interface $\eta$ can be recovered from $\mathcal{W}^-$ and $\mathcal{W}^+$ (see~\cite[Thm~1.4]{zip}).
	\subsection{Metric gluing}
	Since these conformal welding uniqueness results do not give an explicit way to reconstruct the original surface, for applications a more explicit way to glue surfaces together may be required. In the case $\gamma = \sqrt{8/3}$, the theorem~\cite[Thm~1.5]{metglu} states that the $\gamma$-LQG metric on $\mathcal{W}$ can be obtained by \emph{metrically gluing} those on $\mathcal{W}^-$ and $\mathcal{W}^+$ along the conformal welding interface~$\eta$ according to $\gamma$-LQG boundary length, i.e.~as a quotient of the two metric spaces $\mathcal{W}^-$ and $\mathcal{W}^+$ under the identification of points given by the welding homeomorphism. This theorem---stating that conformal welding and metric gluing give the same result---was an essential input into the proof in~\cite{gm_saw} that the self-avoiding walk (SAW) on random quadrangulations converges to $\SLE_{8/3}$ on $\sqrt{8/3}$-LQG. Indeed, one can construct a SAW-decorated random quadrangulation by performing a discrete graph gluing of two quadrangulations with boundary, and~\cite{gm_saw} shows that this construction converges to an analogous one in the continuum using quantum wedges; the result of~\cite{metglu} then applies to show that we get the same surface by \emph{first} passing to the scaling limit of each of the two original quadrangulations and \emph{then} performing the metric gluing in the continuum. The importance of~\cite[Thm~1.5]{metglu} here is that, whilst metric gluing and conformal welding both provide ways to show that an LQG surface is determined by the two surfaces formed by cutting along an independent SLE, metric gluing recovers the original surface via a construction that has a direct discrete analogue, namely the graph gluing.
	
	The notion of ``metric gluing'' here is the natural way to define the quotient space obtained from identifying two metric spaces along a common subset; we define it below.
	\begin{defn}[metric gluing]\label{def:metglu}
		Let $(X,d_X)$ and $(Y,d_Y)$ be \emph{pseudometric spaces} (that is, $d_X$ satisfies all the conditions to be a metric on $X$ except that it need not be positive definite, and likewise for $d_Y$ on $Y$). Let $f$ be a function from a subset of $X$ to a subset of $Y$. Let $\sim$ be the finest equivalence relation on $X\sqcup Y$ such that $x \sim f(x)$ for each $x$ in the domain of $f$, and for each $x \in X\sqcup Y$ let $[x]$ be the equivalence class of $x$ under $\sim$. Define $d'$ on $(X \times X) \sqcup (Y \times Y)$ to equal~$d_X$ on~$X \times X$ and~$d_Y$ on~$Y \times Y$. Then the \textbf{metric gluing} of $X$ and $Y$ along $f$ is the quotient space $(X \sqcup Y)/\sim$ equipped with the \emph{gluing pseudometric} $d$ defined by
		\begin{equation}
			d([x],[y]) = \inf \sum_{i=1}^n d'(x_i,y_i)
		\end{equation}
		where the infimum is over all $n\in\mathbb{N}$ and all sequences $x_1, y_1, x_2, y_2 \ldots, x_n, y_n$ in $X \sqcup Y$ such that $x_1 \in [x]$, $y_n \in [y]$, and $x_{i+1} \sim y_i$ for each $i \in \{1, \ldots, n-1\}$, such that the sum is defined (so for each $i$ we must have~$x_i$ and~$y_i$ either both in $X$ or both in $Y$). If $(X_i,d_i)$ are pseudometric spaces for $i\in I$, we can define the \textbf{metric quotient} of the $X_i$ by an equivalence relation $\sim$ on $X := \bigsqcup_{i\in I}X_i$ by defining the partial function $d'$ on $\bigsqcup_{i\in I} (X_i \times X_i)$ and the gluing pseudometric $d$ on $X/\sim$ in the same way as above.
	\end{defn}
	Note that this $d$ is easily verified to be a pseudometric; in fact, it is the largest pseudometric on the quotient space which is bounded above by $d'$. In the case of~\cite[Thm~1.5]{metglu}, the gluing function $f$ sends a point $z$ on the right-hand part of $\partial \mathcal{W}^-$ to the point $w$ on the left-hand part of~$\partial \mathcal{W}^+$ such that the boundary segments from 0 to $z$ and from 0 to $w$ have equal $\gamma$-LQG boundary length. 
	\subsection{Main results}
	In the light of the construction of the $\gamma$-LQG metric for all $\gamma \in (0,2)$, the main result of this paper extends~\cite[Thm~1.5]{metglu}, giving the analogous statement for the $\gamma$-LQG metric for all values of $\gamma \in (0,2)$. In order to state the result, we need to define what it means for a metric defined on a subspace to extend by continuity to a larger set:
	\begin{defn}
		Let $(X,\tau)$ be a topological space and $Y$ a subset of $X$. If $d$ is a metric on $Y$ that is continuous w.r.t.\ the subspace topology induced by $\tau$ on $Y$ and $Z \subseteq X\setminus Y$, then we say $d$ \textbf{extends by continuity (w.r.t.\ $\tau$) to $Z$} if there exists a metric $d'$ on $Y\cup Z$ which agrees with $d$ on $Y$ and is continuous w.r.t.\ the subspace topology induced by $\tau$ on $Y\cup Z$.
	\end{defn}
	Note that if $Y$ is dense in $Y\cup Z$, there is at most one metric $d'$ extending $d$ by continuity to $Z$.
	\begin{thm}\label{thm:main}
		Let $\gamma \in (0,2)$, $\mathfrak{w}^{-}$, $\mathfrak{w}^{+}>0$ and $\mathfrak{w} = \mathfrak{w}^{-}+\mathfrak{w}^{+}$. Let $(\h,h,0,\infty)$ be a quantum wedge of weight $\mathfrak{w}$ if $\mathfrak{w}\ge \gamma^2/2$, or a single bead of a quantum wedge of weight $\mathfrak{w}$ with area $\mathfrak{a}>0$ and left and right boundary lengths $\mathfrak{l}^-$, $\mathfrak{l}^+>0$ otherwise. Independently sample an $\mathrm{SLE}_{\gamma^2}(\mathfrak{w}^{-}-2;\mathfrak{w}^{+}-2)$ process $\eta$ from 0 to $\infty$ in $\h$ with force points at $0^-$ and $0^+$. Denote the regions to the left and right of $\eta$ by $W^{-}$ and $W^{+}$ respectively, and let $\mathcal{W}^\pm$ be the quantum surface obtained by restricting~$h$ to~$W^\pm$. Let $\mathcal{U}^\pm$ be the ordered sequence of connected components of the interior of $W^\pm$, and let $\fd_h$, $\fd_{h|_{\mathcal{U}^-}}$, $\fd_{h|_{\mathcal{U}^+}}$ respectively be the $\gamma$-LQG metrics induced by $h$, $h|_{\mathcal{U}^-}$ and $h|_{\mathcal{U}^+}$. Then~$\fd_{h|_{\mathcal{U}^-}}$,~$\fd_{h|_{\mathcal{U}^+}}$ respectively extend by continuity (w.r.t.\ the Euclidean metric) to $\partial \mathcal{U}^-$, $\partial \mathcal{U}^+$ and $(\h, \fd_h)$ is obtained by metrically gluing $(\overline{\mathcal{U}^-},\fd_{h|_{\mathcal{U}^-}})$ and $(\overline{\mathcal{U}^+},\fd_{h|_{\mathcal{U}^+}})$ along $\eta$ according to $\gamma$-LQG boundary length.
	\end{thm}

	Although we have no specific application in mind, this result is potentially useful in proving convergence of a path-decorated lattice model in the scaling limit to $\gamma$-LQG decorated by an $\mathrm{SLE}_{\gamma^2}$-type curve, as it would play the role of~\cite[Thm~1.5]{metglu} in an argument along the lines of~\cite{gm_saw}.
	
	A statement weaker than Theorem~\ref{thm:main} follows straightforwardly from a \emph{locality} property in the definition of the LQG metric, which gives that the $\fd_{h|_{\mathcal{U}^-}}$-distance between points in $\mathcal{U}^-$ coincides with the infimum of the $\fd_h$-lengths of paths between the points that stay in~$\mathcal{U}^-$, and likewise for~$\mathcal{U}^+$. It is important to note that this property does \emph{not} imply that the metric gluing recovers $(\h, \fd_h)$; Theorem~\ref{thm:main} is stronger because it rules out certain pathologies which can arise from metric gluings along badly behaved interfaces (note that the interfaces along which we are gluing are SLE-type curves and thus fractal).
	
	One such pathology can occur when the function used to identify boundary segments is insufficiently well behaved: for instance, using a Cantor-type function can collapse the gluing interface to a point (see~\cite[Lemma~2.2]{metglu}). This kind of behaviour does not occur in our setting, since we know that $\fd_h$ is a pseudometric on the glued space that is bounded above by the partial function~$d'$ constructed from $\fd_{h|_{\mathcal{U}^-}}$ and $\fd_{h|_{\mathcal{U}^+}}$, whereas the gluing pseudometric is always the \emph{largest} such pseudometric (and thus in this case is a bona fide metric). The main issue for us is that the definition of the gluing metric only considers paths which cross the gluing interface $\eta$ finitely many times, whereas paths in $(\h, \fd_h)$ which cross the interface infinitely many times might a~priori be significantly shorter.
	
	Though both this metric gluing result and the conformal welding result~\cite[Thm~1.2]{dms} for the same surfaces can be thought of as saying that we can recover the original surface from the two pieces it is cut into by the SLE-type curve~$\eta$, the conformal welding result is more of an abstract measurability statement, whereas the metric gluing result shows concretely how one can reconstruct the metric on the original surface. Nevertheless, we might intuitively expect that, since one result is true, so should the other be---so that we avoid the pathologies that are generally liable to arise from metric gluing along fractal curves. In our setting, we aim to rule out pathological behaviour of $\fd_h$-geodesics hitting $\eta$. An analogous problem in the conformal setting is to show that a curve is \emph{conformally removable}, i.e.\ that any homeomorphism of $\mathbb{C}$ that is conformal off the image of the curve must in fact be conformal everywhere. Indeed, the reason that the conformal welding is unique---that $\eta$ can be recovered from the two surfaces on either side of it---is that $\eta$ is conformally removable. Thus, if there were another welding along the same boundary arc homeomorphism which produced a different interface, the two weldings would differ by a homeomorphism of $\h$ that was conformal off the image of $\eta$, which  by removability would have to be a conformal automorphism of $\h$. The fact~\cite[Thm~5.2]{rohde}~\cite[Cor.~2]{JS_removability} that an $\mathrm{SLE}_\kappa$ curve with $\kappa \in (0,4)$ is conformally removable follows from the fact that it is the boundary of a \emph{H\"{o}lder domain}, i.e.\ a domain which can be uniformized by a H\"{o}lder-continuous map from the unit disc. Proving conformal removability of a curve involves controlling how much a straight line segment near the curve is distorted by such a homeomorphism, whereas as mentioned above our task is to establish control on the extent to which an LQG geodesic is affected by its crossings of $\eta$. Though the two problems are similar in flavour, in the metric gluing setting we do not have a simple sufficient criterion analogous to the H\"{o}lder domain condition for conformal removability.
	
	We also obtain the appropriate generalizations for the other main theorems in~\cite{metglu}. Our version of~\cite[Thm~1.6]{metglu}, concerning gluing the two boundary arcs of a quantum wedge together to create a quantum cone, is as follows:
	
	\begin{thm}\label{thm:16}
		Fix $\gamma \in (0,2)$ and $\mathfrak{w} \ge 0$. Let $(\mathbb{C},h,0,\infty)$ be a quantum cone of weight $\mathfrak{w}$ and let~$\fd_h$ be the $\gamma$-LQG metric induced by $h$. Let $\eta$ be an independent whole-plane $\mathrm{SLE}_{\gamma^2}(\mathfrak{w}-2)$ from~0 to~$\infty$ and let $U = \mathbb{C}\setminus\eta$. Then $\fd_{h|_U}$ almost surely extends by continuity to $\partial U$ (seen as a set of prime ends), and $(\mathbb{C},\fd_h)$ almost surely agrees with the metric quotient of $(U,\fd_{h|_U})$ under identifying the two sides of $\eta$ in the obvious way (i.e., two prime ends corresponding to the same point in $\mathbb{C}$ are identified).
	\end{thm}
	Here the surface $(U,h|_U)$ is a quantum wedge of weight $\mathfrak{w}$ by~\cite[Thm~1.5]{dms}, and this result tells us that we can recover the original cone from this wedge via metric gluing.
	\begin{figure}\label{fig:thm-1-2}
		\includegraphics[width=0.9\textwidth]{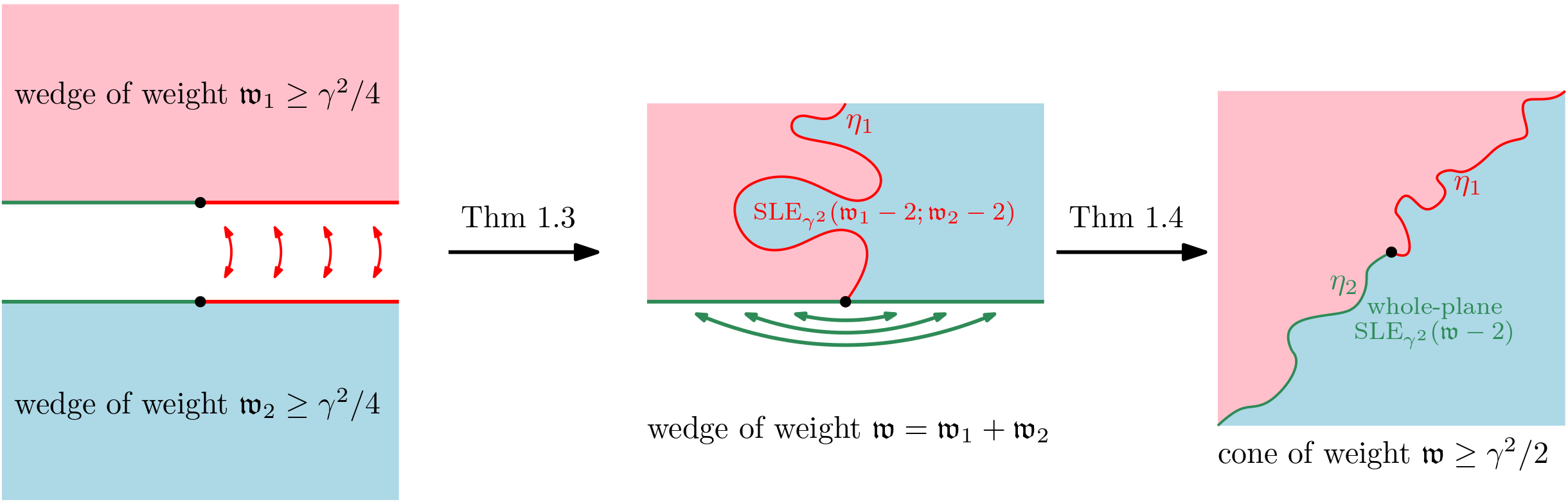}
		\caption{An illustration of Theorems \ref{thm:main} and \ref{thm:16} in the case of two thick wedges ($\mathfrak{w}_1$, $\mathfrak{w}_2 \ge \gamma^2/4$) which are glued along half their boundaries to yield a wedge of weight $\mathfrak{w} = \mathfrak{w}_1 + \mathfrak{w}_2$, then along the other half to yield a cone of weight $\mathfrak{w}$.}
	\end{figure}
	
	We also generalize~\cite[Thm~1.7]{metglu}, which says that a quantum cone cut by a space-filling variant of SLE into countable collection of beads of thin wedges can be recovered by metrically gluing the beads along their boundaries:
	\begin{thm}\label{thm:17}
		Fix $\gamma \in (0,2)$ and let $(\mathbb{C},h,0,\infty)$ be a $\gamma$-quantum cone with associated $\gamma$-LQG metric $\fd_h$. Let $\eta'$ be an independent whole-plane space-filling $\mathrm{SLE}_{16/\gamma^2}$ from $\infty$ to $\infty$ as defined in~\cite[Footnote 4]{dms}, reparametrized by \emph{quantum time} (i.e., so that $\mu_h(\eta([a,b])) = b-a$), with $\eta'(0) = 0$. Let $\mathcal{U}^-$ (resp.\ $\mathcal{U}^+$) be the set of connected components of the interior of $\eta'((-\infty,0])$ (resp.\ $\eta'([0,\infty))$) and for each $U \in \mathcal{U}^-\cup\mathcal{U}^+$ let $\fd_{h|_U}$ be the $\gamma$-LQG metric induced by~$h|_U$. Then almost surely, each~$\fd_{h|_U}$ extends by continuity (w.r.t.\ the Euclidean metric) to $\partial U$, and $(\mathbb{C},\fd_h)$ is the metric quotient of $$\bigsqcup_{U \in \mathcal{U}^-\cup\mathcal{U}^+} (\overline{U},\fd_{h|_U})$$ under the obvious identification.
	\end{thm}
	By~\cite[Thm~1.2, Thm~1.5]{dms}, when $\gamma > \sqrt{2}$ the surfaces $(U,h|_U)$ here are single beads of thin wedges of weight $2-\gamma^2/2$, whereas for $\gamma \le \sqrt{2}$ they are thick wedges. 
	Finally, we generalize~\cite[Thm~1.8]{metglu}, in which we recover a quantum sphere as a quotient of a set of surfaces into which it is cut by a space-filling $\mathrm{SLE}_{16/\gamma^2}$.
	
	\begin{thm}\label{thm:18}
		Fix $\gamma \in (0,2)$ and let $(\mathscr{C},h,-\infty,\infty)$ be a unit area quantum sphere with associated $\gamma$-LQG metric $\fd_h$. Let $\eta'$ be an independent whole-plane space-filling $\mathrm{SLE}_{16/\gamma^2}$ from $+\infty$ to $+\infty$ and reparametrize $\eta'$ by quantum time. Let $T$ be a $U[0,1]$ variable independent of everything else, and let $\mathcal{U}^-$ (resp.\ $\mathcal{U}^+$) be the set of connected components of the interior of $\eta'([0,T])$ (resp.\ $\eta'([T,1])$). For each $U \in \mathcal{U}^-\cup\mathcal{U}^+$ let $\fd_{h|_U}$ be the $\gamma$-LQG metric induced by $h|_U$. Almost surely, each $\fd_{h|_U}$ extends by continuity (w.r.t.\  the Euclidean metric) to $\partial U$, and $(\mathbb{C},\fd_h)$ is the metric quotient of $$\bigsqcup_{U \in \mathcal{U}^-\cup\mathcal{U}^+} (\overline{U},\fd_{h|_U})$$ under the obvious identification.
	\end{thm}
	
	In~\cite{metglu} many of the preliminary results are proved using the results in~\cite{lqg2}. In that paper, the metric $\fd_h$ is constructed in the case $\gamma = \sqrt{8/3}$ (the more general $\gamma \in (0,2)$ result was not established until later). It is then shown that for $\gamma = \sqrt{8/3}$, there almost surely exists an isometry from the quantum sphere to another object, the \emph{Brownian map} introduced by Le Gall~\cite{legall} (whose law intuitively describes that of a metric space chosen ``uniformly at random'' from those spaces with the topology of a sphere), and further that this isometry almost surely pushes forward the LQG measure $\mu_h$ to the natural measure on the Brownian map. Similar isomorphisms of metric measure spaces are established between other quantum and Brownian surfaces. Distances in these surfaces have explicit formulae in terms of Brownian motion-type processes.
	
	Since the equivalence between quantum and Brownian surfaces only holds for $\gamma = \sqrt{8/3}$, the techniques used in~\cite[\S3.2]{metglu} to establish estimates on areas, distances and boundary lengths are not available in this more general setting. We instead obtain analogues of these estimates largely via GFF methods, as well as the conformal welding properties of quantum wedges, which let us transfer our understanding of the interior behaviour of our surfaces to their boundaries (sometimes using existing results about the SLE curves that form the welding interfaces). In fact, in the case $\gamma \neq \sqrt{8/3}$ the existing literature only addresses LQG metrics associated to whole-plane or \emph{zero-boundary} GFFs; this paper provides the first treatment of the metric on surfaces with free boundary conditions for the complete subcritical case $\gamma \in (0,2)$. In particular we establish that the LQG metric given by a free-boundary GFF actually does extend continuously to the boundary:
	\begin{prop}\label{prop:dh}
		Fix $\gamma \in (0,2)$. Let $h$ be a free-boundary GFF on $\h$ with the additive constant fixed so that the semicircle average $h_1(0)$ equals zero, and let $\fd_h$ be the associated $\gamma$-LQG metric on $\h$. Then $\fd_h$ almost surely extends by continuity to a metric on $\overline{\h}$ that induces the Euclidean topology on $\overline{\h}$.
	\end{prop}
	Some of our other results about the $\gamma$-LQG metric on the boundary may be of independent interest. For instance, we establish local bi-H\"{o}lder continuity w.r.t.\ the Euclidean metric:
	\begin{prop}\label{prop:biholder}
		In the setting of Prop.~\ref{prop:dh}, there are exponents $\alpha_1$, $\alpha_2 > 0$ such that, almost surely, for each compact $K \subset \overline{\h}$, there exists $C > 0$ finite such that
		\begin{equation*}
			C^{-1}|z-w|^{\alpha_1} \le \fd_h(z,w) \le C|z-w|^{\alpha_2}
		\end{equation*}
		for each $z,w \in K$.
	\end{prop}
	It should be noted that, although we obtain the right-hand inequality for any $\alpha_2 < \xi(Q-2)$ which is the optimal exponent even away from the boundary~\cite[Thm~1.7]{dfgps}, we make no attempt to obtain the optimal exponent for the left-hand inequality, and we do not expect that the value for $\alpha_1$ resulting from our proof is optimal. During the proof we establish a new regularity estimate for $\mathrm{SLE}_\kappa$ curves with $\kappa \in (0,4)$. Namely, we combine the ``non-self-tracing'' result in~\cite{mmq} for $\mathrm{SLE}_\kappa$ curves with $\kappa \in (0,8)$ with an argument based on conformal covariance of the LQG measure that rules out large bottlenecks to establish that, when $\kappa \in (0,4)$, the (Euclidean) diameter of an $\mathrm{SLE}_\kappa$ segment is at most polynomial in the distance between its endpoints. (Recall that $\kappa \le 4$ is the range for which $\mathrm{SLE}_\kappa$ is simple, though we do not investigate the critical value $\kappa = 4$ here.)
	\begin{prop}\label{prop:bottleneck}
		For each $\kappa \in (0,4)$ there is an exponent $\zeta > 0$ such that the following holds. Let $\eta$ be an $\mathrm{SLE}_\kappa$ in $\h$ from 0 to $\infty$ (with any parametrization). For each compact $K \subset \h$, there almost surely exists $C\in(0,\infty)$ such that
		\begin{equation*}
			\mathrm{diam}\, \eta([s,t]) \le C|\eta(s)-\eta(t)|^\zeta
		\end{equation*}
		whenever $\eta(s)$, $\eta(t) \in K$.
	\end{prop} 
	\subsection{Outline}
	This paper is structured as follows. In \S\ref{section:prelims} we recall the definition and basic properties of the GFF, SLE, quantum wedges and cones, and the Liouville quantum gravity metric. In \S\ref{section:lqg} we show that the LQG metric corresponding to a free-boundary GFF on $\h$ extends continuously to a metric on $\overline{\h}$ that is locally H\"{o}lder continuous w.r.t.\ the Euclidean metric. In \S\ref{section:bh} we prove that the LQG metric on the boundary is locally H\"{o}lder continuous w.r.t.\ the LQG boundary measure, and that the Euclidean metric is locally H\"{o}lder continuous w.r.t.\ the LQG metric. In \S\ref{section:volbound} we use an SLE/GFF coupling to establish a bound on the amount of LQG area within an LQG-metric neighbourhood of a boundary segment. Finally \S\ref{section:pfs} contains the proofs of the main results.
	\subsection*{Acknowledgments}
		L.H.\ was supported by the University of Cambridge Harding Distinguished Postgraduate Scholars Programme. J.M.\ was supported by ERC starting grant 804166 (SPRS). The authors would like to thank an anonymous referee for helpful comments on a previous version of this article. 
	\section{Preliminaries}\label{section:prelims}
	\subsection{The Gaussian free field}
	The \textbf{Gaussian free field} (GFF) is a random process analogous to Brownian motion, where the analogue of the time parameter ranges over a domain in the complex plane. We recall the definition of the \emph{zero-boundary} GFF from~\cite[Def.~2.10]{gffm}, which begins with an open set $D \subset \mathbb{C}$ with \emph{harmonically non-trivial} boundary (meaning that a Brownian motion started from $z\in D$ will almost surely hit $\partial D$). We let $H_s(D)$ be the set of smooth functions with compact support contained in $D$, equipped with the \emph{Dirichlet inner product}
	\begin{equation*}
		(f,g)_\nabla = \frac{1}{2\pi} \int_D \nabla f(x) \cdot \nabla g(x) \, dx,
	\end{equation*}
	and complete this inner product space to a Hilbert space $H(D)$.
	Taking an arbitrary orthonormal basis $(\varphi_n)$ of $H(D)$ and letting $(\alpha_n)$ be i.i.d.\ $N(0,1)$ variables, the zero-boundary GFF in $D$ is then defined as a random linear combination of elements of $H(D)$ given by
	\begin{equation}\label{eq:gffdef}
		h = \sum_n \alpha_n \varphi_n.
	\end{equation}
	It can be shown (see~\cite[Prop.~2.7]{gffm}) that this sum converges almost surely in the space of distributions and in the fractional Sobolev space $H^{-\ve}(D)$ for each $\ve > 0$ (even though it does \emph{not} converge pointwise or in $H(D)$ itself) and that the law of the limit $h$ does not depend on the choice of basis $(\varphi_n)$. This limiting distribution $h$ is the \textbf{zero-boundary Gaussian free field}. Writing $(\cdot,\cdot)$ for the usual $L^2$ inner product, we can define for $f \in H_s(D)$
	\begin{equation*}
		(h,f) := \lim_{n\to\infty} \left(\sum_n \alpha_n \varphi_n, f \right).
	\end{equation*}
	Note that, for each $f \in H_s(D)$, this sum converges almost surely as an $L^2$-bounded martingale. Indeed, the limit almost surely exists for all $f\in H_s(D)$ simultaneously, and is such that $f \mapsto (h,f)$ is a continuous functional on $H_s(D)$.
	
	Moreover, one can define the $L^2$ pairing of $h$ with certain other measures. Most importantly for us, if $h$ is a zero-boundary GFF, $\ve > 0$ and $B(z,\ve) \subset D$, we denote by $h_\ve(z)$ the \textbf{circle average} of $h$ on the circle $\partial B(z,\ve)$, defined as $(h,\rho_{z,\ve}) = -2\pi (h,\Delta^{-1} \rho_{z,\ve})_\nabla$ where $\rho_{z,\ve}$ is the uniform probability measure on $\partial B(z,\ve)$. In~\cite[Prop.~2.1]{hmp}, it is shown that for each fixed $z\in D$, the process $\{h_{e^{-t}}(z) \colon B(z,e^{-t})\subset D \}$ has the covariance structure of a standard Brownian motion on the interval $\{t\colon B(z,e^{-t}) \subset D \}$, and that the circle average process $\{h_{e^{-t}}(z) \colon B(z,e^{-t}) \subset D \}$ has a version that is continuous in both $t$ and $z$.
	
	Given a function $g$ on $\partial D$ such that there exists a unique function $\mathfrak{h}$ on $\overline{D}$ which is continuous at all but finitely many points of $\overline{D}$, equals $g$ on $\partial D$ and is harmonic in $D$, we define the law of a GFF in $D$ with boundary data $g$ to be the law of $h + \mathfrak{h}$ where $h$ is a zero-boundary GFF in $D$.
	
	We can instead set $D$ to be all of $\mathbb{C}$. In this case, as in~\cite[\S2.2.1]{ig4}, we define the \textbf{whole-plane Gaussian free field} $h$ in the same way, except that we consider $h$ \emph{modulo additive constant}. This means that we consider the equivalence relation~$\sim$ on the space of distributions defined by the condition that $h_1 \sim h_2$ if and only if~$h_1-h_2$ is a constant distribution, i.e.\ if and only if there exists $a\in\R$ such that $(h_1,f)-(h_2,f) = a\int_{\mathbb{C}} f(z)\, dz$ for all $f\in H_s(\mathbb{C})$. We take $(\varphi_n)$ to be a fixed orthonormal basis for $H(\mathbb{C})$ and sample i.i.d.\ $N(0,1)$ variables $\alpha_n$, and define $h$ as the equivalence class of~$\sim$ containing $\sum_n \alpha_n \varphi_n$. Equivalently, for~$f \in H_s(\mathbb{C})$, we only consider $(h,f)$ to be defined if $f \in H_{s,0}$, the subspace of those functions in~$H_s(\mathbb{C})$ whose integral over $\mathbb{C}$ is zero. Observe that the circle average process $(h_{e^{-t}}(z)-h_1(z))_{t\in\R}$ is well-defined, since $h_{e^{-t}}(z)-h_1(z)=(h,\rho_{z,e^{-t}}-\rho_{z,1})$ and $\int_{\mathbb{C}} d(\rho_{z,e^{-t}}-\rho_{z,1}) = 0$. It turns out that the process $(h_{e^{-t}}(z)-h_1(z))_{t\in\R}$ has a version which is a standard two-sided Brownian motion starting from 0.
	
	We can also \emph{fix} the additive constant, i.e.\ choose a representative of the equivalence class under~$\sim$. For example, we can stipulate that $h_1(0)=0$, obtaining a random distribution \emph{not} modulo additive constant. Note that we consider two \emph{random} distributions (on the same probability space) to be the same modulo additive constant if their difference is almost surely a constant distribution, i.e.\ constant in the spatial variable $z$; this constant need not be deterministic. Thus, if we have two ways of fixing the additive constant of a whole-plane GFF---say, the normalizations $h_1(0)=0$ and $h_e(0)=0$---their difference need not be a deterministic constant (indeed, in this case it is a standard Gaussian). We say that a random distribution $\widehat{h}$ on $\mathbb{C}$ (\emph{not} modulo additive constant) is a \textbf{whole-plane Gaussian free field plus a continuous function} if there exists a coupling of $\widehat{h}$ with a whole-plane GFF $h$ (with the additive constant fixed in some way) such that $\widehat{h}-h$ is almost surely a continuous function; note that this definition does not depend on how the additive constant for $h$ is fixed.
	
	We will also need the notion of the \textbf{free-boundary Gaussian free field} on a domain $D$ with harmonically non-trivial boundary, as defined in~\cite[\S3.2]{zip}. The free-boundary GFF is defined in the same way as the zero-boundary GFF but with $H(D)$ replaced by the Hilbert space closure $H^F(D)$ of the space of smooth functions whose gradients are in $L^2(D)$, considered modulo additive constant (these functions need not be compactly supported). Note that we have to consider functions only modulo additive constant in order for the Dirichlet inner product to be positive definite on this space. Note also that, since it is constructed as a limit (in a Sobolev space or space of distributions) of functions modulo additive constant, the free-boundary GFF is a distribution modulo additive constant.
	
	We note for later reference some key properties of the GFF. Firstly, it is straightforward to check that the Dirichlet inner product is conformally invariant in two dimensions, from which it follows that the GFF is also conformally invariant. In particular the whole-plane GFF, and the free-boundary GFF on $\h$, are invariant under scalings and translations (when considered modulo additive constant). Secondly, one has the \emph{domain Markov property}~\cite[\S2.6]{gffm}; for a zero-boundary GFF in $D$, this states that if $U \subseteq D$ is open, then we can write $h = h_1 + h_2$ where $h_1$ is a zero-boundary GFF on $U$ and~$h_2$ is a random harmonic function independent of $h_1$. This holds because $H(D)$ is the orthogonal direct sum of the space $H(U)$ and the subspace $H_\mathrm{harm}(U)$ of $H(D)$ given by functions that are harmonic in~$U$, so that one can define $h_1$ and $h_2$ as the orthogonal projections of $h$ onto, respectively, $H(U)$ and $H_\mathrm{harm}(U)$. Independence of $h_1$ and $h_2$ follows by taking the basis for $H(D)$ in~(\ref{eq:gffdef}) to be a union of bases for $H(U)$ and $H_\mathrm{harm}(U)$. Note that the domain Markov property also holds if $h$ is instead a whole-plane GFF, or if $h$ is a free-boundary GFF on $\h$ and $U=\h$. In these cases $h_2$ will only be defined modulo additive constant---this will be discussed further at the beginning of~\S\ref{section:lqg}. 
	
	If $h$ is a free-boundary GFF on $\h$, $x \in \partial \h$ and $\ve > 0$, we denote by $h_\ve(x)$ the \textbf{semicircle average} of $h$ on the semicircular arc $\partial B(x,\ve)\cap \h$, defined as $(h,\rho^+_{x,\ve}) = -2\pi (h,\Delta^{-1} \rho^+_{x,\ve})_\nabla$ where $\rho^+_{x,\ve}$ is the uniform probability measure on $\partial B(x,\ve) \cap \h$. Note that we are using the same notation for semicircle and circle averages, since for fields defined on $\h$ we will usually consider semicircle averages and for fields defined on $\mathbb{C}$ we will usually consider circle averages. If $h$ is defined on $D$, we can simply define $h_\ve(z)$ to mean the average of $h$ on $\partial B(z,\ve) \cap D$, which covers both these cases. However, for a field $h$ defined on $\h$, we will make it clear when we are considering circle averages as opposed to semicircle averages by writing $h_\ve^\mathrm{circ}(z)$ for the average of $h$ over the circle $\partial B(z,\ve)$ (when $B(z,\ve) \subset \overline{\h}$).
	
	By~\cite[Lemma~4.9]
	{dms}, another orthogonal decomposition of $H(\mathbb{C})$ is given by the \textbf{radial--lateral decomposition} into the space $H_\mathrm{rad}(\mathbb{C})$ of radially symmetric functions and the space $H_\mathrm{lat}(\mathbb{C})$ of functions with mean zero on all circles with centre at 0. We define the \emph{radial part} $h_\mathrm{rad}^\mathrm{wp}$ of a whole-plane GFF $h^\mathrm{wp}$, given by the projection of $h^\mathrm{wp}$ onto $H_\mathrm{rad}(\mathbb{C})$, as the function $h^\mathrm{wp}_{|\cdot|}(0)$ whose value on each circle centred at 0 is simply given the average of $h$ on that circle (and is only defined modulo additive constant). We also define the \emph{lateral part} $h_\mathrm{lat}^\mathrm{wp}$ of $h^\mathrm{wp}$ as the projection of $h^\mathrm{wp}$ onto $H_\mathrm{lat}(\mathbb{C})$, which is given by $h^\mathrm{wp}-h^\mathrm{wp}_{|\cdot|}(0)$ and is well-defined not just modulo additive constant. Then the radial--lateral decomposition implies that $h_\mathrm{rad}^\mathrm{wp}$ and $h_\mathrm{lat}^\mathrm{wp}$ are independent.
	
	One also has a radial--lateral decomposition for the free-boundary GFF on $\h$. Indeed, $H^F(\h)$ is the orthogonal sum of the space $H^F_\mathrm{rad}(\h)$ of functions that are radially symmetric about 0 and $H^F_\mathrm{lat}(\h)$ of functions that have the same average on all semicircles centred at 0 (recall that elements of $H^F(\h)$ are only defined modulo additive constant). Note that the radial part, i.e.\ the projection~$h^\mathrm{rad}$ of $h$ onto $H^F_\mathrm{rad}(\h)$, whose values are given by the semicircle average process centred at 0, is only defined modulo additive constant, but we can consider the lateral part $h^\mathrm{lat}=h-h^\mathrm{rad}$ as a function not just modulo additive constant, whose average is zero on every semicircle centred at~0. Again $h^\mathrm{rad}$ and $h^\mathrm{lat}$ are independent.
	
	Finally one can consider the radial--lateral decomposition for the free-boundary GFF $\widetilde{h}$ on the bi-infinite strip $\mathscr{S} = \mathbb{R}\times [0,\pi]$, which by conformal invariance can be obtained as $\widetilde{h}(\cdot) = h(\exp(\cdot))$ for $h$ a free-boundary GFF on $\h$. In this case the orthogonal decomposition of $H^F(\mathscr{S})$ is given~\cite[Lemma~4.3]{dms} by the space $H^F_\mathrm{rad}(\mathscr{S})$ of functions that are constant on the vertical line $u+[0,i\pi]$ for each $u\in\R$ and the space $H^F_\mathrm{lat}(\mathscr{S})$ of functions that have the same average on all such vertical lines. A similar decomposition holds for the bi-infinite cylinder $\mathscr{C}$ given by $\R\times[0,2\pi]$ with $\R\times\{0\}$ and $\R\times\{2\pi\}$ identified.
	
	We will show in Lemma~\ref{lem:cty} that, for $x$ fixed, the process $(h_{e^{-t}}(x)-h_1(x))_{t\in\R}$ has the covariance structure of $\sqrt{2}$ times a standard two-sided Brownian motion, and that the semicircle average process $(h_{e^{-t}}(x)-h_1(x))_{t,x\in\R}$ has a version that is continuous (in both $t$ and $x$). This is a straightforward adaptation of~\cite[Prop.~2.1]{hmp}, the analogous result for \emph{circle} averages of a zero-boundary GFF.
	\subsection{Quantum wedges and cones}\label{subsection:wedges}
	First we explain the basic definitions of the quantum surfaces described in the introduction. Given a domain $D\subseteq\mathbb{C}$ and~$h$ some form of the GFF on $D$ (with the additive constant fixed in some way if necessary), and $\gamma \in (0,2)$, we define the random area measure $\mu_h$ on $D$ as a Gaussian multiplicative chaos measure in the sense of \cite{kahane}, given by the weak limit of the regularized measures
	\begin{equation*}
		\mu_h^\ve := \ve^{\gamma^2/2} e^{\gamma h_\ve(z)}\, dz
	\end{equation*}
	as $\ve\to 0$ along powers of two, where $dz$ is Lebesgue measure on $D$ and $h_\ve(z)$ is the average of $h$ on the circle of radius~$\ve$ centred at $z$ (or on the intersection of this circle with $D$ when $z\in \partial D$). This limit was shown to exist almost surely in~\cite{ds}. Likewise,~\cite{ds} shows the almost sure existence of the corresponding weak limit $\nu_h$ of the measures
	\begin{equation*}
		\nu_h^\ve := \ve^{\gamma^2/4} e^{\gamma h_\ve(x)/2}\, dx
	\end{equation*}
	where $dx$ is Lebesgue measure on a linear segment of $\partial D$.. The regularization procedure implies the conformal coordinate change rule (\ref{eq:reparam}) given in the introduction, under which, by~\cite[Prop.~2.1]{ds}, almost surely we have $\mu_h = \mu_{\widetilde{h}} \circ \psi$ and $\nu_h = \nu_{\widetilde{h}} \circ \psi$. In particular, we can use this to define $\nu_h$ when $\partial D$ is not piecewise linear by conformally mapping to, for example, the upper half-plane (provided the conformal map extends to a homeomorphism $\partial D \to \mathbb{R} \cup \{\infty\}$).
	
	We define a \textbf{quantum surface} as an equivalence class of objects of the form $(D,h)$ where $D$ is a planar domain and $h$ is a random distribution on $D$, where $(\widetilde{D},\widetilde{h})$ and $(D,h)$ are considered equivalent if and only if there exists a conformal map $\psi: \widetilde{D} \to D$ such that $\widetilde{h}$ and $h$ satisfy the rule (\ref{eq:reparam}). Often one also wants to keep track of certain \emph{marked points}; to this end we define a \textbf{quantum surface with $k$ marked points} as an equivalence class of objects of the form $(D,h,z_1,\ldots,z_k)$ where $z_i \in \overline{D}$, so that two quantum surfaces $(D,h,z_1,\ldots,z_k)$ and $(\widetilde{D},\widetilde{h},\widetilde{z}_1,\ldots,\widetilde{z}_k)$ such that the conformal map $\psi: \widetilde{D} \to D$ satisfies the rule (\ref{eq:reparam}) are only considered equivalent as surfaces with~$k$ marked points if in addition we have $\psi(\widetilde{z}_i) = z_i$ for $i = 1,\ldots, k$.
	
	We will now define the notion of ``quantum wedge''; the idea is that we would like to define a quantum surface homeomorphic to $\h$, whose law is invariant under scaling and under the operation of adding a constant to the field, and thus a good candidate for infinite-volume scaling limits. As a warm-up we will define an ``unscaled quantum wedge'', for which the field is only defined modulo additive constant, but keep in mind that the ordinary quantum wedge does not arise by fixing this constant, since such a surface would not have the desired invariance properties.
	
	An \textbf{unscaled $\alpha$-quantum wedge} is given by $(\h,h^F-\alpha \log{|\cdot|},0,\infty)$ where $h^F$ is an instance of the free-boundary GFF on $\h$. (Note that this~$h^F$ is only defined modulo additive constant, meaning that $\mu_h$ and $\nu_h$ are only defined modulo multiplicative constant and thus the unscaled wedge is \emph{not} a quantum surface by our definition above.) The definition arises (as does the nomenclature) by considering a free-boundary GFF on an infinite wedge $W_\vartheta=\{z\in\mathbb{C}:\arg{z}\in[0,\vartheta]\}$ (viewed as a Riemann surface, so that the parametrization is not single-valued if $\vartheta \ge 2\pi$), and then using (\ref{eq:reparam}) to reparametrize by $\h$ via the conformal map $z \mapsto z^{\pi/\vartheta}$, where~$\vartheta=\pi(1-\alpha/Q)$.
	
	We can reparametrize by the infinite strip $\mathscr{S} = \R\times[0,\pi]$ instead of by $\h$. If we use an appropriate branch of $\log$ to map~$\h$ to $\mathscr{S}$, so that 0 maps to $-\infty$ whilst $\infty$ maps to $+\infty$, then the conformal coordinate change formula (\ref{eq:reparam}) gives that the mean of the resulting field $\widetilde{h}$ on the vertical segment $\{t\} \times [0,\pi]$ is given by $B_{2t}+(Q-\alpha)t$, where $B$ is a standard two-sided Brownian motion, defined \emph{modulo additive constant}. We next define an ordinary \textbf{quantum wedge}~\cite[Def.~4.5]{dms} by replacing the process $B_{2t}+(Q-\alpha)t$ by a related but different process, in such a way that we fix the additive constant and thus obtain a genuine quantum surface, whose \emph{law} will nonetheless be invariant under the operation of adding a constant to the field. Namely, define an $\alpha$-quantum wedge by $(\mathscr{S},\widehat{h},-\infty,+\infty)$ where $\widehat{h}$ is obtained from $\widetilde{h}$ by replacing the process $B_{2t}+(Q-\alpha)t$ by $(A_t)_{t\in \R}$, where for $t\le0$ we define $A_t = B_{-2t}+(Q-\alpha)t$ for $B$ a standard Brownian motion started from 0, and for $t > 0$ we define $A_t = \widehat{B}_{2t} + (Q-\alpha) t$ where $\widehat{B}$ is a standard Brownian motion started from zero, independent of $B$ and conditioned on the event that $\widehat{B}_{2t} + (Q-\alpha) t > 0$ for all $t > 0$. 
	
	This is called the \textbf{circle average embedding} since it has the property that, when we use $z \mapsto \exp(z)$ to map from $\mathscr{S}$ back to $\h$ to produce a different parametrization of the surface, namely $(\h,h,0,\infty)$ where $h=\widehat{h}\circ\log - Q\log|\cdot|$, we have $0 = \sup\{ t\in\R: h_{e^{t}}(0)+Qt=0 \}$, where $h_r(z)$ is the \textbf{semicircle average} on $\partial B(z,r)$, i.e.~the average of $h$ on $\partial B(z,r)\cap\h$. One can next construct the circle average embedding of $h+C$ where $C$ is a constant by spatially rescaling by~$e^{t^C}$, where $t^C = \sup\{ t\in\R: h_{e^{t}}(0)+Qt+C=0 \}$---note that by (\ref{eq:reparam}) this corresponds to replacing the field $h+C$ by $h(e^{t^C}\cdot) +Qt^C + C$. From the properties of Brownian motion with drift, one can then check~\cite[Prop.~4.7(i)]{dms} that a quantum wedge has the key property that its law as a quantum surface is invariant under the operation of adding a constant to the field, meaning that the circle average embeddings of~$h$ and~$h+C$ have the same law for a constant $C>0$. One can also observe the convenient property that if $(\h,h,0,\infty)$ is the circle average embedding of an $\alpha$-quantum wedge, then the restriction of $h$ to $\h \cap \mathbb{D}$ (where~$\mathbb{D}$ is the unit disc) has the same law as the restriction of $h^{F,0}-\alpha\log|\cdot|$ to $\h \cap \mathbb{D}$, where $h^{F,0}$ is a free-boundary GFF on $\h$ with the additive constant fixed so that the semicircle average $h^{F,0}_1(0)$ is~0.
	
	Since the conditioning event has probability zero, some care is needed to define the process~$\widehat{B}$; the details, given in~\cite[Remark 4.4]{dms}, are as follows. The process can be constructed by setting $\widehat{B}_{2t}+(Q-\alpha)t = \widetilde{B}_{2(t+\tau)}+(Q-\alpha)(t+\tau)$ for all $t\ge 0$, where $\widetilde{B}$ is a standard Brownian motion started from 0 and $\tau$ is the last time that $\widetilde{B}_{2t}+(Q-\alpha)t$ hits 0. Note that $\tau < \infty$ almost surely since $Q > \alpha$. Then $\widehat{B}$ is characterized by the property that, for each $\ve > 0$, if $\tau_\ve$ is the hitting time of~$\ve$ by~$\widehat{B}$ then $(\widehat{B}_{2(t+\tau_\ve)}+(Q-\alpha)(t+\tau_\ve))_t$ has the law of a Brownian motion with drift $Q-\alpha$ started from $\ve$ and conditioned not to hit 0, which makes the law of $\widehat{B}$ the only sensible choice for the required conditional law. The reason this property characterizes the law of $\widehat{B}$ is that if $X$ is another process with the same property and, for each $\ve>0$, the hitting time of $\ve$ by $X$ is~$\widetilde{\tau}_\ve$, then for each $\ve>0$ there is a coupling of $\widehat{B}$ and $X$ so that $(\widehat{B}_{2(t+\tau_\ve)})_t = (X_{2(t+\widetilde{\tau}_\ve)})_t$, whereas $\tau_\ve, \widetilde{\tau}_\ve \to 0$ as $\ve \to 0$ almost surely, so that in any subsequential limit of such couplings as $\ve \to 0$ we have $\widehat{B} = X$ almost surely. In fact, one can also define $A_t$ (see~\cite[\S1.1.2]{dms}) as the log of a Bessel process of dimension $2+2(Q-\alpha)/\gamma$, parametrized by quadratic variation; this definition also makes sense for $\alpha = Q$. A surface constructed as above (with $\alpha \le Q$) is an \textbf{$\alpha$-quantum wedge}; we refer to such wedges, which are homeomorphic to $\h$, as being \textbf{thick}.
	
	The Bessel process construction generalizes further, to the case $\alpha \in (Q,Q+\gamma/2)$. In this case the Bessel process has dimension between 1 and 2 so will hit 0; we obtain one surface for each excursion of the Bessel process away from 0, and thus by concatenating all these surfaces (see~\cite[\S1.1.2]{dms}) we get an infinite chain called a \textbf{thin quantum wedge} (in this case, the horizontal translation is fixed by requiring the process to attain a maximum at $t=0$). More formally, we use the fact~\cite[Ch.~XI,~XII]{ry} that the excursions of a Bessel process $X$ form a Poisson point process when indexed by local time at 0. Specifically, for each excursion $\mathbf{e}$ of $X$ (say, over the time interval $(a_{\mathbf{e}},b_{\mathbf{e}})$), if $s_{\mathbf{e}}$ is the local time at 0 accumulated by $X|_{[0,a]}$, then the $(s_{\mathbf{e}},\mathbf{e})$ form a Poisson point process with mean measure $ds \otimes N$ where $ds$ is Lebesgue measure on $[0,\infty)$ and $N$ is an infinite measure, the so-called \textbf{It\^{o} excursion measure} corresponding to $X$ (on the space of excursions translated back in time so as to start at time 0). It is known that this process determines $X$. We thus define an $\alpha$-quantum wedge for $\alpha \in (Q,Q+\gamma/2)$ as a point process where the points are of the form $(s_{\mathbf{e}},\mathbf{e}, h_{\mathbf{e}})$ where each $ h_{\mathbf{e}}$ is a quantum surface defined on the strip $\mathscr{S}$ as for a thick quantum wedge but using $\mathbf{e}$ parametrized by quadratic variation (and, for concreteness, with the parametrization chosen so that the maximum is attained at time 0) in place of $A_t$, and where the lateral parts of the $ h_{\mathbf{e}}$ for different excursions $\mathbf{e}$ are independent. Each doubly marked surface $(\mathscr{S},h_{\mathbf{e}},-\infty,+\infty)$ is a \textbf{bead} of the wedge, with the two marked points referred to as the \emph{opening point} ($-\infty$) and the \emph{closing point} ($+\infty$). 
	
	Since such \emph{beaded quantum surfaces} are no longer parametrized by domains in $\mathbb{C}$, we need to slightly amend the notion of equivalence for such surfaces: a beaded quantum surface is parametrized by a closed set $D$ such that each component of the interior of $D$ together with its prime-end boundary is homeomorphic to a closed disc, and we regard two surfaces parametrized by such sets $\widetilde{D}$, $D$ as equivalent if they are related by the formula (\ref{eq:reparam}) for $\psi\colon \widetilde{D}\to D$ a homeomorphism that is conformal on each component of the interior of $\widetilde{D}$.
	
	As noted in the introduction, we will often refer to an $\alpha$-quantum wedge, in either the thick or thin regimes, as a \textbf{quantum wedge of weight $\mathfrak{w}$} where the weight parameter $\mathfrak{w}>0$ is defined as$$\mathfrak{w}=\gamma\left(\frac{\gamma}{2}+Q-\alpha\right).$$Note that the wedge is thick when $\mathfrak{w}\ge\gamma^2/2$ and thin otherwise. The reason for using the weight parameter is that it is additive under the operation of conformally welding two independent wedges according to LQG boundary length to obtain another wedge. Specifically, \cite[Thm~1.2]{dms} states that if $\mathfrak{w}_1, \mathfrak{w}_2>0$ and $\mathfrak{w}=\mathfrak{w}_1+\mathfrak{w}_2$, when a wedge $\mathcal{W}$ of weight~$\mathfrak{w}$ is decorated by $\eta$, an independent $\mathrm{SLE}_{\gamma^2}(\mathfrak{w}_1-2;\mathfrak{w}_2-2)$ from 0 to $\infty$ (or if $\mathfrak{w} < \gamma^2/2$, a concatenation of independent $\mathrm{SLE}_{\gamma^2}(\mathfrak{w}_1-2;\mathfrak{w}_2-2)$ curves from the opening point to the closing point of each bead), then the region~$\mathcal{W}_1$ (resp.\ $\mathcal{W}_2$) to the left (resp.\ right) of $\eta$ is a wedge of weight $\mathfrak{w}_1$ (resp.\ $\mathfrak{w}_2$) and that $\mathcal{W}_1$ and $\mathcal{W}_2$ are independent as quantum surfaces. Moreover, \cite[Thm~1.4]{dms} states that there exists a unique conformal welding of the right-hand side of~$\mathcal{W}_1$ to the left-hand side of $\mathcal{W}_2$ according to $\gamma$-LQG boundary length, which recovers $\mathcal{W}$ and $\eta$.
	
	These results from~\cite{dms} build on the earlier result~\cite[Thm~1.8]{zip} that in the case of a wedge $(\h,h,0,\infty)$ of weight 4 cut by an $\SLE_\kappa$ $\eta$ into two wedges of weight 2, for each $t>0$ the law of the pair of surfaces to the left and right of $\eta$ is invariant under both the operation $\mathcal{Z}_{-t}$ of cutting only along $\eta([0,t])$ (where $\eta$ is parametrized by LQG boundary length) and the operation~$\mathcal{Z}_t$ of conformally welding the boundary segments $(x_t^-,0]$ and $[0,x_t^+)$ according to LQG boundary length, where $x_t^\pm$ are defined so that $\nu_h((x_t^-,0]) = \nu_h([0,x_t^+)) = t$ (in particular, \cite[Thm~1.8]{zip} states that this welding is almost surely unique). The group of transformations $\{\mathcal{Z}_t\colon t\in\R\}$ is called the \textbf{(length) quantum zipper}: for $t>0$, $\mathcal{Z}_t$ ``zips up'' the pair of surfaces by $t$ units of LQG boundary length whilst $\mathcal{Z}_{-t}$ ``unzips'' by $t$ units of LQG boundary length.
	
	We next define the whole-plane analogue of the quantum wedge. Intuitively, an \textbf{$\alpha$-quantum cone} is the doubly marked quantum surface corresponding to a GFF on the surface homeomorphic to $\mathbb{C}$ obtained by gluing together the sides of $W_\theta$, where $\theta = 2\pi(1-\alpha/Q)$, according to Lebesgue measure. It is given by $(\mathbb{C},h,0,\infty)$, where the field $h$ is defined in~\cite[Definition 4.10]{dms} for $\alpha < Q$ by taking the process~$A_t$ as for an $\alpha$-quantum wedge, except with $B_{2t}$ and $\widehat{B}_{2t}$ replaced by~$B_t$ and $\widehat{B}_t$ respectively, and then setting $h$ to be the field on $\mathbb{C}$ whose radial part is given by $A_t$ on the circle of radius~$e^{-t}$ around 0, and whose lateral part is that of an independent whole-plane GFF. Note that the radial part is only defined modulo additive constant, but we generally fix the constant as we do for a wedge, i.e.\ by requiring $A_0 = 0$. As before, the law of a quantum cone is invariant under the operation of adding a constant to the field (i.e., the circle average of the resulting cone will have the same law as that of the original one); analogously to the case with wedges, the restriction of the circle average embedding of an $\alpha$-quantum cone $h$ to the unit disc $\mathbb{D}$ is equal in law to the restriction of $h^{\mathrm{wp}}-\alpha\log|\cdot|$ to $\mathbb{D}$ where~$h^{\mathrm{wp}}$ is a whole-plane GFF with the additive constant chosen so that the circle average $h^{\mathrm{wp}}_1(0)$ is 0. 
	
	Again, instead of using the parameter $\alpha < Q$, we will often refer to an $\alpha$-quantum cone as a \textbf{quantum cone of weight~$\mathfrak{w}$} where this time the weight parameter $\mathfrak{w}>0$ is given by$$\mathfrak{w}=2\gamma(Q-\alpha).$$This choice is convenient because cones of weight $\mathfrak{w}$ are the whole-plane analogues of wedges of weight $\mathfrak{w}$: \cite[Thm~1.5]{dms} says that if a cone $\mathcal{C} = (\mathbb{C},h,0,\infty)$ of weight $\mathfrak{w}$ is decorated by $\eta$, an independent whole-plane $\mathrm{SLE}_{\gamma^2}(\mathfrak{w}-2)$ from 0, then the surface $\mathcal{W}$ described by $(\mathbb{C}\setminus\eta,h,0,\infty)$ is a wedge of weight $\mathfrak{w}$, and there is a unique conformal welding of left-hand and right-hand boundary segments of $\mathcal{W}$ according to $\gamma$-LQG boundary length, which recovers $\mathcal{C}$ and $\eta$.
	
	In order to construct a probability measure on finite-volume surfaces, we can first consider the ``law'' on finite-volume surfaces corresponding in the above constructions to that of a single Bessel excursion. (Note that the ``law'' of a Bessel excursion is an infinite measure, so the words ``law'' and ``sample'' do not have their literal meanings in this setting.) To define a \textbf{quantum sphere of weight $\mathfrak{w}$} we first define an infinite measure $\mathcal{N}_\mathfrak{w}$ on fields parametrized by the bi-infinite cylinder~$\mathscr{C}$ with marked points at $-\infty$ and $+\infty$ as follows. (Recall that $\mathscr{C}$ is given by $\R\times[0,2\pi]$ with $\R\times\{0\}$ and $\R\times\{2\pi\}$ identified.) A ``sample'' $h$ from $\mathcal{N}_\mathfrak{w}$ can be obtained by ``sampling'' a Bessel excursion $Z$ of dimension $2-2\mathfrak{w}/\gamma^2$, setting the radial part of $h$ (i.e., the projection of $h$ onto the space of functions that are constant on vertical lines $\{r\} \times [0,2\pi]$) to be given by $\frac{2}{\gamma}\log{Z}$ parametrized by quadratic variation, and setting the lateral part of $h$ (i.e., the projection of $h$ onto the space of functions that have the same mean on all vertical lines $\{r\} \times [0,2\pi]$) to be given by the corresponding projection of a GFF on $\mathscr{C}$ (so that the lateral part has mean zero on all vertical lines). One can show that, for $0 < a < b < \infty$, $\mathcal{N}_\mathfrak{w}$ assigns finite mass to the event $\mu_h(\mathscr{C}) \in [a,b]$, allowing us to construct the probability measure $\mathcal{N}_\mathfrak{w}(\cdot|\mu_h(\mathscr{C}) = r)$ as a regular conditional probability for almost every $r > 0$. The scaling properties of $\mathcal{N}_\mathfrak{w}$ mean that in fact this measure must exist for every $r > 0$, and we can thus define the law of a \textbf{unit area quantum sphere} as the law of a quantum sphere of weight $4-\gamma^2$ conditioned to have unit area, i.e.~as the probability measure $\mathcal{N}_{4-\gamma^2}(\cdot|\mu_h(\mathscr{C}) = 1)$. (The weight $4-\gamma^2$, corresponding to a $\gamma$-quantum cone, is special because in this case the marked points at 0 and $\infty$ ``look like'' \emph{quantum typical} points, i.e.\ ones sampled according to the measure $\mu_h$---see~\cite[Lemma~A.10]{dms}.) This argument for the existence of the conditional law appears in the discussion after~\cite[Definition~4.21]{dms}.
	
	\subsection{The subcritical Liouville quantum gravity metric}
	In~\cite[Thm~1.2]{gm} it is proven that for $\gamma \in (0,2)$ there exists a measurable map $h \mapsto \fd_h$, from the space of distributions on $\mathbb{C}$ with its usual topology to the space of metrics on $\mathbb{C}$ that induce the Euclidean topology, that is characterized by satisfying the following axioms whenever $h$ is a whole-plane GFF plus a continuous function:
	\begin{description}
		\item[Length space] Almost surely, the $\fd_h$-distance between any two points of $\mathbb{C}$ is the infimum of the $\fd_h$-lengths of continuous paths between the two points.
		\item[Locality] If $U \subseteq \mathbb{C}$ is deterministic and open, then the \textbf{internal metric $\fd_h(\cdot,\cdot;U)$ of~$\fd_h$ on~$U$}, defined between two points of $U$ by taking the infimum of the $\fd_h$-lengths of continuous paths between the two points that stay in $U$, is almost surely determined by $h|_U$.
		\item[Weyl scaling] Let $\xi = \gamma/d_\gamma$ where $d_\gamma$ is the dimension defined in~\cite{dg}. For $f\colon\mathbb{C}\to\R$ continuous and $z,w\in\mathbb{C}$, define
		\begin{equation*}
			(e^{\xi f}\cdot \fd_h)(z,w) = \inf_P \int_0^{\mathrm{length}(P;\fd_h)} e^{\xi f(P(t))} \, dt
		\end{equation*}
		where $P$ ranges over all continuous paths from $z$ to $w$ parametrized at unit $\fd_h$-speed. Then almost surely $e^{\xi f}\cdot \fd_h = \fd_{h+f}$ for all continuous $f$.
		\item[Affine coordinate change] For each fixed deterministic $r>0$ and $z\in\mathbb{C}$ we almost surely have, for all $u,v\in\mathbb{C}$,
		\begin{equation*}
			\fd_h(ru+z,rv+z) = \fd_{h(r\cdot+z)+Q\log{r}} (u,v).
		\end{equation*}
	\end{description}
	This map is unique in the sense that for any two such objects $\fd$, $\widetilde{\fd}$, there is a deterministic constant~$C$ such that whenever~$h$ is a whole-plane GFF plus a continuous function, almost surely we have $\fd_h = C\widetilde{\fd}_h$. We refer to this unique (modulo multiplicative constant) object as the \textbf{$\gamma$-LQG metric}. Existence is proven by constructing the metric as a subsequential limit of the \textbf{$\ve$-Liouville first passage percolation metric} defined by
	\begin{equation*}
		\fd_h^\ve (z,w) = \inf_P \int_0^1 e^{\xi (h\ast p_{\ve^2/2}) (P(t))} |P'(t)| \, dt
	\end{equation*}
	where the infimum is over all piecewise $C^1$ paths from $z$ to $w$, and $p_{\ve^2/2}$ is the heat kernel with variance $\ve^2/2$ (so we are using a mollified version of $h$). Existence of such subsequential limits was shown in~\cite{dddf}; subsequently the paper~\cite{gm} proved that such subsequential limits are unique and characterized by the above axioms, and in~\cite{gm_c} it was established that the resulting metric $\fd$ has a \emph{conformal covariance} property. Noting that we can, for instance, use the domain Markov property to write a zero-boundary GFF $\mathring{h}$ on a proper domain $U \subset \mathbb{C}$ as the restriction of a whole-plane GFF $h$ to $U$ plus a continuous function $f$, we can define the $\gamma$-LQG metric $\fd_{\mathring{h}}$ on $U$ corresponding to $\mathring{h}$ as the internal metric $\fd_{h+f}(\cdot,\cdot;U)$, and thus also define $\fd_{\mathring{h}+g}$ for $g$ continuous on $U$ via Weyl scaling. We will review this construction in more detail at the beginning of~\S\ref{section:lqg}. Then, if $U$, $V$ are domains and $\phi: U \to V$ is conformal, and $h$ is a GFF on $U$ plus a continuous function, the conformal covariance property states that almost surely
	\begin{equation*}
		\fd_{h^U}(z,w) = \fd_{h^U\circ\phi^{-1} + Q\log|(\phi^{-1})'|} (\phi(z),\phi(w))
	\end{equation*}
	for all $z,w \in U$.
	
	The reason the scaling in the axiomatic definition of $\fd_h$ is controlled by $\xi$, rather than $\gamma$, is that, since adding a constant~$C$ to $h$ scales $\mu_h$ by $e^{\gamma C}$, it should be true that $\fd_h$ is scaled by~$e^{\xi C}$, where $\xi := \gamma/d_\gamma$ and $d_\gamma$ is the Hausdorff dimension of the $\gamma$-LQG metric. In order to define the metric $\fd_h$, a candidate $d_\gamma$ was needed to state the scaling axiom. For each $\gamma \in (0,2)$ there is such a value, defined in~\cite{dg}, which describes distances in certain discrete approximations of $\gamma$-Liouville quantum gravity. A posteriori, it was shown in~\cite{gp1} that $d_\gamma$ is indeed the Hausdorff dimension of the $\gamma$-LQG metric. 
	
	\subsection{Schramm--Loewner evolutions}
	We briefly recap some basics about SLE. Firstly we recall (e.g., from~\cite[Def.\ 6.1]{lawler}) the construction of chordal $\SLE$ from 0 to $\infty$ in $\h$ using the \emph{chordal Loewner equation}
	\begin{equation}\label{eq:loewner}
		\partial_t g_t(z) = \frac{2}{g_t(z)-U_t}, \quad g_0(z) = z
	\end{equation}
	where $U \colon [0,\infty) \to \R$ is a continuous function. Here $U$ is the so-called \emph{(Loewner) driving function}. For each fixed $z \in \h$ the \emph{Loewner flow}, i.e.\ the solution to~(\ref{eq:loewner}), is defined up until the time $\tau(z) = \inf\{t \geq 0 : \mathrm{Im}(g_t(z)) = 0\}$. If we define the compact hull $K_t = \overline{\{ z \in \h: \tau(z) \le t \}}$, then~$g_t$ is the unique conformal map from $\h \setminus K_t$ to $\h$ that satisfies the \emph{hydrodynamic normalization} $g_t(z) - z \to 0$ as $z \to \infty$. (We also say that a conformal map $f\colon D\to\widehat{D}$ between unbounded domains ``looks like the identity at $\infty$'' if it satisfies $f(z)-z\to 0$ as $z\to\infty$.)
	
	When $U_t = \sqrt{\kappa} B_t$ for some multiple $\kappa > 0$ of a standard Brownian motion $(B_t)$, there almost surely exists a curve $\eta$ parametrized by $t \in [0,\infty)$ such that for each $t$, $\h \setminus K_t$ is the unbounded component of $\h \setminus \eta([0,t])$ and $g_t(\eta(t))=U_t$; we say that $\eta$ \emph{generates} the family of hulls $(K_t)_{t\ge 0}$. Moreover, the curve $\eta$ is determined by $U$. This was proven for $\kappa \neq 8$ in \cite{rohde}; the case $\kappa = 8$ was proven in \cite{lsw} as a consequence of the convergence of the uniform spanning tree Peano curve, but a proof has since been given in \cite{am2022sle8} for the $\kappa = 8$ case which does not rely on discrete models. The law of $\eta$ is, by definition, that of a \textbf{chordal $\SLE_\kappa$ in~$\h$ from~$0$ to~$\infty$}. The one-parameter family of $\SLE_\kappa$ laws for $\kappa>0$ has three distinct phases. When $\kappa \in (0, 4)$, the curve~$\eta$ is almost surely simple and does not hit $\partial\h$ other than at its endpoints. When $4 < \kappa < 8$,~$\eta$ almost surely does hit $\partial\h$ infinitely often, and has a dense set of double points, but does not cross itself \cite{rohde}; in this phase $\eta$ \emph{swallows} points, i.e.\ disconnects them from $\infty$ without hitting them. When $\kappa \ge 8$, $\eta$ is almost surely space-filling. 
	
	The Markov property of Brownian motion implies that $\SLE_\kappa$ has a \emph{conformal Markov property}~\cite[Thm 2.1(ii)]{rohde}: given $\eta|_{[0,t]}$, the conditional law of the image of $\eta|_{[t,\infty)}$ under the map $g_t-U_t$ is the same as the law of the whole curve~$\eta$. The scale invariance of Brownian motion, and the fact that the only conformal automorphisms of $\h$ that fix 0 and~$\infty$ are scalings, imply that $\SLE$ is conformally invariant up to time reparametrization, so that by applying a conformal map chordal $\SLE$ can be defined (up to time reparametrization) between any two distinct boundary points in any simply connected proper domain. 
	
	This definition can be generalized~\cite[\S8.3]{lsw} to the $\SLE_\kappa(\rho_1;\rho_2)$ processes where $\kappa > 0$ and $\rho_1, \rho_2 > -2$, a variant where one additionally keeps track of marked points known as \emph{force points}. The $\SLE_\kappa(\rho_1;\rho_2)$ process (with force points at $0^-$ and $0^+$) is defined from 0 to $\infty$ in $\h$ using~\eqref{eq:loewner}, where this time $U_t$ satisfies the SDE
	\begin{equation*}
		dU_t = \sqrt{\kappa} dB_t + \left(\frac{\rho_1}{U_t-V_t^1} + \frac{\rho_2}{U_t-V_t^2}\right) dt, \quad dV_t^1 = \frac{2}{V_t^1-U_t} dt, \quad dV_t^2 = \frac{2}{V_t^2-U_t} dt
	\end{equation*}
	with initial conditions $V_0^1 = U_0 = V_0^2=0$ and the further condition that $V_t^1 \le U_t \le V_t^2$ for all $t\ge 0$. To motivate the equations for $V_t^1$ and $V_t^2$, note that for any Loewner flow $(g_t)$ from 0 in $\h$ driven by a continuous function $U_t$, if we set$$x_t = \sup\{g_t(x): x<0, x\notin K_t\}, \quad y_t = \inf\{g_t(x): x>0, x\notin K_t\}$$(noting that the $x$ values quantified over are simple boundary points of $\h\setminus K_t$ and thus $g_t$ extends continuously to them), then (\ref{eq:loewner}) gives$$\partial_t x_t = \frac{2}{x_t-U_t}, \quad \partial_t y_t = \frac{2}{y_t-U_t},$$which means in this case that $V_t^1$ and $V_t^2$ can be seen as the images of $0^-$ and $0^+$ (i.e.\ the left-hand and right-hand prime ends of $\h\setminus K_t$ corresponding to 0) under $g_t$. We think of the two extra terms in the SDE for $U_t$ as providing ``forces'' causing the force points to either repel (for positive~$\rho$ values) or attract (for negative $\rho$ values) the driving function $U_t$.
	
	As before, the resulting family of hulls turns out to be generated by a continuous curve~$\eta$, with $g_t(\eta(t)) = U_t$ \cite[Thm 1.3]{ig1}. This defines the law of an $\SLE_\kappa(\rho_1;\rho_2)$ curve. If $x_L \le 0 \le x_R$, one obtains the same result for the \emph{$\SLE_\kappa(\rho_1;\rho_2)$ process with force points at $x_L$, $x_R$} given by replacing the initial conditions with $x_L = V_0^1 \le U_0 = 0 \le V_0^2 = x_R$. If $\rho_1 = 0$ (resp.~$\rho_2 = 0$), the process is known as $\SLE_\kappa(\rho)$ where $\rho = \rho_2$ (resp.~$\rho = \rho_1$). (Note that if $\rho_1=\rho_2=0$ we have an ordinary $\SLE_\kappa$.) 
	
	The driving function $U_t$ of an $\SLE_\kappa(\rho_1;\rho_2)$ process still satisfies Brownian scaling, and thus we have conformal invariance and can define $\SLE_\kappa(\rho_1;\rho_2)$ between any two distinct boundary points of any simply connected proper domain. The conformal Markov property changes slightly: given $\eta|_{[0,t]}$, the conditional law of the image of $\eta|_{[t,\infty)}$ under the map $g_t-U_t$ is that of an $\SLE_\kappa(\rho_1;\rho_2)$ process with the force points at $V_t^1-U_t$ and $V_t^2-U_t$.
	
	Although in the case $\kappa \le 4$ ordinary $\SLE_\kappa$ cannot intersect the boundary except at its endpoints, force points with sufficiently negative weights can make $\SLE_\kappa(\rho_1;\rho_2)$ processes hit the boundary. In particular, an $\SLE_\kappa(\rho_1;\rho_2)$ process from 0 to $\infty$ in $\h$ almost surely hits $(0,\infty)$ if $\rho_2 < \kappa/2-2$, but almost surely does not hit $(0,\infty)$ if $\rho_2 \ge \kappa/2-2$ (see~\cite[Lemma~2.1]{intersections}). The analogous result holds with $\rho_2$ replaced by $\rho_1$ and $(0,\infty)$ replaced by $(-\infty,0)$.
	
	As well as chordal $\SLE$, which goes from one boundary point to another, one can consider \emph{radial $\SLE$}, which grows from a boundary point towards an interior point. First we define radial $\SLE_\kappa$ in the unit disc $\mathbb{D}$ targeted at 0~\cite[Def.\ 6.20]{lawler} to be the set of hulls $K_t$ associated to the family of conformal maps $(g_t)_{t\ge 0}$ solving the \textbf{radial Loewner equation}
	\begin{equation}\label{eq:r_loewner}
		\partial_t g_t(z) = g_t(z) \frac{W_t+g_t(z)}{W_t-g_t(z)},\quad g_0(z) = z
	\end{equation}
	driven by $W_t = e^{i\sqrt{\kappa}B_t}$ where $B$ is a standard Brownian motion. This time the maps~$g_t\colon\mathbb{D}\setminus K_t \to \mathbb{D}$ are normalized by requiring $g_t(0) = 0$ and $g_t'(0) > 0$, which as in the chordal case defines a unique choice for each $g_t$. As with chordal $\SLE$, radial $\SLE$ has a generalization with a force point; we define 
	\begin{equation*}
		\Psi(z,w) = -z\frac{z+w}{z-w}, \quad \widetilde{\Psi}(z,w)=\frac{\Psi(z,w)+\Psi(1/\overline{z},w)}{2},
	\end{equation*}
	and let $(W,O)$ be the solution to the equations
	\begin{align}\label{eq:wo}
			dW_t &= \left[-\frac{\kappa}{2}W_t + \frac{\rho}{2} \widetilde{\Psi}(O_t,W_t)\right] dt + i\sqrt{\kappa}W_t \, dB_t, \\
		\nonumber dO_t &= \Psi(W_t,O_t) \, dt
	\end{align}
	(this solution exists and is unique---see~\cite[\S2.1.2]{ig4}). We can then define a \textbf{radial $\SLE_\kappa(\rho)$} as the process associated to the solution $(g_t)$ of~\eqref{eq:r_loewner} with this driving function $W$. As before, the family of hulls $(K_t)$ is generated by a continuous curve.
	
	Moreover, we can define a version of radial $\SLE_\kappa(\rho)$ in bi-infinite time. The radial $\SLE_\kappa(\rho)$ equations (\ref{eq:wo}) with $B$ a \emph{two-sided} Brownian motion still have a unique solution; if we take $W$ to be the resulting driving function, then there is a family $(\widetilde{g}_t)_{t\in\R}$ of conformal maps onto $\mathbb{C}\setminus \overline{\mathbb{D}}$ that each fix $\infty$, have positive spatial derivative at $\infty$, and satisfy~\eqref{eq:r_loewner} (without the initial condition) \cite[\S2.3]{ig4}. If we define the hull $K_t$ as the complement of the domain of $g_t$, then the family $(K_t)_{t\in\R}$ is generated by a \textbf{whole-plane $\SLE_\kappa(\rho)$} process from 0 to $\infty$.
	
	For $\kappa > 4$ and $\rho_1, \rho_2 \in (-2,\kappa/2-2)$, the \textbf{space-filling $\SLE_\kappa(\rho_1,\rho_2)$} process was defined in \cite{ig4}. When $\kappa \ge 8$, this coincides with ordinary $\SLE_\kappa(\rho_1,\rho_2)$, which as mentioned above is almost surely space-filling. When $\kappa \in (4,8)$ one starts with an ordinary $\SLE_\kappa(\rho_1,\rho_2)$ $\eta'$ and extends it by sequentially ``filling in'' the regions $\eta'$ disconnects from $\infty$. Indeed, for each component $C_i$ of the complement of $\eta'$, there is a first time $t_i$ such that $\eta'|_{[0,t_i]}$ disconnects $C_i$ from $\infty$. We then define the space-filling $\SLE_\kappa(\rho_1,\rho_2)$ to hit the points in the range of $\eta'$ in the same order that $\eta'$ does, but so that immediately after hitting $\eta'(t_i)$ it traces a $C_i$-filling $\SLE_\kappa$-type loop beginning and ending at $\eta'(t_i)$, constructed using a coupling with the Gaussian free field. This construction is described in \cite[\S1.2.3]{ig4}.
	
	Finally, one can define a \textbf{whole-plane space-filling $\mathrm{SLE}_\kappa$} from $\infty$ to $\infty$ using the chordal version as explained in~\cite[Footnote 4]{dms}. For $\kappa \in (4,8)$, one first uses the SLE/GFF coupling to draw SLE-type curves partitioning the plane into a countable collection of pockets, and then concatenates chordal space-filling $\mathrm{SLE}_\kappa$ curves in each pocket.
	
	\subsection{Notation}
	If $(E(C))_{C\in S}$ is a family of events indexed by a set $S \subseteq \R$ which is unbounded above, we say that $E(C)$ happens with \textbf{superpolynomially high probability} as $C\to\infty$ if for any $N\in\mathbb{N}$ we have $\p[E(C)^c] = O(C^{-N})$ as $C\to\infty$. If $E(C)$ depends on other parameters, we say~$E(C)$ happens with superpolynomially high probability \emph{at a rate which is uniform in} some subset of those parameters if the bounds on $\p[E(C)^c]C^N$ can be chosen not to depend on that subset of those parameters. Similarly, we say that a function $f$ \textbf{decays superpolynomially} if for all $N$ we have $f(x) = O(x^N)$ as $x\to 0$. 
	
	For $z\in \mathbb{C}$ and $r>0$, $B(z,r)$ and $\overline{B}(z,r)$ will always mean, respectively, the \emph{Euclidean} open and closed balls of radius~$r$ centred at $z$; we will define notation ad hoc for balls of other metrics. Likewise, $\mathrm{diam}$ alone will denote Euclidean diameter.
	\section{$\gamma$-LQG metric boundary estimates for the free-boundary GFF on $\overline{\h}$}\label{section:lqg}
	Throughout this section $h$ will be a free-boundary GFF on $\h$, though not always with the same choice of additive constant. Indeed, although the statements of our results require the additive constant for $h$ to be fixed in some way, it is easily seen that all the results of this section remain true regardless of how the constant is fixed, so we will not always specify a choice. In this section we show that the $\gamma$-LQG metric induced by $h$ extends continuously to a metric on $\overline{\h} \times \overline{\h}$, and give some estimates for the regularity of this metric, showing that, almost surely, it is locally H\"{o}lder continuous with respect to the Euclidean metric on $\overline{\h}$.)
	
	For a fixed $\gamma \in (0,2)$ we denote by $\fd_h$, $\mu_h$, $\nu_h$ respectively the $\gamma$-LQG metric, area measure and boundary length measure associated to $h$ on $\h$. As noted previously, the $\gamma$-LQG metric was constructed in~\cite{gm} for the whole-plane GFF, but it is explained in~\cite[Remark 1.5]{gm} how to adapt this to get the LQG metric on a proper domain $U \subset \mathbb{C}$ associated to $\mathring{h}+f$ where $\mathring{h}$ is a zero-boundary GFF on $U$ and $f$ is a continuous function on $U$; this is done as follows. If $h^{\mathrm{wp}}$ is a whole-plane GFF, then we can write $h^{\mathrm{wp}}|_{U} = \mathring{h} + \widehat{\mathfrak{h}}$ where $\mathring{h}$ is a zero-boundary GFF on $U$ and $\widehat{\mathfrak{h}}$ is a random harmonic function (modulo additive constant) independent of $\mathring{h}$. Recall that $\mathring{h}$ and $\widehat{\mathfrak{h}}$ are the projections of $h$ onto the spaces of functions that are, respectively, supported in $U$ and harmonic in $U$. Note that fixing the additive constant for~$h$ corresponds to fixing that for $\widehat{\mathfrak{h}}$ but may or may not preserve the independence of $\mathring{h}$ and $\widehat{\mathfrak{h}}$; for instance, we can fix the constant by requiring $\widehat{\mathfrak{h}}(z)=0$ for some choice of $z\in U$, in which case $\widehat{\mathfrak{h}}$ as a bona fide (random) function is still independent of $\mathring{h}$, or we can require that the average of $h$ on some circle $\Gamma \subset U$ vanishes, in which case $\mathring{h}$ and $\widehat{\mathfrak{h}}$ are \emph{not} independent, since their averages on $\Gamma$ are required to sum to zero. 
	
	Having fixed the additive constant in some way---whether or not $\widehat{\mathfrak{h}}$ with the constant fixed is independent of $\mathring{h}$---we can define~$\fd_{\mathring{h}}$ on~$U$ as a Weyl scaling of the internal metric induced by $\fd_{h^{\mathrm{wp}}}$ on $U$, i.e.
	\begin{equation*}
		\fd_{\mathring{h}}(\cdot,\cdot) = e^{-\xi \widehat{\mathfrak{h}}} \cdot \fd_{h^{\mathrm{wp}}} (\cdot,\cdot;U).
	\end{equation*}
	(This is well-defined since the definition of the internal metric only involves paths which stay in~$U$, so it does not matter that $\widehat{\mathfrak{h}}$ does not extend continuously to the boundary.) Moreover one can define $\fd_{\mathring{h}+f} = e^{\xi f}\cdot \fd_{\mathring{h}}$ for $f$ continuous on $U$. It is easy to see that $\fd_{\mathring{h}+f}$ thus defined is a metric on $\h$ that satisfies the axioms in~\cite[\S1.2]{gm} and conformal covariance. Observe also that $\fd_{\mathring{h}+f}$ induces the Euclidean topology on $\h$. Indeed, the internal metric $\fd_{h^{\mathrm{wp}}}(\cdot,\cdot;\h)$ is at least as large as~$\fd_{h^{\mathrm{wp}}}$, so since Euclidean open sets in $\h$ are open w.r.t.\ $\fd_{h^{\mathrm{wp}}}$ they must also be open w.r.t.\ the internal metric. Around each point $z\in\h$ the $\fd_{h^{\mathrm{wp}}}$-metric balls of sufficiently small radius must be contained in $\h$, so coincide with the $\fd_{h^{\mathrm{wp}}}(\cdot,\cdot;\h)$-metric balls of the same radius. These thus contain Euclidean open discs, which shows that $\fd_{h^{\mathrm{wp}}}(\cdot,\cdot;\h)$ induces the Euclidean topology. Since $-\mathfrak{h}+f$ is a continuous function on $\h$ (and thus locally bounded) the same has to be true for~$\fd_{\mathring{h}+f}$. 
	
	Since we can write $h = \mathring{h} + \widetilde{\mathfrak{h}}$ for $\mathring{h}$ a zero-boundary GFF on $\h$ and $\widetilde{\mathfrak{h}}$ a (random) harmonic function on $\h$, we may define $\fd_h$ (as a function on $\h\times\h$) similarly. Recall that $h$ and thus $\widetilde{\mathfrak{h}}$ are only defined modulo a global additive constant, so the above construction only defines $\fd_h$ modulo a multiplicative constant. Once the constant is fixed, it follows (as above for $\fd_{\mathring{h}+f}$) that this $\fd_h$ is a metric on $\h$ that induces the Euclidean topology on $\h$ and satisfies the axioms in~\cite[\S1.2]{gm} and conformal covariance. As noted, we will often be able to fix the constant somewhat arbitrarily. Note however that the same caveat applies as above: not every choice for fixing the constant makes the zero-boundary and harmonic parts of $h$ independent.
	
	We can extend the LQG metric to the boundary of $\h$ as follows. Firstly, we say that a path $P\colon [a,b] \to \overline{\h}$ is \textbf{admissible} if $P^{-1}(\partial \h)$ is finite, and define the \textbf{$\widehat{d}_h$-length} of such a path~$P$ to be
	\begin{equation*}
		\widehat{d}_h(P) := \sup{ \left\{ \sum_{i=1}^n\fd_h(P(t_{i-1}),P(t_i))  :\; a \le t_0 < t_1 < \cdots < t_n \le b, \; P(t_i) \in \h \right\} }.
	\end{equation*}
	$P^{-1}(\h)$ can be written uniquely as a finite union of disjoint intervals $I$, each of which is open as a subset of $[a,b]$; it is straightforward to check that the lengths of the $P|_{\overline{I}}$ sum to the length of $P$. 
	
	We now define the $\widehat{d}_h$-distance between two points of $\overline{\h}$ as the infimum of the lengths of admissible paths between them. To see that this definition actually does restrict to $\fd_h$ on $\h \times \h$, note that for $z,w\in\h$ we know that $\fd_h(z,w)$ is finite (indeed, one can find a path between $z$ and $w$ of finite $\fd_{\mathring{h}}$-length $L$ that stays in some bounded open set $U$ at positive distance from $\partial\h$, then we have $\fd_h(z,w)\le L\sup_U e^{\xi\widetilde{\mathfrak{h}}} < \infty$ ). Given $\ve > 0$, we can then take a path $P$ in $\overline{\h}$ with $\widehat{d}_h$-length in $[\widehat{d}_h(z,w), \widehat{d}_h(z,w) + \ve)$, and thus find a subdivision of that path with
	\begin{equation*}
		\sum_{i=1}^n \fd_h(P(t_{i-1}),P(t_i)) \le \widehat{d}_h(z,w) + \ve.
	\end{equation*}
	We know that $\fd_h$ is almost surely a length metric, so for each $i$ we can find a path from $P(t_{i-1})$ to $P(t_i)$ in $\h$ with $\fd_h$-length at most $ \fd_h(P(t_{i-1}),P(t_i)) + \ve/n$ and concatenate these to see that $\fd_h(z,w) \le \widehat{d}_h(z,w)+2\ve$. On the other hand clearly $\fd_h \ge \widehat{d}_h$ so the two must agree. Henceforth we will use~$\fd_h$ to refer to the function extended to all of $\overline{\h}$ (which we will show is a metric on $\overline{\h}$).
	\subsection{Joint H\"{o}lder continuity of the semicircle average}\label{subsection:hmp}
	For a point $x \in \R$ and $\ve > 0$, recall that $h_\ve(x)$ denotes the average of $h$ on the semicircular arc $\partial B(x,\ve) \cap \h$, defined as $(h,\rho^+_{x,\ve}) = -2\pi (h,\Delta^{-1} \rho^+_{x,\ve})_\nabla$ where $\rho^+_{x,\ve}$ is the uniform probability measure on $\partial B(x,\ve) \cap \h$. Since $h$ is only defined modulo additive constant, we have to fix the constant in order for the $h_\ve(x)$ to be well defined---the results of this subsection (\S\ref{subsection:hmp}) hold however the constant is fixed, but for concreteness, we will state and prove them using the normalization $h_1(0)=0$. We can establish continuity for this semicircle average via the same Kolmogorov--\v{C}entsov-type argument as in~\cite[Prop.~2.1]{hmp}.
	\begin{lemma}\label{lem:cty}
		Let $h$ be a free-boundary GFF on $\h$ with the additive constant fixed so that $h_1(0)=0$. There exist $\alpha, \beta > 0$ such that, for any $U \subseteq \mathbb{R}$ bounded and open, $\zeta > 1/\alpha$ and $\gamma \in (0, \beta/\alpha)$, there is a modification $\widetilde{X}$ of the process $X(z,r) = h_r(z)$ such that, for some random $M\in(0,\infty)$,
		\begin{equation*}
			|\widetilde{X}(z,r)-\widetilde{X}(w,s)| \le M \left( \log{\frac{2}{r}}\right)^\zeta \frac{|(z,r)-(w,s)|^\gamma}{r^\frac{1+\beta}{\alpha}}
		\end{equation*}
		whenever $z,w \in U$, $r, s \in (0, 1]$ and $1/2 \le r/s \le 2$. (This is unique in that any two such modifications are almost surely equal, by continuity.)
	\end{lemma}
	\begin{proof}
		By the ``modified Kolmogorov--\v{C}entsov'' result~\cite[Lemma~C.1]{hmp} it suffices to show that there exist $\alpha, \beta, C > 0$ such that for all $z,w \in U$ and $r,s \in (0,1]$ we have
		\begin{equation}\label{eq:hrzhsw}
			\mathbb{E}[|h_r(z)-h_s(w)|^\alpha] \le C \left( \frac{|(z,r)-(w,s)|}{r\wedge s} \right)^{2+\beta}.
		\end{equation}	
		Thus we can show continuity for the semicircle average by bounding the absolute moments of $h_r(z)-h_s(w)$. In fact, since this is a centred Gaussian, we need only bound its second moment. We can do this by considering the Green's function for $h$, given by the Neumann Green's function in $\h$:
		\begin{equation*}
			G(x,y) = -\log{|x-\bar{y}|}-\log{|x-y|}.
		\end{equation*}
		This $G$ is the Green's function such that
		\begin{equation*}
			-\Delta^{-1} \rho(\cdot) = \frac{1}{2\pi} \int_\h G(\cdot,y) \rho(y) \, dy.
		\end{equation*}
		Recall that $\rho_{x,\ve}$ denotes the uniform probability measure on $\partial B(x,\ve)$ and $\rho_{x,\ve}^+$ denotes that on $\partial B(x,\ve) \cap \h$. Since
		\begin{align*}
			\nonumber\int G(z,y) \, \rho_{x,\ve}^+(dy)& = \int (-\log{|z-\bar{y}|}-\log{|z-y|}) \, \rho_{x,\ve}^+(dy) \\
			&= \int -2\log{|z-y|} \, \rho_{x,\ve}(dy) = -2\log{\max{(|z-x|,\ve)}},
		\end{align*}
		we find that 
		\begin{align}
			\nonumber &\mathbb{E}[(h_s(w)-h_r(z))(h_s(w)-h_r(z))] \\
			= &\int [ -2\log{\max{(|\zeta-w|,s)}} +2\log{\max{(|\zeta-z|,r)}} ] \, (\rho_{w,s}^+ - \rho_{z,r}^+)(d\zeta) \label{eq:greenfn}\\
			\nonumber\le &\int \left|-2\log{\max{(|\zeta-w|,s)}} +2\log{\max{(|\zeta-z|,r)}}\right| \, (\rho_{w,s}^+ + \rho_{z,r}^+)(d\zeta) \\
			\nonumber \le &\int 2\, \frac{|z-w|+|r-s|}{r\wedge s} \, (\rho_{w,s}^+ + \rho_{z,r}^+)(d\zeta) = 4\, \frac{|z-w|+|r-s|}{r\wedge s}.
		\end{align}
		Here we used (as in the proof of~\cite[Prop.~2.1]{hmp}) that $|\log{\frac{a}{b}}| \le \frac{|a-b|}{a\wedge b}$ for $a,b>0$, and that $|(a\vee b)-(c\vee d)| \le |a-c| \vee |b-d|$ for all $a,b,c,d$. Since $h_s(w)-h_r(z)$ is a centred Gaussian, we now have that for every $\alpha > 0$ there is $C_\alpha$ such that
		\begin{equation*}
			\mathbb{E}[|h_s(w)-h_r(z)|^\alpha] \le C_\alpha \left(\frac{|(z,r)-(w,s)|}{r\wedge s}\right)^\frac{\alpha}{2},
		\end{equation*} 
		concluding the proof. Observe also that when $w=z$ and $s>r$ the integral (\ref{eq:greenfn}) becomes
		\begin{equation*}
			\int 2\log\left( \frac{\max(|\zeta-z|,r)}{\max(|\zeta-z|,s)} \right) (\rho^+_{z,s}(d\zeta)-\rho^+_{z,r}(d\zeta));
		\end{equation*}
		the integral w.r.t.\ $\rho^+_{z,s}(d\zeta)$ vanishes whilst the integral w.r.t.\ $\rho^+_{z,r}(d\zeta)$ gives $2\log{(s/r)}$. A similar computation shows that the increments $h_s(z)-h_r(z)$ and $h_t(z)-h_u(z)$ have zero covariance when $r < s \le t < u$, which together with continuity implies that $\widetilde{X}(z,e^{-t})-\widetilde{X}(z,1)$ evolves as $\sqrt{2}$ times a two-sided standard Brownian motion.
	\end{proof}
Note that, since the boundary conditions are Neumann rather than Dirichlet, the semicircle average process evolves as $\sqrt{2}$ times a Brownian motion, as opposed to circle averages which yield standard Brownian motion. This remains true for the free-boundary GFF; since we will need it later and the calculation is similar, we will now give a corresponding estimate for circle averages of the free-boundary GFF.
\begin{lemma}\label{lem:var}
	Let $h$ be a free-boundary GFF with the additive constant fixed such that $h_1(0)=0$. Let $K\subset\overline{\h}$ be compact. Then there exists a constant $C=C(K)$ such that, for all $w\in K$ and $s\in(0,1]$ such that $B(w,s)\subset\overline{\h}$, we have$$\mathrm{var}\, [h^\mathrm{circ}_s(w)-h^\mathrm{circ}_1(i)] \le -\log{s}+C(K) .$$
\end{lemma}
\begin{proof}
	First we compute
	\begin{align*}
		\int G(z,y) \rho_{x,\ve}(dy) &= \int (-\log|z-\bar{y}| -\log|z-y|) \rho_{x,\ve}(dy) \\
		&= -\log\max(|z-x|,\ve) -\log\max(|z-\bar{x}|,\ve).
	\end{align*}
	Thus, for $w, z\in \h$ and $r, s>0$ such that $r\le \mathrm{Im}\, z$, $s\le\mathrm{Im}\, w$,
	\begin{align*}
		\mathbb{E} [(h^\mathrm{circ}_s&(w)-h^\mathrm{circ}_r(z))(h^\mathrm{circ}_s(w)-h^\mathrm{circ}_r(z))] \\
		= \int \bigg[ &-\log\max(|\zeta-w|,s) -\log\max(|\zeta-\bar{w}|,s) \\
		&+\log\max(|\zeta-z|,r) +\log\max(|\zeta-\bar{z}|,r)  \bigg] \, (\rho_{w,s}-\rho_{z,r})(dy) \\
		= \int [&-\log\max(|\zeta-w|,s) +\log\max(|\zeta-z|,r) ] \, (\rho_{w,s}-\rho_{z,r})(dy),
	\end{align*}
	where in the last line we used that $|\zeta-\bar{w}|\ge s$ and $|\zeta-\bar{z}|\ge r$ for $\zeta\in\h$, so the corresponding integrals w.r.t.~$\rho_{w,s}$ and~$\rho_{z,r}$ cancel.
	
	Setting $z=i$, $r=1$, note that the term $\log\max(|\zeta-i|,1)$ vanishes on $\partial B(i,1)$ and is bounded above by some constant (depending only on $K$) on the closed Euclidean 1-neighbourhood of $K$, whereas $\log\max(|\zeta-w|,s)$ is equal to $\log{s}$ on $\partial B(w,s)$ and bounded above by some constant (depending only on $K$) on $\partial B(i,1)$. The claimed result follows.
\end{proof}
	\subsection{Thick points on the boundary}
	We refer to $x\in\mathbb{R}$ as an \textbf{$\alpha$-thick point} if
	\begin{equation*}
		\lim_{r\to 0} \frac{h_r(x)}{\log(1/r)} =\alpha.
	\end{equation*}
	
	Our aim in this subsection is to show that boundary points have maximum thickness 2. This matches the maximum thickness in all of $\h$ for the zero-boundary GFF, as calculated in~\cite{hmp}. This is because $\partial\h$ has Euclidean dimension half that of $\h$, but the semicircle averages centred at boundary points for the free-boundary GFF behave like $\sqrt{2}$ times the circle averages of the zero-boundary GFF; these two effects cancel each other out.
	\begin{lemma}\label{lem:thicc}
		Let $h$ be a free-boundary GFF with the additive constant fixed such that $h_1(0)=0$. Almost surely, for every $x\in\R$ we have
		\begin{equation}\label{eq:thicc}
			\limsup_{r\to 0} \frac{h_r(x)}{\log{(1/r)}} \le 2.
		\end{equation}	
	\end{lemma}
	The proof is an application of the argument that proves~\cite[Lemma~3.1]{hmp}.
	\begin{proof}
		Fix $a>2$, and choose $\ve$ such that $0 < \ve < \frac{1}{2} \wedge (\frac{a^2}{8}-\frac{1}{2})$. For each $n \in \mathbb{N}$ let $r_n = n^{-1/\ve}$. By setting $U=(-1,1)$, $\zeta = \frac{1}{2}$, $\alpha = \frac{16}{\ve}$, $\beta = \frac{8}{\ve} - 2$, $\gamma = \frac{1}{2}-\frac{3\ve}{16}$ in Lemma~\ref{lem:cty} there is a random $M \in \R$ such that
		\begin{equation}\label{eq:tholder}
			|h_{s'}(y) - h_s(x)| \le M \left(\log{\frac{2}{s'}}\right)^\frac{1}{2} \frac{|(y,s')-(x,s)|^{\frac{1}{2}-\frac{3\ve}{16}}}{s'^{\frac{1}{2}-\frac{\ve}{16}}},
		\end{equation}
		whenever $x,y \in (-1,1)$, $s,s' \in (0,1]$ with $1/2\le s'/s\le 2$.
		Thus, for all $x \in (0,1)$ and $n > (2^\ve-1)^{-1}$ (so that $r_n/r_{n+1} \in (1,2)$), and all $\log{\frac{1}{r_n}} < t \le \log{\frac{1}{r_{n+1}}}$, we have
		\begin{align*}
			\nonumber |h_{e^{-t}}(x) - h_{r_n}(x)| &\le M \left(\log{\frac{2}{r_n}}\right)^\frac{1}{2} \frac{\left(r_n-r_{n+1} \right)^{\frac{1}{2}-\frac{3\ve}{16}}}{r_n^{\frac{1}{2}-\frac{\ve}{16}}} \\
			\nonumber &= M \left( \log{2} + \frac{1}{\ve} \log{n}\right)^\frac{1}{2}  \left( 1-\left( \frac{n}{n+1} \right)^\frac{1}{\ve} \right)^{\frac{1}{2}-\frac{3\ve}{16}}  n^\frac{1}{8} \\
			\nonumber & \le 2M \ve^{-\frac{1}{2}} (\log{n})^\frac{1}{2} \left( 1-\left( \frac{n}{n+1} \right)^\frac{1}{\ve} \right)^{\frac{1}{2}-\frac{3\ve}{16}}  n^\frac{1}{8} \\
			&\le 2M\ve^{-1+\frac{3\ve}{16}} n^{-\frac{3}{8}+\frac{3\ve}{16}} (\log{n})^\frac{1}{2},
		\end{align*}
	where in the last step we used that $(1-1/(n\ve))^\ve < (1-1/((n+1)\ve))^\ve < 1-1/(n+1)=n/(n+1)$, so that $1-(n/(n+1))^{1/ve}\le 1/(n\ve)$. This shows
		\begin{equation}\label{eq:rns}
			\limsup_{r\to 0} \frac{h_r(x)}{\log{(1/r)}} > a \Leftrightarrow \limsup_{n\to\infty} \frac{h_{r_n}(x)}{\log{(1/r_n)}} > a.
		\end{equation}
		Given $x \in (0,1)$ we can find $k \in \mathbb{N}$ such that $kr_n^{1+\ve} \in (0,1)$ with $|x-kr_n^{1+\ve}| < r_n^{1+\ve}$. Then (\ref{eq:tholder}) gives
		\begin{equation*}
			|h_{r_n}(x)-h_{r_n}(kr_n^{1+\ve})| < M \left(\log{\frac{2}{r_n}}\right)^\frac{1}{2} r_n^{(1+\ve)\frac{\ve}{8}} =  M \left(\log{\frac{2}{r_n}}\right)^\frac{1}{2} n^{-\frac{1}{8}(1+\ve)}.
		\end{equation*}
		So if the right-hand side of (\ref{eq:rns}) holds for some $x\in(0,1)$, then for some $\delta >0$ there are infinitely many $n$ for which $h_{r_n}(x)>(a+\delta)\log{\frac{1}{r_n}}$, and thus infinitely many $n$ for which some $k \in \{1, 2, \ldots, \lfloor r_n^{-(1+\ve)} \rfloor \}$ makes the event $A_{n,k} = \{ h_{r_n}(kr_n^{1+\ve}) > a\log{\frac{1}{r_n}} \}$ hold. We can now apply a Gaussian tail bound. Indeed, for each $x$, $(h_{e^{-t}}(x)-h_1(x))_t$ is $\sqrt{2}$ times a standard Brownian motion, independent of $h_1(x) = h_1(x)-h_1(0)$ (recall our choice of additive constant) whose variance is at most some constant $C$ for $x \in (0,1)$, as follows from (\ref{eq:hrzhsw}). Thus $h_{r_n}(kr_n^{1+\ve}) \sim \mathcal{N}(0,c+2\log{\frac{1}{r_n}})$ for some $c\le C$. Thus for any $\eta > 0$, if $n$ is sufficiently large we have 
		\begin{align*}
			\p[A_{n,k}] = \p\left[Z > \frac{a\log{\frac{1}{r_n}}}{\sqrt{c+2\log{\frac{1}{r_n}}}} \right] < \p\left[Z > \frac{(a-\eta)\log{\frac{1}{r_n}}}{\sqrt{2 \log{\frac{1}{r_n}}}}\right] < r_n^{\frac{(a-\eta)^2}{4}},
		\end{align*}
		where $Z \sim \mathcal{N}(0,1)$. So by a union bound we have, for $n$ sufficiently large depending on $\eta$,
		\begin{equation*}
			\p\left[\bigcup_{k=1}^{\lfloor r_n^{-(1+\ve)} \rfloor} A_{n,k}\right] \le r_n^{\frac{(a-\eta)^2}{4} - 1 -\ve} = n^{1+\frac{1}{\ve}\left(1-\frac{(a-\eta)^2}{4}\right)}.
		\end{equation*}
		Since $1+\frac{1}{\ve}\left(1-\frac{a^2}{4}\right) < -1$, we can choose $\eta$ so that $1+\frac{1}{\ve}\left(1-\frac{(a-\eta)^2}{4}\right) < -1$, in which case the Borel--Cantelli lemma gives
		\begin{equation*}
			\p\left[\bigcap_{m=1}^\infty \bigcup_{n=m}^\infty \bigcup_{k=1}^{\lfloor r_n^{-(1+\ve)} \rfloor} A_{n,k}\right] = 0.
		\end{equation*}
		Thus almost surely there is no $x \in (0,1)$ for which
		\begin{equation*}
			\limsup_{r\to 0} \frac{h_r(x)}{\log{(1/r)}} > a.
		\end{equation*}
		By translation invariance we can conclude that there is no such $x \in \R$; all that changes is the bound $c$ on the variance of $h_1(x)$, but we will still locally have boundedness. Since $a>2$ was arbitrary we are done.
	\end{proof} 
	\begin{remark}
		In the proof of Lemma~\ref{lem:thicc}, we find that almost surely there is some $N$ such that $\bigcup_k A_{n,k}$ never happens for $n \ge N$, which gives a uniform bound on $h_r(x)/\log{(1/r)}$ for $x \in (0,1)$ and $r \le r_N$. Using this, and translation invariance, we can deduce the stronger statement that, almost surely, for every $K \subset \R$ compact we have
		\begin{equation}\label{eq:uthicc}
			\limsup_{r\to 0} \, \sup_{x\in K} \, \frac{h_r(x)}{\log{(1/r)}} \le 2.
		\end{equation}
		Indeed, since we have
		\begin{equation*}
			\p\left[ \bigcup_{n=m}^\infty \bigcup_{k=1}^{\lfloor r_n^{-(1+\ve)} \rfloor} A_{n,k}\right] = O\left(m^{2+\frac{1}{\ve}\left(1-\frac{(a-\eta)^2}{4} \right)}\right),
		\end{equation*}
		it in fact holds that, as $r\to 0$,
		\begin{equation}\label{eq:quthicc}
			\p\left[ \sup_{s\le r} \, \sup_{x\in K} \, \frac{h_s(x)}{\log{(1/s)}} > a \right] = O\left( r^{\frac{(a-\eta)^2}{4}-1-2\ve} \right),
		\end{equation}
		where since $\eta$ and $\ve$ can be made arbitrarily small the exponent is arbitrary subject to being less than $a^2/4-1$.
	\end{remark}
	\subsection{Controlling pointwise distance via semicircle averages}\label{subsection:dub}
	We give an analogue of part of~\cite[Prop.~3.14]{dfgps}.
	\begin{prop}\label{prop:dub}
		Fix $a \in \R$ and $r > 0$. Let $h$ be a free-boundary GFF on $\h$ and let $K \subset \h \cap \partial B(0,1)$ be a closed arc consisting of more than a single point. Then we have
		\begin{equation}\label{eq:a2k}
			\fd_h(a,a+rK) \le C \int_{\log{(1/r)}}^\infty e^{\xi(h_{e^{-t}}(a)-Qt )} \, dt
		\end{equation}
		with superpolynomially high probability as $C \to\infty$, at a rate which is uniform in $a$ and $r$.
	\end{prop}
	\begin{proof}
		Without loss of generality fix the additive constant for $h$ such that $h_1(a)=0$ (note that the statement does not depend on the normalization, since adding a constant $c$ to the field scales both sides by $e^{\xi c}$). 
		
		Fix a coupling of $h$ with $h^{\mathrm{wp}}$, a whole-plane GFF with the additive constant fixed so that $h^{\mathrm{wp}}_1(a)=0$, such that $h^{\mathrm{wp}}|_\h = \mathring{h} + \widehat{\mathfrak{h}}$ and $h = \mathring{h} + \widetilde{\mathfrak{h}}$, where $\mathring{h}$ is a zero-boundary GFF on $\h$ and~$\widehat{\mathfrak{h}}$ and $\widetilde{\mathfrak{h}}$ are random functions harmonic in $\h$. (Note that $h_1(a)$ is a semicircle average whereas $h^{\mathrm{wp}}_1(a)$ is a circle average; note also that although we could choose the normalizations differently to make $\mathring{h}$, $\widehat{\mathfrak{h}}$ and $\widetilde{\mathfrak{h}}$ independent, we do not do so in this proof as we will not require this.)
		
		Let $\mathfrak{h} = \widetilde{\mathfrak{h}}-\widehat{\mathfrak{h}}$; then $\fd_h = e^{\xi \mathfrak{h}} \cdot \fd_{h^{\mathrm{wp}}|_{\h}}$, by Weyl scaling. We will bound $\fd_h(a+e^{1-n}K,a+e^{-1-n}K)$ for each $n \in \mathbb{N}$. Let $U$ be a bounded connected open set containing $eK \cup K \cup e^{-1}K$ and at positive distance from $\partial \h$. Then by~\cite[Prop.~3.1]{dfgps}\footnote{Many of the results in~\cite{dfgps} involve constants $\mathfrak{c}_r$ for each $r>0$, which describe the scaling of LQG distances. In~\cite[Thm~1.5]{dfgps} a ``tightness'' result is obtained for the $\mathfrak{c}_r$ in lieu of actual scale invariance, which was established later in \cite{gm}; in this work, we will thus use the subsequent result (see~\cite[Thm~1.8]{gm}) that we can take $\mathfrak{c}_r = r^{\xi Q}$.}, with superpolynomially high probability as $A \to \infty$, at a rate which is uniform in $n$, we have
		\begin{equation}\label{eq:3.1hat}
			\fd_{h^{\mathrm{wp}}}(a+e^{1-n}K,a+e^{-1-n}K;a+e^{-n}U) \le A e^{\xi (h^{\mathrm{wp}}_{e^{-n}}(a) - Qn)}.
		\end{equation}
		(Recall that, for an open set $V \subset \mathbb{C}$, $\fd_{h^{\mathrm{wp}}}(\cdot,\cdot;V)$ is the internal metric on $V$ induced by $\fd_{h^{\mathrm{wp}}}$.) Since $a+U$ is at positive distance from $\partial \h$, $\mathfrak{h}$ is almost surely bounded on $a+U$. Thus the variables
		\begin{equation*}
			\{\mathfrak{h}(z) - [h_1(a)-h^{\mathrm{wp}}_1(a)]: z\in a+ U \}
		\end{equation*}
		form an almost surely bounded Gaussian process, so by the Borell--TIS inequality (see~\cite[Thm~2.1.1]{at}) the supremum has a Gaussian tail: there exist $c_1,c_2 > 0$ such that for all $M$ sufficiently large we have
		\begin{equation*}
			\p\left[\sup_{z\in a+ U}{(\mathfrak{h}(z) - [h_1(a)-h^{\mathrm{wp}}_1(a)])} > M \right] < c_1 e^{-c_2M^2}.
		\end{equation*}
		Setting $u>0$ to be the Euclidean distance between $U$ and $\partial\h$ and writing $h_r^\mathrm{circ}(z)$ for the average of $h$ on a circle $\partial B(z,r) \subset \overline{\h}$, we can write 
		\begin{equation*}
			\sup_{z\in a+ U}{(\mathfrak{h}(z) - [h_1(a)-h^{\mathrm{wp}}_1(a)])} = \sup_{z\in a+ U} ( (h_u^\mathrm{circ}(z)-h_1(a)) - (h_u^\mathrm{wp}(z)-h^{\mathrm{wp}}_1(a)) ),
		\end{equation*}
		where both differences on the right hand side are independent of how the additive constants for~$h$ and~$h^\mathrm{wp}$ are fixed, since they only depend on the fields $h$ and $h^{\mathrm{wp}}$ when integrated against mean-zero test functions). Thus, by scale invariance, with the same $c_1, c_2$ we have for every $n$
		\begin{equation*}
			\p\left[\sup_{z\in a+ e^{-n}U}{(\mathfrak{h}(z) - [h_{e^{-n}}(a)-h^{\mathrm{wp}}_{e^{-n}}(a)])} > M \right] < c_1 e^{-c_2M^2}.
		\end{equation*}
		It follows that
		\begin{equation}\label{eq:fraktail}
			\sup_{z\in a+ e^{-n}U}{\mathfrak{h}(z) - [h_{e^{-n}}(a)-h^{\mathrm{wp}}_{e^{-n}}(a)]} \le \log{A} 
		\end{equation} 
		with superpolynomially high probability as $A \to \infty$, at a rate which is uniform in $n$. Since 
		\begin{align*}
			&\fd_h(a+e^{1-n}K,a+e^{-1-n}K;a+e^{-n}U) \\
			&\le e^{\xi (\sup_{z\in a+e^{-n}U}{\mathfrak{h}(z)})} \fd_{h^{\mathrm{wp}}}(a+e^{1-n}K,a+e^{-1-n}K;a+e^{-n}U)
		\end{align*}
		and $h_{e^{-n}}(a) = h^{\mathrm{wp}}_{e^{-n}}(a) + \mathfrak{h}_{e^{-n}}(a)$, we find that 
		\begin{equation}\label{eq:across}
			\fd_h(a+e^{1-n}K,a+e^{-1-n}K;a+e^{-n}U) \le A e^{\xi (h_{e^{-n}}(a) - Qn)}
		\end{equation}
		on the intersection of the events of (\ref{eq:3.1hat}) with  $A$ replaced by $A^{1/2}$ and (\ref{eq:fraktail}) with $A$ replaced by~$A^{1/(2\xi)}$; the probability of this event is superpolynomially high as $A\to\infty$.
		
		By replacing $U$ by a suitable bounded connected open neighbourhood $\widetilde{U}$ of $K$, again with positive distance to $\partial \h$, and using compact subsets of $\widetilde{U}$ on either side of $K$, a similar argument shows that for each $n$ there is a path $\gamma_n$ in $e^{-n}\widetilde{U}$ whose intersection with $U$ disconnects $e^{1-n}K$~and~$e^{-1-n}K$ in $U$ such that \begin{equation}\label{eq:gamman}
		\mathrm{length}(\gamma_n;\fd_h)	\le A e^{\xi (h_{e^{-n}}(a) - Qn)}
		\end{equation} with superpolynomially high probability as $A\to\infty$, uniformly in $n$. This provides the adaptation of~\cite[Prop.~3.1]{dfgps} that we need---namely, fixing $\zeta>0$ small, as $C\to\infty$ the probability is superpolynomially high that (\ref{eq:across}) holds, and there is a path $\gamma_n$ such that (\ref{eq:gamman}) holds, with $A=C$ whenever $n\le C^{1/\zeta}$ and with $A=n^\zeta$ whenever $n > C^{1/\zeta}$. Stringing together the paths $\gamma_n$ with paths of near-minimal length connecting $a+e^{1-n}K$ and $a+e^{-1-n}K$ for each $n$, we find that
		\begin{equation}
			\fd_h(a,a+rK) \le C r^{\xi Q} \sum_{n=0}^{\lfloor C^{1/\zeta} \rfloor} e^{\xi h_{re^{-n}}(0)-\xi Qn}  + \sum_{\lfloor C^{1/\zeta} \rfloor + 1}^\infty n^\zeta e^{\xi h_{re^{-n}}(0)-\xi Qn}.
		\end{equation} 
		We now have to bound the right-hand side by the integral in (\ref{eq:a2k}). The argument for this, using Gaussian tail bounds, is exactly the same as in Steps 2--3 of the proof of~\cite[Prop.~3.14]{dfgps}.\footnote{The proof there gives their result with an added factor $\psi(t) = o(t)$ in the exponent, which arises in the proof from the fact that exact scale invariance was not then known, but $\psi$ can be taken to be identically zero in light of the relation $\mathfrak{c}_r = r^{\xi Q}$.} We thus conclude that (\ref{eq:a2k}) holds with superpolynomially high probability. (Uniformity in $a$ follows by translation invariance for $\fd_h$ and the fact that the result, and in particular the probability (\ref{eq:a2k}), does not depend on the choice of additive constant for $h$.)
	\end{proof}
	\begin{figure}
		\includegraphics[width=0.9\textwidth]{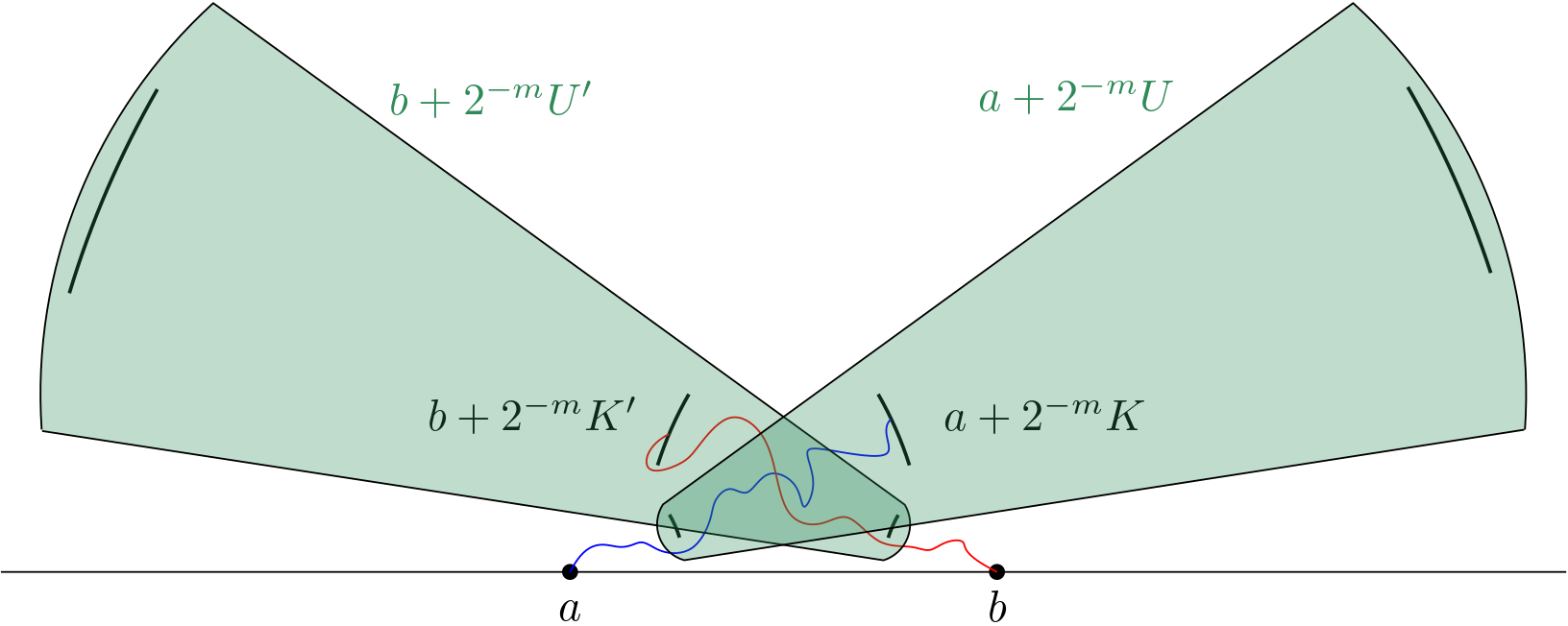}
		\caption{\label{fig:a-b-k}An illustration of the arcs $K$, $K'$ and their neighbourhoods $U$, $U'$ from the proof of Prop.~\ref{prop:dint}.}
	\end{figure}
	Given points $a$, $b$ with $n$ maximal such that $|b-a|<2^{1-n}$, arcs $K$, $K'$. and open sets $U$, $U'$ at positive distance to $\partial\h$ and such that $eK\cup K\cup e^{-1}K\subset U$, $eK'\cup K'\cup e^{-1}K'\subset U'$, we can apply Prop.~\ref{prop:dub} to find paths from $a$ to $a+2^{-n}K$ and from $b$~to~$b+2^{-n}K'$ that respectively stay in $a+\bigcup_{m\ge n}e^{-m}U$, $a+\bigcup_{m\ge n} e^{-m}U'$ and whose lengths are respectively bounded by$$C \int_{\log{(|b-a|^{-1})}}^\infty e^{\xi(h_{e^{-t}}(a)-Qt )} \, dt, \quad C \int_{\log{(|b-a|^{-1})}}^\infty e^{\xi(h_{e^{-t}}(b)-Qt )} \, dt ,$$with superpolynomially high probability as $C\to\infty$. By judiciously choosing $K,K'$ and the open sets $U,U'$ (see Figure~\ref{fig:a-b-k}), we can arrange that the path from $a$ to $a+2^{-n}K$ and that from~$b$~to~$b+2^{-n}K'$ cross each other, provided $2^{1-n} > |b-a|$, giving an analogue of part of~\cite[Prop.~3.15]{dfgps}:
	\begin{prop}\label{prop:dint}
		For $h$ a free-boundary GFF on $\h$ and $a,b \in \R$ with $ 0 <|b-a| \le 1$, we have
		\begin{equation}\label{eq:dint}
			\fd_h(a,b) \le C \int_{\log{(|b-a|^{-1})}}^\infty \left[e^{\xi(h_{e^{-t}}(a)-Qt )} + e^{\xi(h_{e^{-t}}(b)-Qt )} \right] \, dt 
		\end{equation}
		with superpolynomially high probability as $C \to \infty$.
	\end{prop}
	\begin{remark}\label{rem:dint}
		Since the choices of $K$, $K'$, $U$, $U'$ depend only on $2^{1-n}/|b-a|$, if we assume that $|b-a| = 2^{-n}$ for some $n \in \mathbb{N}$ then we get that the rate at which the probability decays is uniform in the choice of $a$ and $b$ (and $n$).
	\end{remark}
	We now want to use Prop.~\ref{prop:dint} to find sequences of points in $\h$ that converge to a point in~$\R$ w.r.t.\ both $\fd_h$ and the Euclidean metric; eventually we will use these to show that, almost surely, both metrics induce the same topology.
	\begin{lemma}\label{lemma:h2r}
		Let $h$ be a free-boundary GFF on $\h$ and fix $a\in\R$. Almost surely, for every closed arc $K \subset \h\cap \partial B(0,1))$ consisting of more than one point, we have 
		\begin{equation*}
			\fd_h(a,a+2^{-n}K) \to 0
		\end{equation*}
		as $n\to\infty$. Thus, this convergence almost surely holds simultaneously for all such $K$ and all dyadic rationals $a$.
	\end{lemma}
	\begin{proof}
		Since every such $K$ contains an arc with endpoints at rational angles, we can assume $K$ is fixed. Fix also $\ve \in (0,Q-2)$. Applying Prop.~\ref{prop:dub} and the Borel--Cantelli lemma for all $n\in\mathbb{N}$ with $r=2^{-n}$ and $C=2^{\xi\ve n}$, it almost surely holds that for $n$ large enough, we have
		\begin{equation*}
			\fd_h(a,a+2^{-n}K) \le 2^{\xi\ve n} \int_{n\log 2}^\infty e^{\xi (h_{e^{-t}}(a)-Qt)} \, dt \le \int_{n\log 2}^\infty e^{\xi (h_{e^{-t}}(a)-(Q-\ve)t)} \, dt.
		\end{equation*}
		Moreover, for $u \in (0,Q-\ve-2)$, by (\ref{eq:uthicc}) it almost surely holds that the integrand of the rightmost integral is bounded by $e^{-ut}$ for $n$ large enough, so the rightmost integral almost surely tends to 0 as $n \to \infty$, as required.
	\end{proof}
\subsection{Local H\"{o}lder continuity w.r.t.\ the Euclidean metric}
	We now prove that $\fd_h$ is almost surely locally H\"{o}lder continuous w.r.t.\ the Euclidean metric on $\overline{\h}$.
	\begin{prop}\label{prop:holder}
	Let $h$ be a free-boundary GFF on $\h$ with some choice of additive constant. Almost surely, for each $u\in(0,\xi(Q-2))$ and each compact $K \subset \overline{\h}$ there exists $C>0$ finite such that whenever $z, w\in K$, we have
	\begin{equation}\label{eq:holder}
		\fd_h(z,w) \le C|z-w|^{\xi(Q-2)-u}.
	\end{equation}
\end{prop} 
\begin{proof}
	By scale invariance it suffices to consider $K = [0,1]^2$. We will use the domain Markov property to couple $h$ with a zero-boundary field $\mathring{h}$ and a harmonic correction $\widetilde{\mathfrak{h}} := h-\mathring{h}$ given respectively by the projections of $h$ onto the spaces of functions, respectively, supported in and harmonic on $\h$. Unlike in the proof of Prop.~\ref{prop:dub}, this time we will need $\mathring{h}$ and $\widetilde{\mathfrak{h}}$ to be independent, so we will fix the additive constant for $h$ so that $\widetilde{\mathfrak{h}}(i)=0$. (Note that the claimed result does not depend on the choice of normalization.) Fix $S$ to be the rectangle $[0,1]\times[\frac{1}{2},1]$, and $U$ to be a neighbourhood of $S$ at positive distance from $\partial\h$. By~\cite[Prop.~3.9]{dfgps} we know that, for $h^{\mathrm{wp}}$ a whole-plane GFF with the additive constant fixed such that $h^{\mathrm{wp}}_1(0)=0$, and for any $p<4d_\gamma/\gamma^2$,
	\begin{equation}\label{eq:withca}
		\mathbb{E}[( \sup_{z,w\in S} \fd_{h^{\mathrm{wp}}} (z,w;U) )^p] = \mathbb{E}[( e^{-\xi h^{\mathrm{wp}}_1(0)} \sup_{z,w\in S} \fd_{h^{\mathrm{wp}}} (z,w;U) )^p] < \infty.
	\end{equation}
	Now consider the coupling of $h^{\mathrm{wp}}$ with $\mathring{h}$ and $\widehat{\mathfrak{h}}$ from Prop.~\ref{prop:dub} (we will not need independence here). Using that $\sup_S \widehat{\mathfrak{h}}$ has a Gaussian tail by Borell--TIS and thus $\mathbb{E}[\sup_S e^{q\xi\widehat{\mathfrak{h}}}] < \infty$ for all $q>0$, we get that
	\begin{equation}\label{eq:sansca}
		\mathbb{E}[( \sup_{z,w\in S} \fd_{\mathring{h}} (z,w;U) )^p] < \infty
	\end{equation}
	for each $p < 4d_\gamma/\gamma^2$ (by applying (\ref{eq:withca}) for a slightly larger value of $p$ and using H\"{o}lder's inequality). Now define for each $n \in \mathbb{N}$, $1 \le j \le 2^{n+1}-1$, $1\le k \le 2^n$ the sets$$S_{n,j,k} = 2^{-n}\left(S+k-1+\frac{1}{2}(j-1)i\right) = [(k-1)2^{-n},k2^{-n}] \times [j2^{-(n+1)},(j+1)2^{-(n+1)}],$$$$U_{n,j,k} = 2^{-n}\left(U+k-1+\frac{1}{2}(j-1)i\right).$$Note that for each $n$ the $S_{n,j,k}$ divide the rectangle $[0,1]\times [2^{-(n+1)},1]$ into rectangles of dimensions $2^{-n} \times 2^{-(n+1)}$. Then by scaling and translation invariance, for each $p<4d_\gamma/\gamma^2$, there exists $M_p < \infty$ such that
	\begin{equation}\label{eq:sansca2}
		\mathbb{E}[( 2^{n \xi Q} \sup_{z,w\in S_{n,j,k}} \fd_{\mathring{h}} (z,w;U_{n,j,k}) )^p] \le M_p,
	\end{equation}
	Moreover, we know that
	\begin{equation*}
		\sup_{z,w\in S_{n,j,k}} \fd_h (z,w;U_{n,j,k}) \le \sup_{z,w\in S_{n,j,k}} \fd_{\mathring{h}} (z,w;U_{n,j,k}) \cdot \sup_{z \in U_{n,j,k}} e^{\xi\widetilde{\mathfrak{h}}(z)}.
	\end{equation*}
	Since $\mathring{h}$ and $\widetilde{\mathfrak{h}}$ are independent, we have
	\begin{equation}\label{eq:holdere}
		\mathbb{E}\,\left[\left(2^{n\xi Q}\sup_{z,w\in S_{n,j,k}} \fd_h (z,w;U_{n,j,k})\right)^p\right] \le \mathbb{E}\,\left[\left(2^{n\xi Q}\sup_{z,w\in S_{n,j,k}} \fd_{\mathring{h}} (z,w;U_{n,j,k})\right)^p\right] \mathbb{E}\,\left[\sup_{z \in U_{n,j,k}} e^{p\xi\widetilde{\mathfrak{h}}(z)}\right].
	\end{equation}
	Note now that $\sup_{z \in U} \xi(\widetilde{\mathfrak{h}}(z) = \sup_{z \in U} \xi(\widetilde{\mathfrak{h}}(z)-\widetilde{\mathfrak{h}}(i))$ has a Gaussian tail by Borell--TIS, since $U$ is at positive distance from $\partial \h$. By scaling and translation invariance, we can conclude that there are $\sigma$, $c$ for which, for all $t > 0$ and all $n$, $j$,~$k$,
	\begin{equation*}
		\p[\sup_{z \in U_{n,j,k}} \xi(\widetilde{\mathfrak{h}}(z)-\widetilde{\mathfrak{h}}(2^{-n}(k-1+\frac{1}{2}(j+1)i))) \ge t] \le ce^{-t^2/2\sigma^2}.
	\end{equation*}
	This means that for every $p > 0$, there is a constant $K_p < \infty$ such that, for all $n$, $j$, $k$,
	\begin{equation}\label{eq:harmdifftis}
		\mathbb{E} \left[ \sup_{z \in U_{n,j,k}} e^{p\xi(\widetilde{\mathfrak{h}}(z)-\widetilde{\mathfrak{h}}(2^{-n}(k-1+\frac{1}{2}(j+1)i)))} \right] \le K_p.
	\end{equation}
	Note that $\widetilde{\mathfrak{h}}(2^{-n}(k-1+\frac{1}{2}(j+1)i))$ is a Gaussian variable; we proceed to bound its variance. Denoting by~$\widetilde{\mathfrak{h}}_r^\mathrm{circ}(z)$ the average of $\widetilde{\mathfrak{h}}$ on the circle $\partial B(z,r)$, we can write $\widetilde{\mathfrak{h}}(2^{-n}(k-1+\frac{1}{2}(j+1)i)) = \widetilde{\mathfrak{h}}_{2^{-n}}^\mathrm{circ}(2^{-n}(k-1+\frac{1}{2}(j+1)i))$ by harmonicity, using that $2^{-n}(k-1+\frac{1}{2}(j+1)i)$ is at distance at least $2^{-n}$ from the boundary. We now use that~$\widetilde{\mathfrak{h}}$ is an orthogonal projection of $h$, so that there is a constant $c > 0$ not depending on~$n$,~$j$,~$k$ such that (by Lemma~\ref{lem:var})
	\begin{align}\label{eq:harmvar}
		\nonumber\mathrm{var} \, \widetilde{\mathfrak{h}}(2^{-n}(k-1+\frac{1}{2}(j+1)i)) &= \mathrm{var} \, [\widetilde{\mathfrak{h}}_{2^{-n}}^\mathrm{circ}(2^{-n}(k-1+\frac{1}{2}(j+1)i)) - \widetilde{\mathfrak{h}}_1^\mathrm{circ} (i)] \\
		&\le \mathrm{var} \, [h_{2^{-n}}^\mathrm{circ}(2^{-n}(k-1+\frac{1}{2}(j+1)i)) - h_1^\mathrm{circ} (i)] \le n \log{2} + c.
	\end{align}
	(We could also compute this variance exactly using a similar argument to~\cite[Lemma~2.9]{GMS_multifractal}.)
	Now fix $q' > q > 0$; we compute
	\begin{align*}
		\nonumber\mathbb{E} \left[\sup_{z \in U_{n,j,k}} e^{q\xi\mathfrak{h}(z)}\right]
		&\le (K_{\frac{qq'}{q'-q}})^\frac{q'-q}{q'} \left(\mathbb{E}\,\left[ e^{\xi q' \mathfrak{h}(2^{-n}(k-1+\frac{1}{2}(j+1)i)) }\right]\right)^{q/q'} &\text{((\ref{eq:harmdifftis}) and H\"{o}lder)} \\
		&\le (K_{\frac{qq'}{q'-q}})^\frac{q'-q}{q'} e^{\xi^2 qq' (n \log{2} + c)/2} &\text{((\ref{eq:harmvar}) and Gaussian m.g.f.)} \\
		&=  C(q,q') 2^{n\xi^2qq'/2}. &
	\end{align*}
	Using these bounds (with $\delta>0$, $q=p$, $q' = p+\delta/p$) in (\ref{eq:holdere}), we have for $p < 4d_\gamma/\gamma^2$ that
	\begin{equation*}
		\mathbb{E}\,\left[\left(2^{n\xi Q}\sup_{z,w\in S_{nj,,k}} \fd_h (z,w;U_{n,k})\right)^p \right]\le \widetilde{C}(p,\delta) 2^{n\xi^2(p^2+\delta)/2} 
	\end{equation*}
	and thus, for any $s\in\R$,
	\begin{equation}\label{eq:holderp}
		\p\,\left[2^{n\xi Q}\sup_{z,w\in S_{n,j,k}} \fd_h (z,w;U_{n,j,k}) \ge 2^{ns} \right] \le 2^{-nps} \widetilde{C}(p,\delta) 2^{n\xi^2(p^2+\delta)/2}.
	\end{equation}
	If $s > 2\xi$ is sufficiently close to $2\xi$, setting $p = s/\xi^2$ we have $p < 4d_\gamma/\gamma^2$ since $\gamma > 2$ implies $2/\xi > 4/(\xi\gamma) = 4d_\gamma/\gamma^2$, so the right-hand side becomes $\widetilde{C}(p,\delta)2^{n(-s^2/(2\xi^2)+\xi^2\delta/2)}$, and if $\delta$ is sufficiently small we have $-s^2/(2\xi^2)+\xi^2\delta/2<-2$. Thus, using Borel--Cantelli and setting $u = s-2\xi$, we conclude that, almost surely, for every $u > 0$, there exists $n_0$ such that for all $n\ge n_0$, $1\le j \le 2^{n+1}-1$ and $1 \le k \le 2^n$, we have 
	\begin{equation}\label{eq:snk}
		\sup_{z,w\in S_{n,j,k}} \fd_h (z,w;U_{n,j,k}) < 2^{-n(\xi (Q-2)-u)}. 
	\end{equation}
	Fixing $u>0$ and taking $n$ as above, if $m \ge n_0$, $b \in (0,2^{-m}]$, $a \in [0,1]$, we can concatenate near-minimal paths connecting $a+2^{-n}i$ to $a+2^{-(n+1)i}$ in $U_{n,1,\lceil 2^na \rceil}$ for each $n \ge m+1$ with a near-minimal path connecting $a+2^{-(m+1)}i$ to $a+bi$ in $U_{m,1,\lceil 2^ma \rceil}$, to find that
	\begin{equation}\label{eq:hldrabi}
		\fd_h(a,a+bi) \le 2^{(1-m)(\xi(Q-2)-u)}.
	\end{equation}
	Indeed, by the same token it follows that, whenever $a\in[0,1]$ and $b \le 2^{-n_0}$, the $\fd_h$-diameter of the vertical line segment $[a,a+bi]$ from $a$ to $a+bi$ satisfies
	\begin{equation}\label{eq:vertdiam}
		\mathrm{diam}([a,a+bi];\fd_h) \le (4b)^{\xi(Q-2)-u}.
	\end{equation}
	Now consider general $w$, $z$ in $K$. If $|w-z| \le 2\min\{ \mathrm{Im}\, w, \mathrm{Im}\, z \}$, then with $n$ such that $2^{-n}\le |w-z| < 2^{-(n-1)}$, we have $\min\{ \mathrm{Im}\, w, \mathrm{Im}\, z \} \ge 2^{-(n+1)}$, so we can find $j_1$, $j_2$, $k_1$, $k_2$ such that $z \in S_{n,j_1,k_1}$ and $w \in S_{n,j_2,k_2}$. Moreover, $|w-z| < 2^{-(n-1)}$ implies that $|j_1-j_2| \le 5$ and $|k_1-k_2| \le 3$, which means that if $n\ge n_0$, then by applying (\ref{eq:snk}) to a set of rectangles of the form $S_{n,j,k}$ connecting $S_{n,j_1,k_1}$ and $S_{n,j_2,k_2}$ we find that
	\begin{equation}\label{eq:7}
		\fd_h(z,w) \le 7 \cdot  2^{-n(\xi (Q-2)-u)} \le 7 |w-z|^{\xi(Q-2)-u}.
	\end{equation}
	On the other hand, suppose $|w-z| > 2\min\{ \mathrm{Im}\, w, \mathrm{Im}\, z \}$. Since it also holds that $2|w-z| > 2(\max\{ \mathrm{Im}\, w, \mathrm{Im}\, z \} - \min\{ \mathrm{Im}\, w, \mathrm{Im}\, z \})$, by adding these inequalities we find $\max\{ \mathrm{Im}\, w, \mathrm{Im}\, z \} < \frac{3}{2}|w-z|$. Moreover, with $n$ such that $2^{-n}\le |w-z| < 2^{-(n-1)}$, we can find $k_1$ and $k_2$ such that $\mathrm{Re}\, w + |w-z|i \in S_{n-1,1,k_1}$, $\mathrm{Re}\, z + |w-z|i \in S_{n-1,1,k_2}$ and (since $|\mathrm{Re}\, w - \mathrm{Re}\, z| \le |w-z| < 2^{-(n-1)}$) $|k_1-k_2| \le 1$. Thus
	\begin{align*}
		\fd_h(w,z) &\le \fd_h(w,\mathrm{Re}\, w + |w-z|i) + \fd_h(\mathrm{Re}\, w + |w-z|i,\mathrm{Re}\, z + |w-z|i) + \fd_h(z,\mathrm{Re}\, z + |w-z|i) \\
		&\le \mathrm{diam}([\mathrm{Re}\, w, \mathrm{Re}\, w + \frac{3}{2}|w-z|i];\fd_h) + \mathrm{diam}(S_{n-1,1,k_1}; \fd_h(\cdot,\cdot; U_{n-1,1,k_1})) \\
		&+ \mathrm{diam}(S_{n-1,1,k_2}; \fd_h(\cdot,\cdot; U_{n-1,1,k_2})) + \mathrm{diam}([\mathrm{Re}\, z, \mathrm{Re}\, z + \frac{3}{2}|w-z|i];\fd_h).	
	\end{align*}
	Assuming now that $n-2\ge n_0$ we can use (\ref{eq:vertdiam}) to bound the first and fourth terms each by $(6|w-z|)^{\xi(Q-2)-u}$ and use (\ref{eq:snk}) to bound the second and third terms each by $(2|w-z|)^{\xi(Q-2)-u}$. Along with (\ref{eq:7}), we have just shown that (\ref{eq:holder}) holds with $C=\max\{7,2(2^{\xi(Q-2)-u} +6^{\xi(Q-2)-u}) \}$ provided $|w-z| < 2^{-(n_0+1)}$. Since $K$ clearly has finite $\fd_h$-diameter (e.g., combine (\ref{eq:vertdiam}) with $b=2^{-n_0}$ and (\ref{eq:snk}) for all $j$ and $k$ with $n=n_0$), the result for general $w$ and $z$ in $K$ follows by possibly increasing the constant $C$.
	
	Note that our exponent matches the one in~\cite[Prop.~3.18]{dfgps} for the zero-boundary GFF in the bulk, which is proved there to be optimal in the sense that $\fd_{h^\mathrm{wp}}$ is almost surely not locally $(\xi(Q-2)+u)$-H\"{o}lder continuous w.r.t.\ the Euclidean metric on any bounded open set for any $u>0$. Since $\mathfrak{h}$ is continuous away from the boundary it is easy to see that the same holds for $\fd_h$. We obtain the same optimal exponent here because, as we have already seen, for the free-boundary GFF the maximum thickness at the boundary is the same as that in the bulk.
\end{proof}
\begin{remark}\label{rem:avoid-r}
	Note that the above argument provides near-minimal paths that do not intersect~$\partial\h$ except possibly at their endpoints. In particular, it follows that with $h$, $u$, $K$, $C$ as in the statement of Prop.~\ref{prop:holder}, we have
	\begin{equation}\label{eq:avoid-r}
		\sup_{z,w\in K\cap\h} \inf\left\{ \mathrm{length}(P;\fd_h) \,\middle|\, P\colon z\rightsquigarrow w, P^{-1}(\partial\h) =\varnothing \right\} \le C (\mathrm{diam}\, K)^{\xi(Q-2)-u}.
	\end{equation}
\end{remark}
	\subsection{Positive definiteness}
The aim of this subsection is to establish positive definiteness for $\fd_h$ on $\overline{\h}$, completing the proof that it is a metric.
\begin{prop}\label{prop:posdef}
	If $h$ is a free-boundary GFF on $\h$ with some choice of additive constant, then the function $\fd_h$ is almost surely a metric on $\overline{\h}$; in particular it is almost surely positive definite.
\end{prop}
Note that, since finiteness follows for instance from Lemma~\ref{lemma:h2r}, we now have only to establish positive definiteness. Firstly we show positive definiteness at the boundary:
\begin{lemma}\label{lem:pdboundary}
	If $h$ is a free-boundary GFF on $\h$ with some choice of additive constant, then almost surely for all $a,b \in \partial\h$ we have $\fd_h(a,b) > 0$.
\end{lemma}
\begin{proof}
	We want to show that, almost surely, $\fd_h(a,b)>0$ whenever $a,b\in\partial\h$ are distinct. Firstly we can consider the analogous problem for the quantum wedge. Recall that part of~\cite[Thm~1.8]{zip} states that, if a $(\gamma-2/\gamma)$-quantum wedge $(\h,h,0,\infty)$ (equivalently, a wedge of weight~4) is decorated by an independent $\mathrm{SLE}_{\gamma^2}$  $\eta$ in $\h$ from 0 to $\infty$, then the surfaces parametrized by the left and right components $(W_1,h|_{W_1},0,\infty)$ and $(W_2,h|_{W_2},0,\infty)$ are independent $\gamma$-quantum wedges (equivalently, wedges of weight 2). 
	
	If we take any two distinct points on $\eta\setminus \{0\}$, we know that they are at positive $\fd_h$-distance w.r.t.~$h$ in $\h$, since they are away from the boundary $\partial \h$ (since $\eta$ does not hit $\partial\h \setminus \{0\}$ by~\cite[Thm~6.1]{rohde}). The distance w.r.t.\ $\fd_h(\cdot,\cdot;W_i)$ cannot be less than that w.r.t.~$h$, which means that any two distinct points on the right-hand (resp.\ left-hand) side of the boundary of $W_1$ (resp.~$W_2$) are at positive LQG distance w.r.t.~$h$ in $W_1$ (resp.~$W_2$). (By conformal covariance, this remains true regardless of the embedding of these wedges.) This suffices to establish positive definiteness of $\fd_h$ on $(0,\infty)$ for the $\gamma$-wedge. Note also that if we consider the canonical (circle-average) embedding of the wedge given by~$W_1$ into~$\h$, then fix a particular compact set $K \subset \overline{\h}$ not containing~0, we can find $L\subset \overline{\h}$ not containing~0 such that $K \subset \mathrm{int}\, L$ (i.e.\ the relative interior of $L$ within $\overline{\h}$); then almost surely the LQG metric distance between $K$ and $\partial L$ in $\overline{\h}$ is positive, so that the positive definiteness of the LQG metric on $K \cap \partial \h$ is determined by the field values inside $\mathrm{int}\, L$.
	
	Given the result for the $\gamma$-wedge, we can deduce it for the free-boundary GFF on $\h$ by an absolute continuity argument, using the radial--lateral decomposition. Indeed, recalling the construction of the circle average embedding from \S\ref{subsection:wedges}, when we parametrize by $\mathcal{S}$, the field average on $\{t\} \times [0,\pi]$ for the $\gamma$-wedge can be expressed, for a standard two-sided Brownian motion $B$ considered modulo vertical translation, as $B_{2(t+\tau)}+(Q-\gamma)(t+\tau)$ where $\tau$ is the last time at which $B_{2t}+(Q-\gamma)t$ hits 0. Translating horizontally by $\tau$ gives a field whose average on $\{t\} \times [0,\pi]$ is given by $B_{2t}+(Q-\gamma)t$, and whose lateral part (i.e., the part with mean zero on vertical line segments) has the same distribution as that of the wedge's lateral part (since it is independent of the radial part, with scale-invariant distribution). By conformal covariance, we again have positive definiteness for the LQG distance defined with respect to this field (when we map back to $\h$, on $(0,\infty)$).
	
	Finally the process $B_{2t}+(Q-\gamma)t$, with $t$ restricted to any compact subset of $\R$, is mutually absolutely continuous with the law of $B_{2t}$. So, considering the radial--lateral decomposition, the result for the wedge implies that for the free-boundary GFF (at least away from 0, but translation invariance then covers the case when $a$ or $b$ is 0).
\end{proof}
In the light of Lemma~\ref{lem:pdboundary} it is straightforward to complete the proof of Prop.~\ref{prop:posdef}.
\begin{proof}[Proof of Prop.~\ref{prop:posdef}]
	It remains only to rule out the possibility that there exist some $a\in\R$, $z\in\h$ at $\fd_h$-distance zero from each other. If we had $a\in\R$ and $\fd_h(a,z)=0$ for some $z\neq a$, then by Lemma~\ref{lem:pdboundary} we necessarily have $z\in\h$. Taking $z_n\to a$ w.r.t.\ the Euclidean metric, by Prop.~\ref{prop:holder} we also have $z_n\to a$ w.r.t.\ $\fd_h$. Thus $z_n\to z$ w.r.t.\ $\fd_h$, but (since $\mathrm{Im}\, z_n\to 0$) not w.r.t.\ the Euclidean metric. This contradicts the fact (noted in the discussion at the start of \S\ref{section:lqg}) that $\fd_h$ induces the Euclidean topology on $\h$, concluding the proof.
\end{proof}
	\section{Further H\"{o}lder continuity estimates for $\gamma$-LQG metrics}\label{section:bh}
	In this section we will show that, almost surely, $\fd_h$ is in fact locally \emph{bi}-H\"{o}lder continuous w.r.t.\ the Euclidean metric on~$\overline{\h}$ (so in particular $\fd_h$ induces the Euclidean topology on $\overline{\h}$). First, however, we will show that the $\fd_h$-distance between two points of $\partial\h$ is almost surely bi-H\"{o}lder continuous w.r.t.\ the $\nu_h$-measure of the interval between them. Throughout this subsection, $h$ will always be a free-boundary GFF on $\h$ with some choice of additive constant (in this subsection it will never be necessary to specify the choice).

	\subsection{Upper bound on distance in terms of boundary measure}
For what follows we will need to adapt the version of~\cite[Lemma~4.6]{ds} for the boundary measure $\nu_h$. 
\begin{lemma}\label{lem:46}
 There exist $c_1,c_2>0$ such that, for all $a \in \R$, $\ve > 0$ and $\eta \le 0$, we have
	\begin{equation*}
		\p [\nu_h([a,a+\ve]) < e^{\eta+\frac{\gamma}{2}\sup_{\delta \le \ve} (h_\delta(a)-Q\log{\frac{1}{\delta}})}  ] \le c_1e^{-c_2\eta^2}.
	\end{equation*}
\end{lemma}
\begin{proof}
	Note that this result does not depend on the choice of additive constant for $h$, since adding a constant $c$ scales both sides of the inequality by $e^{\gamma c/2}$. The original~\cite[Lemma~4.6]{ds} gives the almost sure lower tail bound
	\begin{equation*}
		\p\left[\nu_h([a-\ve,a+\ve]) < e^{\eta + \frac{\gamma}{2} (h_\ve(a)-Q\log{\frac{1}{\ve}})} \middle| (h_{\ve'}(a): \ve' \ge \ve) \right] \le C_1e^{-C_2\eta^2},
	\end{equation*}
	for all $\eta \le 0$, with $C_1, C_2 > 0$ deterministic constants independent of $a$, $\ve$, $\eta$. (Actually in \cite{ds} the statement and proof are given just for an analogue involving the area measure $\mu_h$ instead of the boundary measure $\nu_h$, but the proof is similar for the boundary measure---see~\cite[\S6.3]{ds}.) We will need a lower tail bound for the conditional law of $\nu_h([a,a+\ve])$ given $(h_{\ve'}(a): \ve' \ge \ve)$, rather than for that of $\nu_h([a-\ve,a+\ve])$, but this in fact follows from the proofs of~\cite[Lemmas~4.5--4.6]{ds}, which (when reformulated for the boundary measure) proceed by partitioning $[a-\ve,a+\ve]$ into $[a-\ve,a]$ and $[a,a+\ve]$ and thus actually obtain that, almost surely (possibly changing $C_1, C_2$):
	\begin{equation}\label{eq:lotail}
		\p\left[\nu_h([a,a+\ve]) < e^{\eta + \frac{\gamma}{2} (h_\ve(a)-Q\log{\frac{1}{\ve}})} \middle| (h_{\ve'}(a): \ve' \ge \ve) \right] \le C_1e^{-C_2\eta^2},
	\end{equation}
	Observe also that if $\delta < \ve$, since $[a,a+\delta] \subset [a,a+\ve]$ we have, almost surely,
	\begin{equation*}
		\p\left[\nu_h([a,a+\ve]) < e^{\eta + \frac{\gamma}{2} (h_\delta(a)-Q\log{\frac{1}{\delta}})} \middle| (h_{\ve'}(a): \ve' \ge \delta) \right] \le C_1e^{-C_2\eta^2}.
	\end{equation*}
	Moreover, if $T$ is a stopping time for the process $(h_{e^{-t}}(a))_t$ such that $e^{-T} \le \ve$, then, by continuity of this process, the usual argument considering the discrete stopping times $T_n := 2^{-n}\lceil 2^nT \rceil$ yields that
	\begin{equation*}
		\p\left[\nu_h([a,a+\ve]) < e^{\eta + \frac{\gamma}{2} (h_{e^{-T}}(a)-QT)} \middle| (h_{e^{-t}}(a): t\le T) \right] \le C_1e^{-C_2\eta^2}
	\end{equation*}
	almost surely on the event $\{T < \infty\}$. Note that, if we set $Y_t = h_{e^{-t}}(a)-Qt$, then $Z_{t,a} := \sup_{s\ge t} Y_s-Y_t$ is the maximum over all time of a standard Brownian motion with drift $-Q/2$, and is thus an $\mathrm{Exp}(Q)$ random variable. (This standard result can be obtained by using the Girsanov theorem to deduce the law of the maximum over the time interval $[0,T]$ from that for a Brownian motion without drift and then sending $T\to\infty$---see~\cite[Prop.~10.4]{privault}.) Thus $\p[Z_{t,a} \le \log{2}] =: q$ is positive and independent of $t$ and $a$. 
	Now, let $T^M := \inf\{t\ge \log\frac{1}{\ve}: Y_t \ge M \}$. Conditioning on $\{\sup_{t\ge\log\frac{1}{\ve}} Y_t \ge M\}$ (equivalently, on $\{T^M < \infty\}$), the above applied to the stopping time $T^M$ gives that
	\begin{equation*}
		\p\left[\nu_h([a,a+\ve]) < e^{\eta + \frac{\gamma}{2} M} \middle| \sup_{t\ge\log\frac{1}{\ve}} Y_t \ge M \right] \le C_1e^{-C_2\eta^2}.
	\end{equation*}
	But since we also have 
	\begin{equation*}
		\p\left[\sup_{t\ge\log\frac{1}{\ve}} Y_t \le M + \log{2} \middle| \sup_{t\ge\log\frac{1}{\ve}} Y_t \ge M\right] = \p\left[Z_{T^M,a} \le \log{2} \middle| T^M < \infty\right] = q
	\end{equation*} 
	(by the strong Markov property), by conditioning on which interval of the form $[(n-1)\log 2,n\log 2)$ contains $\sup_{t\ge\log\frac{1}{\ve}} Y_t$ we have
	\begin{equation*}
		\p\left[\nu_h([a,a+\ve]) < e^{\eta + \frac{\gamma}{2} (\sup_{t\ge\log\frac{1}{\ve}} Y_t - \log{2})}\right] \le C_1 q^{-1} e^{-C_2\eta^2},
	\end{equation*}
	from which the result follows.
\end{proof}
\begin{lemma}\label{lem:dnu}
	Fix $\alpha \in (0, 2/d_\gamma)$. Then there exist a constant $\ve > 0$ and a random integer $N$ such that whenever $m \in \mathbb{N}$, $m \ge N$, $a \in [0,1] \cap 2^{-m}\mathbb{Z}$, we have
	\begin{equation}\label{eq:dnu}
		\fd_h(a,a+2^{-m}) \le \nu_h([a,a+2^{-m}])^\alpha \cdot 2^{-m\ve}.
	\end{equation}
	Moreover, $\ve > 0$ can be chosen so that the minimal such $N$ satisfies $\p[N \ge n] = O(2^{-n\beta})$ for every $\beta \in (0,Q^2/4-1)$.
\end{lemma}
\begin{proof}
	Recall $d_\gamma = \gamma/\xi$. Fix $\zeta > \gamma/(2\xi)$. Then to prove the result for $\alpha = \zeta^{-1}$, it suffices to show that for each fixed $a$ and $m$, the complement of the event in (\ref{eq:dnu}) has probability $O(2^{-\lambda m})$ (uniformly for $a\in[0,1]\cap 2^{-m}\mathbb{Z}$) for some $\lambda > 1$ (then we can apply Borel--Cantelli to prove the result with $\beta=\lambda-1$). 
	
	Put $b = a+2^{-m}$. By Prop.~\ref{prop:dint} (and the subsequent Remark~\ref{rem:dint}), for any $\ve > 0$, it holds that, for every $n$, we have
	\begin{align*}
		\p[\nu_h([a,b]) \le |a-b|^{-\zeta\ve}\fd_h(a,b)^\zeta] \le \, & \p\left[\nu_h([a,b]) \le |a-b|^{-2\zeta\ve} \left(\int_{\log{(|a-b|^{-1})}}^\infty e^{\xi(h_{e^{-t}}(a)-Qt )} \, dt\right)^\zeta  \right] \\
		&+\p\left[\nu_h([a,b]) \le |a-b|^{-2\zeta\ve} \left(\int_{\log{(|a-b|^{-1})}}^\infty e^{\xi(h_{e^{-t}}(b)-Qt )} \, dt\right)^\zeta  \right] \\
		&+ O(|a-b|^n).
	\end{align*}
	Consider the first of the probabilities on the RHS. By~\cite[Lemma~A.5]{dms}, if we let the integral be $I$ and the supremum of its integrand be $M$, then for any $p>0$ there is some $c_p < \infty$ for which we have 
	\begin{equation*}
		\mathbb{E} [I^p|M] \le c_p M^p.
	\end{equation*} 
	(Actually~\cite[Lemma~A.5]{dms} is stated for an integral with lower limit 0 and a Brownian motion started from 0, so to obtain our statement we can, say, use the lemma to bound the expectation conditional on both $h_{|a-b|}(a)$ and $\sup_{u\ge 0} e^{\xi(h_{e^{-u}|a-b|}(a)-h_{|a-b|}(a)-Qu)}$, and use that these have finite moments of all positive orders.)
	Dividing and taking expectations we find that $I/M$ has finite moments of all positive orders, so that by Markov's inequality, for any $k > 0$ we have that~$I$ is bounded by $|a-b|^{-k}$ times $M$ except on an event of superpolynomially decaying probability (in $|a-b|^{-1}$). 
	Thus we have, for every $n$:
	\begin{align}\label{eq:poly}
		\nonumber&\p\left[\nu_h([a,b]) \le |a-b|^{-2\zeta\ve} \left(\int_{\log{(|a-b|^{-1})}}^\infty e^{\xi(h_{e^{-t}}(a)-Qt )} \, dt\right)^\zeta  \right] \\
		\nonumber\le \, &\p\left[\nu_h([a,b]) \le |a-b|^{-2\zeta\ve} \exp{\left(\zeta\xi\sup_{t\ge\log{(|a-b|^{-1})}}{(h_{e^{-t}}(a)-Qt)}\right)} \right] + O(|a-b|^n) \\
		\nonumber\le \, &\p\left[ \nu_h([a,b]) \le |a-b|^{\zeta\ve} \exp{\left(\frac{\gamma}{2}\sup_{t\ge\log{(|a-b|^{-1})}}{(h_{e^{-t}}(a)-Qt)}\right)}  \right] + O(|a-b|^n) \\
		&+ \p\left[ \zeta\xi\sup_{t\ge\log{(|a-b|^{-1})}}{(h_{e^{-t}}(a)-Qt)} \ge 3\zeta\ve\log{|a-b|}+ \frac{\gamma}{2}\sup_{t\ge\log{(|a-b|^{-1})}}{(h_{e^{-t}}(a)-Qt)} \right] .
	\end{align}
	Note that the first probability on the right-hand side of (\ref{eq:poly}) decays superpolynomially in $|a-b|$ by Lemma~\ref{lem:46}, whereas the last probability is equal to
	\begin{equation}\label{eq:polybound}
		\p\left[ \sup_{t\ge\log{(|a-b|^{-1})}}{(h_{e^{-t}}(a)-Qt)} \ge \frac{3\zeta\ve}{\zeta\xi-\gamma/2} \log|a-b| \right].
	\end{equation}
	It thus suffices to show that the union of the event in (\ref{eq:polybound}) over all $m\ge n$, $a\in [0,1]\cap 2^{-m}\mathbb{Z}$, $b=a+2^{-m}$ has probability $O(2^{-n\beta})$ for some $\beta > 0$. 
	Setting $\delta = \frac{3\zeta\ve}{\zeta\xi-\gamma/2}$, this union is contained in the event
	\begin{equation*}
		\left\{	\sup_{r\le 2^{-n}} \sup_{x\in[0,1]} \frac{h_r(x)}{\log{(1/r)}} \ge Q-\delta \right\},
	\end{equation*}
	which by (\ref{eq:uthicc}) has probability $O(2^{-n( (Q-\delta)^2/4-1-u )})$ for every $u>0$. Since $\ve$ and therefore $\delta$ can be made arbitrarily small, this completes the proof for arbitrary $\beta \in (0, Q^2/4-1)$ (note that the result is non-trivial since $Q>2$).
\end{proof}
\begin{prop}\label{prop:dnu}
	Fix $\alpha \in (0, 2/d_\gamma)$. Then almost surely there exists $C' \in (0,\infty)$ such that, for all $a$, $b$ in $[0,1]$, we have
	\begin{equation*}
		\fd_h(a,b) \le C' \nu_h([a,b])^\alpha .
	\end{equation*}
	Moreover $\p[C'\ge x]$ decays at worst polynomially in $x$.
\end{prop}
\begin{proof}
	With $N$ as in Lemma~\ref{lem:dnu}, note that for $m<N$, for $a\in[0,1]\cap 2^{-m}\mathbb{Z}$, we have
	\begin{align*}
		\fd_h(a,a+2^{-m}) \le \sum_{i=1}^{2^{N-m}} \fd_h(a+(i-1)2^{-N},a+i2^{-N}) &\le 2^{-N\ve} \sum_{i=1}^{2^{N-m}} \nu_h([a+(i-1)2^{-N},a+i2^{-N}])^\alpha \\
		&\le 2^{-N\ve} 2^{(N-m)(1-\alpha)} \nu_h([a,a+2^{-m}])^\alpha,
	\end{align*}
	using the power mean inequality (since $\alpha < 2/d_\gamma < 1$).
	It follows that there is a random constant~$C$ such that whenever $m \in \mathbb{N}$, $a \in [0,1] \cap 2^{-m}\mathbb{Z}$, we have
	\begin{equation*}
		\fd_h(a,a+2^{-m}) \le C \nu_h([a,a+2^{-m}])^\alpha \cdot 2^{-m\ve},
	\end{equation*}
	and that, assuming $\ve<1-\alpha$,
	\begin{equation*}
		\p[C \ge 2^{n(1-\alpha-\ve)}] \le \p[N \ge n] = O(2^{-n\beta}),
	\end{equation*}
	i.e., $C$ has polynomial decay. We now argue as in the proof of the Kolmogorov criterion: if $a,b$ are dyadic rationals in $[0,1]$ we can partition $[a,b]$ as $a = a_0 < a_1 < \cdots < a_l = b$ where $[a_{i-1},a_i] = [2^{-m_i}n_i, 2^{-m_i}(n_i+1)] $ for non-negative integer $m_i$, $n_i$ such that no three of the $m_i$ are equal. Then we have
	\begin{equation}\label{eq:kc}
		\frac{\fd_h(a,b)}{\nu_h([a,b])^\alpha} \le \sum_{i=1}^l \frac{\fd_h(a_{i-1},a_i)}{\nu_h([a_{i-1},a_i])^\alpha} \le 2C \sum_{n=0}^\infty 2^{-n\ve} =: C'.
	\end{equation}
	The same argument works for $a,b$ arbitrary, using a countably infinite partition $a = a_0 < a_1 < a_2 < \cdots$ with $a_n\uparrow b$. In order to obtain the analogue of the first inequality in (\ref{eq:kc}) we have to verify that
	\begin{equation}\label{eq:infinitesum}
		\fd_h(a,b) \le \sum_{i=1}^\infty \fd_h(a_{i-1},a_i).
	\end{equation} 
	This follows by using the triangle inequality to obtain
		\begin{equation*}
		\fd_h(a,b) \le \sum_{i=1}^n \fd_h(a_{i-1},a_i) + \fd_h(a_n,b)
	\end{equation*} 
	and noting that $\fd_h(a_n,b)\to 0$ by Prop.~\ref{prop:holder}.
\end{proof} 
\subsection{Local reverse H\"{o}lder continuity}

		For the metric gluing proof we will need a H\"{o}lder exponent for the Euclidean metric w.r.t.~$\fd_h$. We know this exists away from $\partial \h$, since this is true for the whole-plane GFF by the results of~\cite{dfgps}, and the harmonic correction is well-behaved away from the boundary. However, we can obtain H\"{o}lder continuity even at the boundary. We can establish such a reverse H\"{o}lder inequality by an argument based on~\cite[Lemma~3.22]{dfgps}. First we will need the following lemma:
	\begin{lemma}\label{lem:bounded}
		Let $\alpha < Q$, and let $h$ be either the circle average embedding into $\h$ of an $\alpha$-quantum wedge, or equal to $h^F-\alpha\log|\cdot|$ where $h^F$ is a free-boundary GFF on $\h$ with $h^F_1(0)=0$. Then, almost surely, every $\fd_h$-bounded subset of~$\overline{\h}$ is also Euclidean-bounded.
	\end{lemma}
	\begin{proof}
	First observe that since $\fd_h$ is locally bounded (indeed, locally H\"{o}lder continuous) w.r.t.\ the Euclidean topology, as $a\downarrow 0$,$$\p[\fd_h(\partial B(0,r) \cap \overline{\h}, \partial B(0,2r)\cap \overline{\h})  \ge a r^{\xi Q} e^{\xi h_r(0)} ] \to 1,$$and that by conformal covariance this probability does not depend on $r$. Since in (\ref{eq:gffdef}), $\alpha_n$ and $\sigma(h|_{\h\setminus B(0,r)})$ are independent if $\mathrm{supp}\, \varphi_n \subset B(0,r)$, the tail $\sigma$-algebra $\bigcap_{r>0}\sigma(h|_{\h  \setminus B(0,r)})$ is trivial, and from this and the fact that $\fd_h$ is locally determined by $h$ it follows that for $a$ large enough there are almost surely infinitely many $k\in\mathbb{N}$ for which$$\fd_h(\partial B(0,2^k) \cap \overline{\h}, \partial B(0,2^{k+1})\cap \overline{\h})  \ge a 2^{k\xi Q} e^{\xi h_{2^k}(0)}.$$Since $e^{t\xi Q} e^{\xi h_{e^t}(0)} \to \infty$ as $r\to\infty$ (for $t\ge0$ it has the law of the exponential of a Brownian motion with positive drift $Q-\alpha$, albeit started at the last time this process hits 0 in the case of a wedge), it follows that for any compact $K\subset\overline{\h}$, $\fd_h(K,\partial B(0,r)) \to \infty$ as $r\to\infty$. 
	\end{proof}
	\begin{prop}\label{prop:bihldrabi}
		Let $h$ be a free-boundary GFF on $\h$ with some choice of additive constant. Almost surely, for each $u > 0$	and $K \subset \overline{\h}$ compact there exists $C\in(0,\infty)$ such that whenever $a$, $b \in \R$ with $a$, $a+bi \in K$, we have
		\begin{equation*}
			\fd_h(a,a+bi) \ge C^{-1} b^{\xi(Q+2)+u}.
		\end{equation*}
	\end{prop}
	\begin{proof}
		Since we know that $\mathrm{diam}\, (K;\fd_h)$ is finite, by Lemma~\ref{lem:bounded} we can find a (random) $K' \supseteq K$ compact so that $\fd_h(K,\partial K' \setminus \partial \h) > \mathrm{diam}\, (K;\fd_h)$, so that $\fd_h(z,w) = \fd_h(z,w;\mathrm{int}\, K')$ for each $z,w \in K$. Thus it suffices to prove the result for $\fd_h(a,a+bi;U)$ and each $a,a+bi\in U$ for each bounded~$U$ open in~$\overline{\h}$. 
		
		\begin{figure}\label{fig:u-tilde}
			\includegraphics[width=0.5\textwidth]{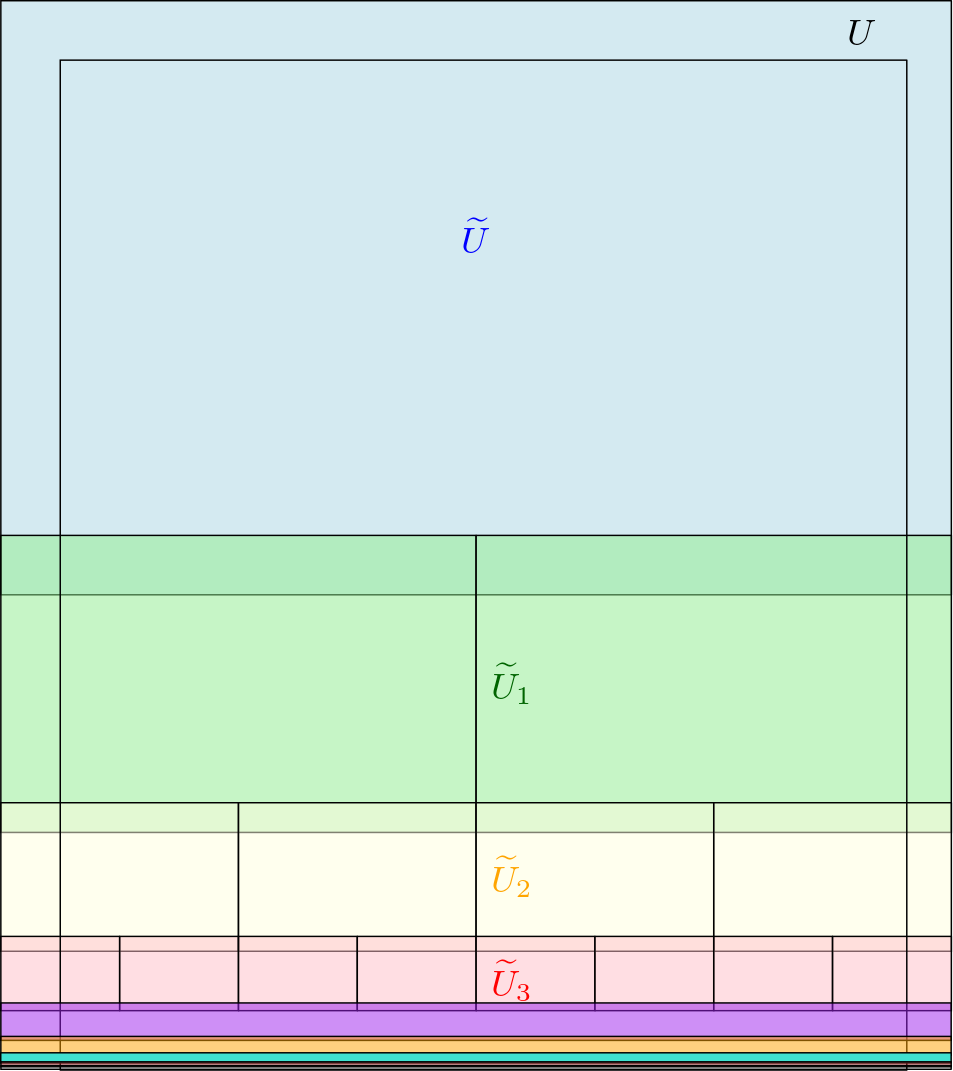}
			\caption{The sets $\widetilde{U}$, $\widetilde{U}_n$ used in the proof of Prop.~\ref{prop:bihldrabi}.}
		\end{figure}
		Choose $U$ to be an axis-parallel rectangle containing 0 with dyadic rational vertices. Let $\widetilde{U}$ be another such rectangle containing $U \cap (\h+i)$ so that the lower vertices of $\widetilde{U}$ have imaginary part greater than $1/2$ and the upper vertices have imaginary part greater than $2$, so that there exists~$\widetilde{U}_n$ a union of $2^n$ horizontal translates of $2^{-n}\widetilde{U}$ covering$$(U \cap (\h+2^{-n}i)) \setminus (\h+2^{-(n-1)}i)).$$Define $h^\mathrm{wp}$ and $\mathfrak{h}$ by coupling with $h$ as in Prop.~\ref{prop:dub}. By~\cite[Lemma~3.22]{dfgps} (and the Borel--Cantelli lemma applied to $\ve = 2^{-n}$) we know that the Euclidean metric on $K$ is $(\chi')^{-1}$-H\"{o}lder continuous w.r.t.\ $\fd_{h^{\mathrm{wp}}}$ for each $\chi' > \xi(Q+2)$. By Borell--TIS and a union bound over the $2^n$ translates of $2^{-n}\widetilde{U}$, we find that there are $c_1$, $c_2$ for which, for each $t\ge 0$,
		\begin{equation*}
			\p [\inf_{z\in\widetilde{U}_n} e^{\xi (\mathfrak{h}(z)-\mathfrak{h}(2^{-n}i)) } \le e^{-t}] \le 2^n c_1 e^{-c_2t^2}.
		\end{equation*}	
		Note that setting $t = n\ve \log{2}$ makes this summable. As before, we can get a similar tail bound for $\mathfrak{h}(2^{-n}i)-\mathfrak{h}(i)$; combining all of these we find that (with $\chi = \chi'-2\ve$) we get that for every $\chi > \xi(Q+2)$ there is almost surely a finite constant $C>0$ for which we have
		\begin{equation}\label{eq:dyadicholder}
			\fd_h(\R+2^{-(n-1)}i, \R+2^{-n}i;U) \ge C 2^{-n\chi}
		\end{equation}
		for all $n$, from which the result follows.
	\end{proof}
 We will now need estimates for LQG areas of Euclidean balls, deduced from the results of~\cite{rv}:
	\begin{prop}\label{prop:volume_moment}
		Fix $K\subset \h$ compact; let $h$ be a free-boundary GFF on $\h$ with the constant fixed such that $h_1(0)=0$. Then whenever $\zeta_1 > \gamma(Q+2)>\gamma(Q-2)> \zeta_2 > 0$ there almost surely exists a random $\ve_0 > 0$, such that, for all $\ve \in (0,\ve_0)$ and all $z \in K$, if $B(z,\ve)$ is the Euclidean ball of radius $\ve$ around $z$ then
		\begin{equation*}
			\ve^{\zeta_1} \le \mu_h(B(z,\ve) \cap \overline{\h}) \le \ve^{\zeta_2}.
		\end{equation*} 
		Moreover, define $$s_+=\frac{4\gamma^2+2\sqrt{2}\gamma\sqrt{(2+\gamma^2)(8+\gamma^2)}}{(4+\gamma^2)^2}.$$ If the condition on $K$ is weakened to $K\subset \overline{\h}$ compact (so that $K$ is allowed to intersect $\partial\h$), then provided $\zeta_2 \in (0,\gamma(Q-2)(1-s_+))$ there still almost surely exists a random $\ve_0 > 0$ such that the upper bound holds for all $z\in K$, $\ve \in (0,\ve_0)$.
		
	\end{prop}
	\begin{proof}	
		We will first assume that $K$ is at positive distance from $\partial{\h}$, using the fact~\cite[Prop.~3.5, Prop.~3.6]{rv} that the $\mu_{h^{\mathrm{wp}}}$-area of a fixed Euclidean ball has finite $p^\mathrm{th}$ moments for all $p \in (-\infty,4/\gamma^2)$. Fix the additive constant for $h^{\mathrm{wp}}$ so that $h^{\mathrm{wp}}_1(0)=0$, and for each $p \in (-\infty,4/\gamma^2)$ define $C_p := \mathbb{E}[\mu_{h^{\mathrm{wp}}} (B(0,1))^p]$. Then by (\ref{eq:reparam}), we have
		\begin{equation}
			\mu_{h^{\mathrm{wp}}} (B(0,\ve)) \overset{(d)}{=} e^{\gamma h^{\mathrm{wp}}_\ve (0)} \ve^{\gamma Q} \mu_{h^{\mathrm{wp}}} (B(0,1)).
		\end{equation}
		If $p>0$, fixing $q,q'>p$ with $1/q+1/q'=1/p$ and using that $h^{\mathrm{wp}}_\ve (0) \sim N(0,\log{(1/\ve)})$, we obtain by H\"{o}lder's inequality that
		\begin{equation}
			\mathbb{E}[\mu_{h^{\mathrm{wp}}} (B(0,\ve))^p] \le C_q^{p/q} \ve^{\gamma Qp-\gamma^2pq'/2}.
		\end{equation}
		Likewise, since the centred Gaussian variables $h^\mathrm{wp}_\ve(z)-h^\mathrm{wp}_1(0)$ have variance $\log(1/\ve)+O(1)$ uniformly in $\ve$ and $z\in K$, we obtain the same result up to a multiplicative constant for $\mu_{h^{\mathrm{wp}}} (B(z,\ve))$ and conclude by taking $q'$ sufficiently close to $p$ that, for each $p, u > 0$, we have$$\mathbb{E}[\mu_{h^{\mathrm{wp}}} (B(z,\ve))^p] \le C(p,u) \ve^{\gamma p(Q-\gamma p/2)-u}$$for each $\ve > 0$ and each $z\in K$, where $C(p,u)$ depends neither on $\ve$ nor on $z$. 
		
		Fixing a neighbourhood $U$ of $K$ still at positive distance from $\partial\h$, using that $\sup_U (h-h^\mathrm{wp})$ has a Gaussian tail, a further application of H\"{o}lder's inequality gives the same result for $h$ (except with a different $C(p,u)$ and only over choices of $\ve$ and $z$ such that $B(z,\ve)\subset U$). Setting $p=2/\gamma<4/\gamma^2$, we obtain the exponent $2Q-2-u$. Now we can cover $K$ by $O(\ve^{-2})$ balls of radius $\ve$ such that each ball $B_\ve$ satisfies$$\p[\mu_h(B_\ve) > \ve^{\zeta_2}] \le \ve^{-p\zeta_2} \mathbb{E}[\mu_h(B_\ve)^p] = O(\ve^{2Q-2-2\zeta_2/\gamma-u}).$$This exponent is greater than 2 whenever $\zeta_2 < \gamma(Q-2)$ and $u$ is chosen small enough, so applying Borel--Cantelli to covers $\mathcal{C}_n$ with $\ve=2^{-n}$ for each $n$, and noting that each ball of radius $\ve$ centred in $K$ and such that $B(z,2\ve)$ is contained in $U$ can be covered by an absolute constant number of balls in $\mathcal{C}_{\lfloor \log_2(1/\ve) \rfloor}$, we obtain the second inequality whenever $\zeta_2 < \gamma(Q-2)$. (Note that $Q-2> 0$.) 
		
		The same argument for $p=-2/\gamma$, considering $\inf_U (h-h^\mathrm{wp})$ instead, produces the bound$$\p[\mu_h(B_\ve) < \ve^{\zeta_1}] =O(\ve^{2\zeta_1/\gamma -2Q-2-u}),$$giving summability whenever $\zeta_1 > \gamma(Q+2)$ and hence the first inequality.
		
		We will again employ the result~\cite[Thm~1.8]{zip} that a $(\gamma-2/\gamma)$-quantum wedge (a wedge of weight 4) is cut by an independent $\mathrm{SLE}_{\gamma^2}$ (call it $\eta$) into two independent $\gamma$-quantum wedges (of weight 2). (It suffices to prove the result for a wedge, by mutual absolute continuity.) Parametrize the original wedge by $(\h,h,0,\infty)$, so that if $\eta$ is an independent $\mathrm{SLE}_{\gamma^2}$ from 0 to $\infty$ and $W^{-}$,~$W^{+}$ are respectively the left and right sides of $\eta$, then $(W^{-},h|_{W^{-}},0,\infty)$ and $(W^{+},h|_{W^{+}},0,\infty)$ are independent $\gamma$-quantum wedges.
		
		It follows from~\cite[Thm~5.2]{rohde} that the components of the complement of an $\mathrm{SLE}_\kappa$ in a smooth bounded domain for $\kappa \in (0,4)$ are almost surely H\"{o}lder domains. Therefore (e.g.\ by using a M\"{o}bius map to transfer from $\h$ to the unit disc) we have a conformal map $\varphi: \h \to W$ that is almost surely \emph{locally} H\"{o}lder continuous (away from a single point on $\partial \h$) for $W = W^{-}$ or $W^{+}$. Moreover~\cite[Cor.~1.8]{GMS_multifractal} gives that this holds with any exponent $\alpha < 1-s_+$ (and~\cite[Remark~1.2]{GMS_multifractal} we have $s_+ < 1$).
		So even if $K$ intersects $\partial\h$, if $0\notin K$ we can use $\varphi$ to map $K$ H\"{o}lder-continuously to a subset of $W$ away from $\partial\h$, then deduce the upper bound for $K$ from that for~$\varphi(K)$ and~$\mu_{\widetilde{h}}$, where $\widetilde{h} = h \circ \varphi^{-1} + Q \log{(|\varphi'|^{-1})}$.
	\end{proof}
	We now prove H\"{o}lder continuity for the Euclidean metric w.r.t.\ $\fd_h$. We begin by establishing this at the boundary:
	\begin{prop}\label{prop:bihldrr}
		Let $h$ be a free-boundary GFF on $\h$ with the additive constant fixed in some way. There exists $\widetilde{\beta} > 0$ such that the following holds. Almost surely, for every $u \in (0, \widetilde{\beta})$ and each fixed compact interval $I \subset \R$ there is a finite constant $C>0$ such that $|x-y| \le C\fd_h(x,y)^{(\widetilde{\beta}-u)}$ for all $x,y\in I$.
	\end{prop}
	In order to prove this we will begin by proving Prop.~\ref{prop:bottleneck}, showing that $\mathrm{SLE}_\kappa$ curves for $\kappa < 4$ cannot bottleneck too much.
	\begin{proof}[Proof of Prop.~\ref{prop:bottleneck}]
		Let $h$, $\eta$ and $W^{\pm}$ be as in the proof of Prop.~\ref{prop:volume_moment} (so $\eta$ cuts the wedge $(\h,h,0,\infty)$ of weight 4 into independent wedges of weight 2 parametrized by $W^{\pm}$). We will need the result of~\cite[\S4.2]{mmq} that for $\kappa < 8$, chordal $\mathrm{SLE}_\kappa$ curves in $\h$ from 0 to $\infty$ almost surely satisfy the following \emph{non-tracing hypothesis}: for any compact rectangle $K \subset \h$ and any $\widetilde{\alpha}>\widetilde{\xi}>1$ there exists $\delta_0>0$ such that for any $\delta \in (0,\delta_0)$, and any $t$ such that $\eta(t) \in K$, there exists a point~$y$ with the following properties:
		\begin{figure}\label{fig:bottleneck}
			\includegraphics[width=0.9\textwidth]{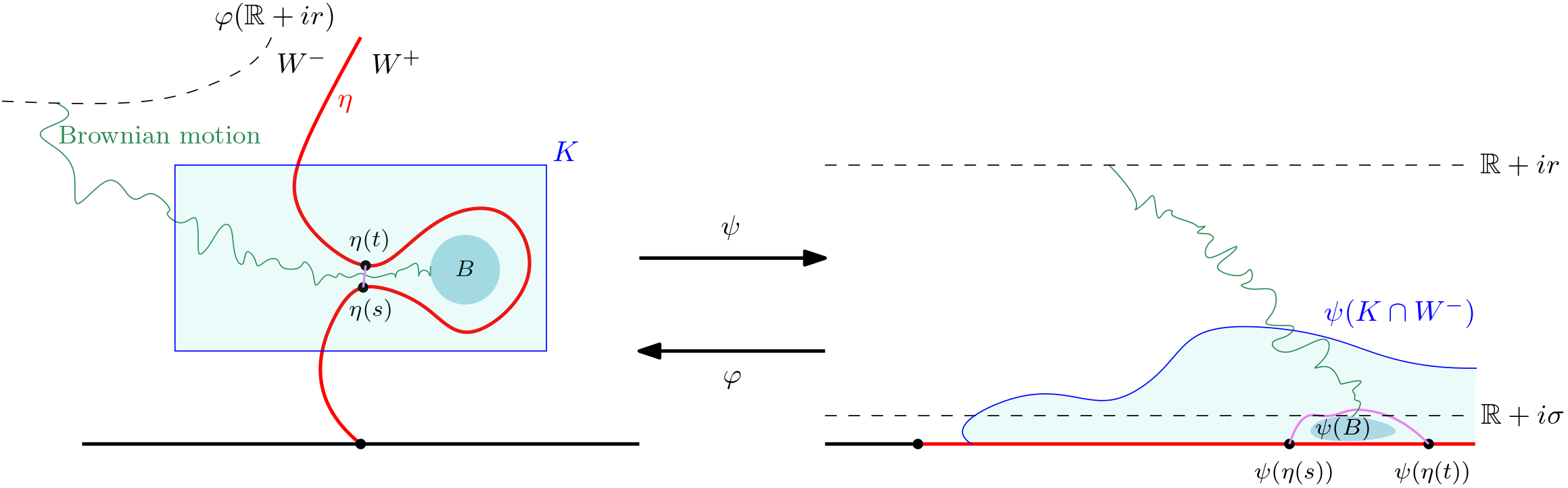}
			\caption{We establish a bound on the narrowness of bottlenecks in $\mathrm{SLE}_\kappa$ curves for $\kappa \in (0,4)$. If $\mathrm{diam}\, \eta({[}s,t{]}) \gg |\eta(s)-\eta(t)| $, we have a large ball $B$ surrounded by the union of $\eta({[}s,t{]})$ and the line segment ${[}\eta(s),\eta(t){]}$. Since a Brownian motion started on $\varphi(\R+ir)$ is unlikely to hit $B$ before exiting $W^-$, a Brownian motion started on $\R+ir$ is unlikely to hit $\psi(B)$ before exiting $\h$, making $\mathrm{diam}\, \psi(B)$ small. This is impossible since the conformal coordinate change preserves quantum areas, which are bounded above and below by polynomials in Euclidean diameter, so the diameter of $\psi(B)$ cannot be smaller than a certain power of the diameter of $B$.}
		\end{figure}
		\begin{itemize}
			\item $B(y,\delta^{\widetilde{\alpha}}) \subseteq B(\eta(t),\delta)\setminus \eta$, and $B(y,2\delta^{\widetilde{\alpha}})$ intersects $\eta$;
			\item if $O$ is the connected component of $y$ in $B(\eta(t),\delta)\setminus\eta$, and $a \in \partial O\setminus \eta(t;\delta)$, then every path from $y$ to $a$ in $O\cup \{a\}$ exits the ball $B(y,\delta^{\widetilde{\xi}})$. (Here $\eta(t;\delta))$ is defined as the SLE segment $\eta([\sigma,\tau])$, where $\sigma$ and $\tau$ are respectively the last time before $t$, and the first time after $t$, that $\eta$ hits $B(\eta(t),\delta)$.) 
		\end{itemize}
		Note that the proof in that paper can be used to find such points on either side of $\eta$, and that the hypothesis does not depend on the parametrization of $\eta$; indeed, the parametrization will not matter for what follows, so we will choose to parametrize our $\mathrm{SLE}$ curves by capacity.
		
		 Let $W$ be either $W^{-}$ or $W^{+}$, fix compact axis-parallel rectangles $K', K \subset \h$ such that $K' \subset \mathrm{int}\, K$, and fix a conformal map $\psi: W\to \h$ fixing 0 and $\infty$, with inverse $\varphi$. Fix $r$ such that $\mathrm{Im}\, z \le r/2$ for each $z \in \psi(K\cap W)$. Given $0 < s < t$ such that $\eta(s),\eta(t)\in~K$, let $[\eta(s),\eta(t)]$ be the straight line segment from~$\eta(s)$ to~$\eta(t)$, and let $\ell=|\eta(s)-\eta(t)|$. Also let $\p_z$ be the law of a complex Brownian motion started at $z$. Then if $B$ is a closed ball in $\h$ contained in $K'$ and $B$ is disconnected from $\varphi(\R+ir)$ by the union of $\eta$ and $[\eta(s),\eta(t)]$, the Beurling estimate gives that there exists $c>0$ such that, for each $z \in B$, $a\in \R$,
		\begin{equation}
			\p_z [ \mathrm{hit} \, \varphi(\R+ir) \, \mathrm{before} \, \eta\,\cup\,\R ] \le \frac{c\ell^{1/2}}{(\mathrm{dist}\, (K',\partial K))^{1/2}}, \quad \p_{\varphi(a+ir)}[ \mathrm{hit} \, B \, \mathrm{before} \, \eta\,\cup\,\R ] \le \frac{c\ell^{1/2}}{(\mathrm{dist}\, (K',\partial K))^{1/2}}.
		\end{equation}
		By conformal invariance of Brownian motion, applying $\psi$ gives
		\begin{equation}\label{eq:hitpsib}
			\p_{\psi(z)} [ \mathrm{hit} \, \R+ir \, \mathrm{before} \, \R ] \le \frac{c\ell^{1/2}}{(\mathrm{dist}\, (K',\partial K))^{1/2}}, \quad \p_{a+ir}[\mathrm{hit} \, \psi(B) \, \mathrm{before} \, \R ] \le \frac{c\ell^{1/2}}{(\mathrm{dist}\, (K',\partial K))^{1/2}}.
		\end{equation}
		Let $\sigma = \sup_{z\in\psi(B)}{\mathrm{Im}\, z}$. Then by compactness we can choose $a\in\R$ so that $a+i\sigma \in \psi(B)$. Now by gambler's ruin we have
		\begin{equation}
			\p_{a+i\sigma} [ \mathrm{hit} \, \R+ir \, \mathrm{before} \, \R ] = \frac{\sigma}{r}
		\end{equation}
		and therefore $\sigma/r \le \frac{c\ell^{1/2}}{(\mathrm{dist}\, (K',\partial K))^{1/2}}$.
		
		Suppose $z, w \in \psi(B)$ with $|\mathrm{Re}(z-w)| = \rho>0$. Without loss of generality suppose $\mathrm{Im} \, z > \mathrm{Im} \, w$. Set $a = \mathrm{Re}(\frac{z+w}{2})$. Then set $L_1, L_2, L_3$ to be horizontal line segments of respective lengths $\rho/2$, $2\rho/3$, $2\rho/3$ and centres $a+i\,\mathrm{Im}\, z$, $a$, $a +i\,\mathrm{Im}\, z$. Let $R$ be the rectangle that has $L_2$ and $L_3$ as opposite sides. If a Brownian motion from $a+ir$ exits $\h+i\,\mathrm{Im}\, z$ through $L_1$ and after hitting~$L_1$ makes an excursion across $R$ from $L_3$ to $L_2$ without hitting the vertical sides of $R$ or hitting $(\R+i\mathrm{Im}\, z)\setminus L_3$, then it must hit $\psi(B)$ before $\R$. Since we already have $\mathrm{Im}\, z$, $\mathrm{Im}\, w = O(\ell^{1/2})$, the probability from $a+ir$ of exiting $\h+i\,\mathrm{Im}\, z$ through $L_1$ is $\frac{2}{\pi}\arctan{\frac{\rho}{4(r-\mathrm{Im}\, z)}}$ (using the Poisson kernel in $\h$), whereas assuming $\rho \ge \ell^{1/2}$ (and thus $\rho = \Omega(\sigma) = \Omega(\mathrm{Im}\, z)$) the probability from $u$ of the latter event can be bounded below by a constant uniformly for $u \in L_1$, so by the second inequality in (\ref{eq:hitpsib}) we must have $\rho = O(\ell^{1/2})$. We have therefore shown that the (Euclidean) diameter of $\psi(B)$ is $O(\ell^{1/2})$. 
		
		Now, given $s<t$ and $\widetilde{\alpha} > 1$, we can apply the non-self-tracing hypothesis to the ball$$B\left(\eta(u),\frac{1}{4}\mathrm{diam} \, \eta([s,t])\right)$$for some $s<u<t$, so that the connected component of $\eta(u)$ in that ball is a subsegment of $\eta([s,t])$. This gives us a ball~$B$ of radius $\left(\frac{1}{4}\mathrm{diam} \, \eta([s,t])\right)^{\widetilde{\alpha}}$ which is disconnected from~$\infty$ by the union of~$\eta$ and the straight line segment from $\eta(s)$ to $\eta(t)$ (recalling that we can choose the ball to be on the appropriate side of $\eta$).
		
		We can now compare $\mathrm{diam} \, \eta([s,t])$ to $\ell$ using Prop.~\ref{prop:volume_moment}, which implies that for $\zeta_1$ and $\zeta_2$ as in the statement of that proposition, if $B$ is the ball above, we have $\mu_h(B) = \Omega((\mathrm{diam}\,\eta([s,t]))^{\widetilde{\alpha}\zeta_1})$, but if $\widetilde{h} = h\circ\varphi + Q\log{|\varphi'|}$ then $\widetilde{h}$ itself has the law of a quantum wedge, so we have $\mu_h(B) = \mu_{\widetilde{h}}(\psi(B)) = O(\ell^{\zeta_2/2})$. So $\mathrm{diam}\,\eta([s,t]) = O(\ell^{\zeta_2/(2\widetilde{\alpha}\zeta_1)})$, as required.
	\end{proof}
	We will now prove Prop.~\ref{prop:bihldrr} from the results of~\cite{dfgps} giving H\"{o}lder continuity away from the boundary.
	\begin{proof}[Proof of Prop.~\ref{prop:bihldrr}]
		Continuing in the setting of the proof of Prop.~\ref{prop:bottleneck}, another use of the Poisson kernel in $\h$ gives that, for $a = \frac{1}{2}\mathrm{Re}\,(\psi(\eta(s))+\psi(\eta(t)))$, 
		\begin{equation*}
			\p_{a+ir}[\mathrm{exit} \, \h \, \mathrm{through} \, [\psi(\eta(s)),\psi(\eta(t))]] = \Theta(|\psi(\eta(s))-\psi(\eta(t))|).
		\end{equation*}
		But the LHS is equal to
		\begin{equation*}
			\p_{\varphi(a+ir)}[\mathrm{exit} \, W \, \mathrm{through} \,  \eta([s,t])] =O( (\mathrm{diam} \, \eta([s,t]))^{1/2})
		\end{equation*}
		(by the Beurling estimate). Combining this with the diameter estimate we find
		\begin{equation*}
			|\psi(\eta(s))-\psi(\eta(t))| = O(\ell^{\zeta_2/(4\widetilde{\alpha}\zeta_1)}) = O(|\eta(s)-\eta(t)|^{\zeta_2/(4\alpha\zeta_1)}).
		\end{equation*}
		In other words, $\psi$ is locally H\"{o}lder continuous on $\eta$. Note that for a fixed compact set $K \subset \h$ at positive distance from~$\partial \h$, we have $\chi^{-1}$-H\"{o}lder continuity of the Euclidean metric w.r.t.\ $\fd_h$ on $K$ for any $\chi > \xi(Q+2)$ (this follows from~\cite[Prop.~3.18]{dfgps} for the whole-plane GFF $h^{\mathrm{wp}}$ and the almost sure finiteness of $\sup_K{\mathfrak{h}}$). So, if $\eta(s)$, $\eta(t) \in K$, then
		\begin{equation*}
			|\psi(\eta(s))-\psi(\eta(t))| = O(|\eta(s)-\eta(t)|^{\zeta_2/(4\widetilde{\alpha}\zeta_1)}) = O(\fd_h(\eta(s),\eta(t))^{\zeta_2/(4\widetilde{\alpha}\chi\zeta_1)}).
		\end{equation*}
		Observing finally that $\fd_h(\eta(s),\eta(t)) \le \fd_h(\eta(s),\eta(t);W) = \fd_{\widetilde{h}}(\psi(\eta(s)),\psi(\eta(t)))$ gives the desired Prop.~\ref{prop:bihldrr} with $\widetilde{\beta} = \zeta_2/(4\widetilde{\alpha}\chi\zeta_1)$. Observe that since we require $\widetilde{\alpha}>1$, $\chi > \xi(Q+2)$, $\zeta_1>\gamma(Q+2)>\gamma(Q-2)(1-s_+)>\zeta_2$, we obtain that the result holds in the range\begin{equation}0<\widetilde{\beta}< \frac{(Q-2)(1-s_+)}{4\xi(Q+2)^2}.
		\end{equation}
	\end{proof}
	\begin{proof}[Proof of Prop.~\ref{prop:biholder}]
		It suffices to prove the left-hand inequality, since the right-hand inequality is given by Prop.~\ref{prop:holder}. Let $\chi > \xi(Q+2)$ and $\sigma < \xi(Q-2)$ be arbitrary. Fix some $\beta < \widetilde{\beta}$, i.e.\ $\beta < (Q-2)(1-s_+)/(4\xi(Q+2)^2) $. To recap what we have proven so far, suppose we are on the intersection of the almost sure events of Prop.~\ref{prop:bihldrabi} and Prop.~\ref{prop:bihldrr}. Then fixing $K \subset \overline{\h}$ compact, there is some finite $C > 0$ on which we have
		\begin{equation}\label{eq:case2}
			|(a_1+b_1i)-(a_2+b_2i)| \le C\fd_h(a_1+b_1i,a_2+b_2i)^\beta
		\end{equation}
		provided $a_1+b_1i$, $a_2+b_2i \in K$ and \emph{either} $b_1=b_2=0$ \emph{or} $b_1=0$, $a_1=a_2$. (Note that we can use Prop.~\ref{prop:bihldrabi} because $\beta < (\xi(Q+2))^{-1}$.)
		
		In order to remove this second condition and thus deduce a H\"{o}lder exponent for the Euclidean metric w.r.t.\ $\fd_h$ on a compact set $K \subseteq \overline{\h}$, we will split into cases. Fix $\rho > 1/(\beta\sigma)$. Note that $ 1/(\beta\sigma)$ can be made arbitrarily close to $ 4(Q+2)^2/((Q-2)^2(1-s_+)$ and thus can be chosen to force $\rho > 1$. Suppose $\mathrm{Im}\, w \ge \mathrm{Im}\, z$ and $\mathrm{Im}\, w >0$.  We will justify that, almost surely, there exist finite constants $C_1$, $C_2$, $C_3$, $C_4 > 0$ such that:
		\begin{itemize}
			\item Firstly, if $\mathrm{Im}\, w > |w-z|^\rho$ then
\begin{equation}
\label{eq:wR}\fd_h(w,\R+(i/4)\mathrm{Im}\, w) \ge C_1|w-z|^{\chi\rho}; \\
\end{equation}
			\item Secondly, if $4\,\mathrm{Im}\, z > \mathrm{Im}\, w$ and $\mathrm{Im}\, w > |w-z|^\rho$ then
\begin{equation}
\label{eq:wzH}\fd_h(w,z; \h+(i/4)\mathrm{Im}\, w) \ge C_2 |w-z|^\chi.
\end{equation}
			\item Finally, if $\mathrm{Im}\, w \le |w-z|^\rho$ and $|w-z|\le 4^{1/(1-\rho)}$ then:
			\begin{align}
			\label{eq:RwRz} \fd_h(\mathrm{Re} \, w, \mathrm{Re} \, z) &\ge C_3 |w-z|^{1/\beta}; \\
			\label{eq:wRw}  \fd_h(w, \mathrm{Re} \, w) &\le C_4 |w-z|^{\sigma\rho}; \\
			\label{eq:zRz}  \fd_h(z, \mathrm{Re} \, z) &\le C_4 |w-z|^{\sigma\rho}.
			\end{align}
		\end{itemize}

		Fixing the compact set $K$, we can choose $U$ a bounded axis-parallel rectangle open in $\overline{\h}$ and containing $K$ with the property that $\mathrm{dist} (z,\partial U \setminus \R) \ge 1$ for all $z\in K$. By the local H\"{o}lder continuity of $\fd_h$ w.r.t.\ the Euclidean metric, this ensures that $\fd_h (z,\partial U \setminus \R)$ will almost surely be bounded below by some $C_K > 0$ uniformly in $z\in K$. We now bound $\fd_h(w,\R+\frac{i}{4}\mathrm{Im}\, w)$ from below by the minimum of $C_K$ and the internal metric distance $\fd_h(w,\R+\frac{i}{4}\mathrm{Im}\, w; U)$. By (\ref{eq:dyadicholder}), there exists an almost surely finite constant $C'$ such that for all $w\in K$, we have 
		\begin{equation*}
			\fd_h(w,\R+(i/4)\mathrm{Im}\, w;U) \ge C' (\mathrm{Im}\, w)^\chi,
		\end{equation*}
		which together with the assumption $\mathrm{Im}\, w > |w-z|^\rho$ establishes (\ref{eq:wR}).
		
		By~\cite[Prop.~3.18]{dfgps}, for $\ve > 0$ sufficiently small the Euclidean metric is almost surely $(\chi-\ve)^{-1}$-H\"{o}lder continuous on~$K$ w.r.t.\ $\fd_{h^{\mathrm{wp}}}$ (where the additive constant in $h^{\mathrm{wp}}$ is fixed, say, so that $h^{\mathrm{wp}}_1(0)=0$), and thus also w.r.t.\ the larger internal metrics $\fd_{h^{\mathrm{wp}}}(\cdot,\cdot;\h+yi)$ for each $y>0$, with the same H\"{o}lder constant. Thus there almost surely exists $C_5>0$ such that$$\fd_{h^\mathrm{wp}} (w,z; \h+(i/4)\, \mathrm{Im}\, w) \ge C_5 |w-z|^{\chi-\ve}.$$This implies that$$\fd_h(w,z;\h+(i/4)\, \mathrm{Im}\, w) \ge \min\left\{C_K, C_5 |w-z|^{\chi-\ve} \cdot \inf_{U\cap(\h+(i/4)\, \mathrm{Im}\, w)} e^{\xi(h-h^\mathrm{wp})} \right\}.$$By the proof of Prop.~\ref{prop:bihldrabi}, if $h^\mathrm{wp}$ is coupled with $h$ so that the difference $h^\mathrm{wp}-h$ is a harmonic function, we almost surely have 
		\begin{equation}\label{eq:corrections}
			\sup_{U \cap (\h + yi)} e^{\xi(h^\mathrm{wp}-h)} = O(y^{-\ve/\rho})
		\end{equation}
		for each $\ve > 0$, while the same holds for $e^{\xi(h-h^\mathrm{wp})}$. In other words, $\inf_{U\cap(\h+yi)} e^{\xi(h-h^\mathrm{wp})}  = \Omega(y^{\ve/\rho})$, and since $\mathrm{Im}\, w > |w-z|^\rho$ we have $(\mathrm{Im}\, w)^{\ve/\rho} > |w-z|^\ve$, which establishes (\ref{eq:wzH}).
		
		Now turn to the case $\mathrm{Im}\, w \le |w-z|^\rho$. Since the assumption $|w-z|\le 4^{1/(1-\rho)}$ gives that $|w-z|^\rho \le \frac{1}{4}|w-z|$, we have $|\mathrm{Re}\, w - \, \mathrm{Re}\, z| \ge \frac{1}{2}|w-z|$, and thus$$\fd_h(\mathrm{Re}\, w, \mathrm{Re}\, z)  \ge C^{-1/\beta} |\mathrm{Re}\, w - \, \mathrm{Re}\, z|^{1/\beta} \ge (2C)^{-1/\beta}|w-z|^{1/\beta}.$$We thus obtain (\ref{eq:RwRz}) with $C_3 = (2C)^{-1/\beta}$.
		
		The existence of $C_4>0$ finite satisfying (\ref{eq:wRw}) and (\ref{eq:zRz}) follows from Prop.~\ref{prop:holder} together with the assumption that $\mathrm{Im}\, z\le \, \mathrm{Im}\, w \le |w-z|^\rho$. 

		Having justified the estimates (\ref{eq:wR})--(\ref{eq:zRz}) we now finish the proof. If $\mathrm{Im}\, w > |w-z|^\rho$ then either $4\, \mathrm{Im}\, z \le \mathrm{Im}\, w$, in which case it follows from (\ref{eq:wR}) that we have
		\begin{equation*}
			\fd_h(w,z) \ge \fd_h(w,\R+(i/4)\,\mathrm{Im}\, w) \ge C_1 |w-z|^{\chi\rho},
		\end{equation*}
		or $4\,\mathrm{Im}\, z > \mathrm{Im}\, w$. In this latter case, since $\rho>1$ and $\mathrm{Im}\, w > |w-z|^\rho$, it follows from (\ref{eq:wR}) and (\ref{eq:wzH}) that there almost surely exists $C_6 > 0$ such that 
		\begin{equation*}
			\fd_h(w,z) \ge \min\{ \fd_h(w,\R+(i/4)\,\mathrm{Im}\, w), \fd_h(w,z; \h+(i/4)\,\mathrm{Im}\, w) \} \ge C_5 |w-z|^{\chi\rho}.
		\end{equation*}
	
		On the other hand, if $\mathrm{Im}\, w \le |w-z|^\rho$ then, since by the triangle inequality $\fd_h(\mathrm{Re} \, w, \mathrm{Re} \, z) \le \fd_h(\mathrm{Re} \, w,w) + \fd_h(w,z) + \fd_h(z, \mathrm{Re} \, z)$, it follows from (\ref{eq:RwRz}), (\ref{eq:wRw}) and (\ref{eq:zRz}) that 
		\begin{equation*}
			\fd_h(w,z) \ge C_3 |w-z|^{1/\beta} - 2C_4 |w-z|^{\sigma\rho}.
		\end{equation*}
		Since $\sigma\rho > 1/\beta$, the last three displays imply the left-hand inequality of the proposition for any~$\alpha_1$ subject to 
		\begin{equation}
			\alpha_1^{-1} < \frac{(Q-2)^2(1-s_+)}{4\xi(Q+2)^3}
		\end{equation}
		at least when $|w-z|\le 4^{1/(1-\rho)}$, but we can deduce it for general $w$, $z$ by considering points $w = w_0, w_1, \ldots, w_k =z$ along a path of finite $\fd_h$-length from $w$ to $z$ such that $|w_i-w_{i-1}| \le 4^{1/(1-\rho)}$ and using a Kolmogorov criterion-type argument as in (\ref{eq:kc}).
	\end{proof}
	Note that local bi-H\"{o}lder continuity on $\overline{\h}$ implies that $\fd_h$ induces the Euclidean topology on~$\overline{\h}$ and not just on $\h$, which completes the proof of Prop.~\ref{prop:dh}. Finally, we show the existence of geodesics.
	\begin{prop}\label{prop:geodesic-existence}
		Let $h$ be a free-boundary GFF on $\h$ minus $\alpha\log|\cdot|$ for some $\alpha < Q$, with the additive constant fixed such that $h_1(0)=0$. Then it is almost surely the case that for any $z$, $w \in \overline{\h}$ there exists a $\fd_h$-geodesic between $z$ and $w$, which does not hit $\infty$.
	\end{prop}
	\begin{proof}
		Since $\fd_h$-bounded subsets are also Euclidean-bounded by Lemma~\ref{lem:bounded}, and the two metrics induce the same topology, the Heine--Borel theorem gives that $(\overline{\h},\fd_h)$ is almost surely a \emph{boundedly compact} space (i.e., closed bounded sets are compact), which implies that there exists a geodesic between any two points (provided they are connected by a rectifiable curve, which we know holds for any two points since $\fd_h$ is a length metric). This is a standard result in metric geometry~\cite[Cor.~2.5.20]{bbi} proven by taking an infimizing sequence of paths which by bounded compactness can be assumed to lie in a compact set, extracting a uniformly converging subsequence by an Arzel\`{a}--Ascoli-type result, and applying lower semicontinuity of length to conclude that the limit is a geodesic.
	\end{proof}
	\section{Bound on $\gamma$-LQG area near the boundary}\label{section:volbound}
	Our aim in the entirety of this section is to prove the following lower bound on the $\mu_h$-area near a boundary segment. We will achieve this via the result~\cite[Thm~1.2]{dms} that an independent SLE-type curve cuts a quantum wedge into two independent wedges, but here we will use several curves to cut out many independent surfaces that each have a positive chance to accumulate a large $\mu_h$-area within a $\fd_h$-neighbourhood of our boundary segment. These surfaces can be described as contiguous portions of a space-filling $\mathrm{SLE}_{16/\gamma^2}$-type curve similar to the one that generates the ``topological mating'' in~\cite[\S8]{dms}, but we will not need that description here.
	\begin{prop}\label{prop:volbound}
		Let $h$ be a free-boundary GFF with the additive constant fixed so that $h_1(0)=0$. For $\delta > 0$ and $a$, $b\in \R$ with $a<b$, define $\mathcal{B}_\delta([a,b])$ to be the set of points at $\fd_h$-distance $< \delta$ from the interval $[a,b]$. Fix $I \subset \R$ a compact interval and $u > 0$. Then there almost surely exists $M>0$ such that for each $[a,b] \subset I$ and $\delta \in (0,1)$ such that $\nu_h([a,b]) \ge 4\delta^{d_\gamma/2-u}$, we have
		\begin{equation}\label{eq:volbound}
			\mu_h(\mathcal{B}_\delta([a,b])) \ge M \delta^{\frac{d_\gamma}{2}} \nu_h([a,b]).
		\end{equation}
	\end{prop}
	\begin{proof}
		As before, using the radial--lateral decomposition and mutual absolute continuity, it is enough to prove this for a quantum wedge. In particular, we will consider a $\gamma$-wedge $(\h,h,0,\infty)$ (i.e., a wedge of weight 2, which is thick since $2 > \gamma^2/2$ for $\gamma \in (0,2)$), and use the result~\cite[Prop.~1.7]{zip} that the law of such a wedge is invariant under translating one marked point by a fixed amount of $\nu_h$-length. More precisely, if $(\h,h,0,\infty)$ is a $\gamma$-wedge and we fix $L>0$ and let $y>0$ be defined by $\nu_h([0,y])=L$, then the surface given by recentring the wedge such that $y$ becomes the origin (which can be described either by $(\h,h,y,\infty)$ or by $(\h,h(\cdot+y),0,\infty)$) is itself a $\gamma$-wedge.
		
		By the conformal welding/cutting result~\cite[Thm~1.2]{dms}, an independent $\SLE_{\gamma^2}(-1;-1)$ from~0 to~$\infty$ cuts the wedge $(\h,h,0,\infty)$ into two independent wedges of weight 1; by shift invariance, for any $L>0$, the same is true if we instead use an independent $\SLE_{\gamma^2}(-1;-1)$ from $a_L$ to~$\infty$, where~$a_L$ is defined as the point in $(0,\infty)$ for which $\nu_h([0,a_L])=L$.
		
		We can couple $\SLE_{\gamma^2}(-1;-1)$ curves $\eta_x$ from each $x \in \R$ to $\infty$ (or at least from each $x$ in a countable dense subset of $\R$) using the imaginary geometry results from~\cite{ig1}. Indeed, by~\cite[Thm~1.1]{ig1}, the flow line of a zero-boundary GFF $\mathring{h}$ on $\h$ started at $x\in\R$ is an $\SLE_{\gamma^2}(-1;-1)$ curve from $x$ to $\infty$, so we can simultaneously generate~$\eta_r$ for different values of $r>0$ by sampling such a GFF $\mathring{h}$ independently of $h$. By~\cite[Thm~1.5(ii)]{ig1}, almost surely, whenever any two such curves $\eta_c$, $\eta_{c'}$ intersect, they merge immediately upon intersecting and never subsequently separate. Moreover, by~\cite[Lemma~7.7]{ig1}, if $K$ is the set formed by the initial portions of two such curves $\eta_c$, $\eta_{c'}$ run until they intersect, then the subsequent merged curve stays in the unbounded component of $\h\setminus K$.
		
		Note that for $c<c'$ the curves $\eta_c$, $\eta_{c'}$ will merge almost surely. Indeed, if $-1 < \gamma^2/2-2$, i.e.\ $\gamma > \sqrt{2}$, then $\eta_c$ hits $(0,\infty)$ almost surely, and by scale invariance $\eta_c$ will then almost surely hit arbitrarily large $x>0$. Thus $\eta_c$ swallows $c'$ and then the transience of $\eta_{c'}$ implies that the two curves merge. 
		On the other hand, when $\gamma \le \sqrt{2}$, $\eta_c$ almost surely does not hit $\partial\h$. In this case one can map the unbounded region to the right of $\eta_c$ back to $\h$ via a conformal map~$\phi$; since~$\eta_c$ is a flow line, the field on $\h$ given by $\mathring{h}\circ\phi^{-1}-\chi\arg(\phi^{-1})'$ (the appropriate imaginary geometry coordinate change formula for $\phi(\eta_{c'})$ to be a flow line in $\h$) has boundary conditions $\lambda = \pi/\gamma$ on $(-\infty,\phi(c))$ and 0 on $(\phi(c),\infty)$. This means that, by~\cite[Thm~1.1]{ig1}, $\phi(\eta_{c'})$ is an $\SLE_{\gamma^2}(\underline{\rho})$ from $\phi(c')$ with two left-hand force points of weight $-1$ at $\phi(c)$ and $\phi(c')^-$ and a right-hand force point of weight $-1$ at $\phi(c')^+$. Since the weights of the force points on the left sum to $-2$, the curve $\phi(\eta_{c'})$ must collide with a left-hand force point, meaning that it merges with the left-hand boundary segment $\phi(\eta_c)$---indeed the denominator $V^{2,L}-W$ in~\cite[(1.11)]{ig1} evolves until hitting~0 as a Bessel process of dimension 1, i.e.\ a Brownian motion, and thus will hit~0 almost surely.
		
		\begin{figure}\label{fig:s5}
			\includegraphics[width=0.9\textwidth]{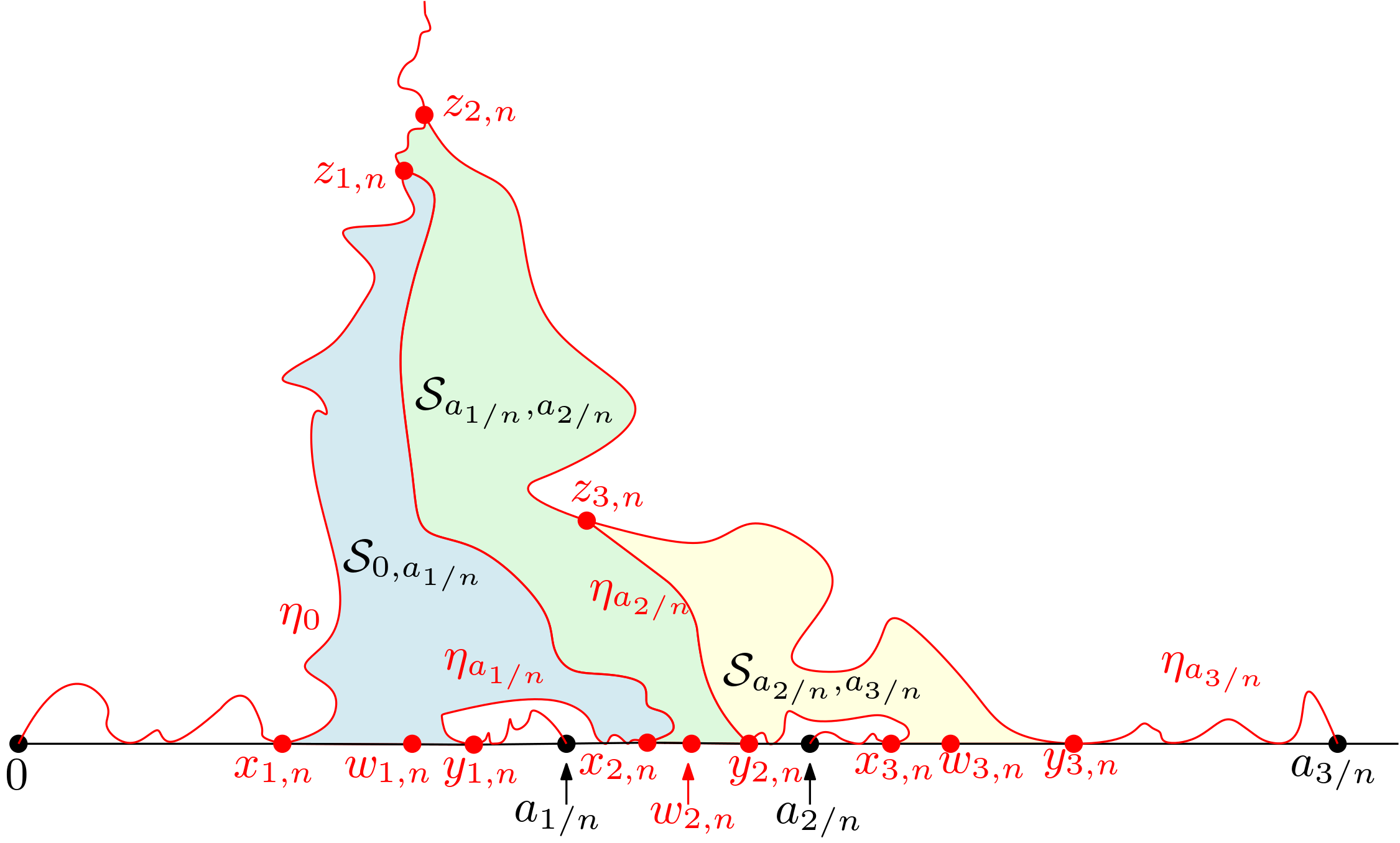}
			\caption{We show a lower bound on the $\mu_h$-area near a boundary segment by using coupled $\SLE_{\gamma^2}(-1;-1)$ curves to cut a wedge into independent surfaces each of which have a positive chance of accumulating some positive amount of $\mu_h$-area within a small $\fd_h$-distance of the boundary.}
		\end{figure}
		
		Given $c<c'$ denote by $\mathcal{S}_{a_c,a_{c'}}$ the quantum surface described by the restriction of the field $h$ to the unique connected component $S_{a_c,a_{c'}}$ of $\h\setminus(\eta_{a_c}\cup\eta_{a_{c'}})$ which is to the right of $\eta_{a_c}$ and to the left of~$\eta_{a_{c'}}$ and whose boundary contains non-trivial segments of both~$\eta_{a_c}$ and~$\eta_{a_{c'}}$. Almost surely, $\eta_{a_c}$,~$\eta_{a_{c'}}$ do not intersect on $\R$. Indeed, since it is an $\SLE_{\gamma^2}(-1;-1)$ from $a_c$ to $\infty$ and given~$a_c$ is conditionally independent of $h$, $\eta_{a_c}$ almost surely does not hit $a_{c'}$. On this event, $\eta_{a_{c'}}$ is an $\SLE_{\gamma^2}(-1;-1)$ from $a_{c'}$ to $\infty$, and since given $a_{c'}$ this curve is conditionally independent of $h$, $\eta_{a_{c'}}$ almost surely does not hit the unique point on $\mathbb{R} \cap \eta_{a_c}$ that is on the boundary of the unbounded component of $\h\setminus\eta_{a_c}$ to the right of $\eta_{a_c}$). Thus, it is almost surely the case that the intersection of $\partial S_{a_c,a_{c'}}$ with $\mathbb{R}$ is a bounded interval and that, for each interior point $x$ of this interval, there exists $r>0$ such that $S_{a_c,a_{c'}}$ contains the Euclidean semi-disc $\overline{B}(x,r)\cap\overline{\h}$. Moreover, since the law of $h$ on each deterministic open set not containing 0 is absolutely continuous w.r.t.~that of a free-boundary Gaussian free field, it almost surely holds that for each such $x$ and $r$, the smaller semi-disc $\overline{B}(x,r/2) \cap \overline{\h}$ has finite diameter w.r.t.\ the internal metric $\fd_h(\cdot,\cdot;B(x,r)\cap\overline{\h})$ (since we can find rationals $q_1$, $q_2$, $r_1$, $r_2$ for which $\overline{B}(x,r/2) \subseteq \overline{B}(q_1,r_1) \subset B(q_2,r_2) \subseteq \overline{B}(x,r)$, and it is almost surely the case that, for all $q_1$, $q_2$, $r_1$, $r_2$ such that $\overline{B}(q_1,r_1) \subset B(q_2,r_2)$, $\overline{B}(q_1,r_1)\cap\overline{\h}$ has finite diameter w.r.t.\ $\fd_h(\cdot,\cdot;B(q_2,r_2)\cap\overline{\h})$).
		
		Notice $\mathcal{S}_{a_{c_1},a_{c_2}}$ and $\mathcal{S}_{a_{c_3},a_{c_4}}$ are independent as quantum surfaces (i.e., modulo embedding) when $c_1 < c_2 \le c_3 < c_4$. Indeed, we know from the conformal welding result~\cite[Thm~1.2]{dms} that the surfaces given by the restrictions of $h$ to the regions to the left and right of $\eta_{c_2}$ are independent; the same holds for $\mathring{h}$ since $\eta_{c_2}$ is a flow line. Moreover, $h$ and $\mathring{h}$ are independent of each other. These independences together imply that $\mathcal{S}_{a_{c_1},a_{c_2}}$ and $\mathcal{S}_{a_{c_3},a_{c_4}}$ are independent.
		
		For each $k,n\in\mathbb{N}$, we can consider the surfaces $\mathcal{S}_{a_{(k-1)/n},a_{k/n}}$, with three marked points $x_{k,n}$, $y_{k,n}$, $z_{k,n}$ given by, respectively, the last point on $\R\cap S_{a_{(k-1)/n},a_{k/n}}$ that $\eta_{a_{(k-1)/n}}$ hits before merging with $\eta_{a_{k/n}}$, the last point on $\R\cap S_{a_{(k-1)/n},a_{k/n}}$ that $\eta_{a_{k/n}}$ hits before merging with $\eta_{a_{(k-1)/n}}$, and the point in $\h$ where the two curves merge. As explained above these surfaces are independent. By shift invariance, these surfaces are identically distributed when considered as triply marked surfaces modulo embedding.
		
		Consider a point $w_{k,n}$ in the interval $(x_{k,n},y_{k,n})$ (which has positive length almost surely, since $\eta_{a_{(k-1)/n}}$ and~$\eta_{a_{k/n}}$ do not merge on $\R$); for concreteness we may set $w_{k,n}$ to be the unique point in the interval such that $\nu_h([x_{k,n},w_{k,n}])=\nu_h([w_{k,n},y_{k,n}])$. As explained earlier, we can almost surely find $r>0$ such that $\overline{B}(w_{k,n},r) \cap \overline{\h}$ is contained in $S_{a_{(k-1)/n},a_{k/n}}$, and that  $\overline{B}(w_{k,n},r/2) \cap \overline{\h}$ has finite diameter w.r.t.\ the internal metric $\fd_h(\cdot,\cdot;B(w_{k,n},r)\cap\overline{\h})$. Thus, the set $\mathcal{B}^{n^{-2/d_\gamma}}_{k,n}$ consisting of the intersection of $\mathrm{int}\, S_{a_{(k-1)/n},a_{k/n}}$ with the open ball of radius $n^{-2/d_\gamma}$ centred on $w_{k,n}$ w.r.t.\ the internal metric $\fd_h(\cdot,\cdot;\mathrm{int}\, S_{a_{(k-1)/n},a_{k/n}})$ is non-empty and open w.r.t.~the Euclidean topology, so has positive $\mu_h$-measure almost surely. Thus, for every $p\in(0,1)$, there exists $c>0$ such that 
		\begin{equation}\label{eq:pc}
			p_c := \p[ \mu_h(\mathcal{B}^{n^{-2/d_\gamma}}_{k,n}) \ge cn^{-2} ] > p.
		\end{equation}
		Indeed $p_c\to 1$ as $c\to 0$. Observe that by shift invariance $p_c$ does not depend on $k$.
		
		Adding a constant $C$ to the field $h$ scales $\nu_h$-lengths by $e^{\gamma C/2}$ and $\mu_h$-areas by $e^{\gamma C}$, as well as scaling $\fd_h$-distances by $e^{\xi C}$. By \cite[Prop.~4.7(i)]{dms}, the circle-average embedding of~$h+C$ into~$\h$ has the same law as that of $h$, so if we add a constant $C$ to the field $h$ on $\h$ and then rescale appropriately to achieve the circle-average embedding, the resulting surface has the same law as $(\h,h,0,\infty)$. The rescaling factor is independent of $\mathring{h}$, which itself has a scale-invariant law, so if we also apply the rescaling to the field $\mathring{h}$ and $\eta$ then the joint law is invariant. This shows that, as a triply marked quantum surface, the law of $\eta_{a_0,a_t}$ is the same as the law of $\eta_{a_0,a_1}$ but with $\nu_h$-lengths scaled by $t$, $\mu_h$-areas scaled by $t^2$ and $\fd_h$-lengths scaled by $t^{2/d_\gamma}$. This implies that the probability $p_c$ in (\ref{eq:pc}) does not depend on $n$.
		
		For $c>0$ and $k,n\in\mathbb{N}$, define the event$$A_{c,k,n} =\left\{\mu_h(\mathcal{B}^{n^{-2/d_\gamma}}_{k,n}) \ge cn^{-2}\right\}.$$If $c>0$ is chosen so that $p_c>1/2$, then by a standard binomial tail estimate (see, for example,~\cite[Lemma~2.6]{mq}), there exists $C_0(p_c)>0$ for each $N, n_0\in\mathbb{N}$, we have 
		\begin{equation}\label{eq:ackn}
			\p \left[ \sum_{n=n_0}^{n_0+N-1} \mathbbm{1}_{A_{c,k,n}} \le N/2 \right] \le e^{-C_0(p_c)N}.
		\end{equation}
		Moreover $C_0(p_c)\to\infty$ as $p_c\to 1$, i.e.~as $c\to 0$. Thus, since the $\mathcal{B}^{n^{-2/d_\gamma}}_{k,n}$ are disjoint for different~$k$, if $t, s\in\R$ are such that $t-s\ge 1$ (so that $\lfloor t-s \rfloor \ge \frac{1}{2}(t-s)$), we have $$\p\left[\mu_h(\mathcal{B}_{n^{-2/d_\gamma}}([a_{s/n},a_{t/n}])) \le \frac{1}{4} (t-s) cn^{-2} \le \frac{1}{2} \lfloor t-s \rfloor cn^{-2} \right] \le e^{-C_0(p_c)\lfloor t-s \rfloor} \le e^{-\frac{1}{2}C_0(p_c)(t-s)}.$$
		For $T,v>0$ fixed, this probability converges when summed over all choices of $n=2^{m}$, $s=j2^{mv}$, $t=(j+1)2^{mv}$ with $m, j$ non-negative integers such that $(j+1)2^{-m(1-v)} \le T$. Indeed, the sum is bounded by $\sum_m 2^{m(1-v)}Te^{-\frac{1}{2}C_0(p_c)2^{mv}}$ which converges superpolynomially fast in $2^m$. Thus we find that, with superpolynomially high probability in $2^{m_0}$ as $m_0\to\infty$, whenever $m\ge m_0$ and $j\ge 0$ is an integer such that $(j+1)2^{-m+mv} \le T$, we have
		\begin{equation}
			\mu_h(\mathcal{B}_{2^{-2m/d_\gamma}}([a_{j2^{-m+mv}},a_{(j+1)2^{-m+mv}}])) > \frac{c}{4} 2^{-2m+mv}.
		\end{equation}
		Furthermore, by disjointness of the $\mathcal{B}^{n^{-2/d_\gamma}}_{k,n}$, on the event considered above (and thus still with superpolynomially high probability in $2^{m_0}$ as $m_0\to\infty$) it holds that whenever $j < k$ are non-negative integers with $k2^{-m+mv}\le T$, we have
		\begin{equation}
			\mu_h(\mathcal{B}_{2^{-2m/d_\gamma}}([a_{j2^{-m+mv}},a_{k2^{-m+mv}}])) > \frac{c}{4} 2^{-2m+mv} (k-j).
		\end{equation}
		
		On this event, for each $m\ge m_0$ (\ref{eq:volbound}) holds for each subinterval $[a,b]$ of $[a_0,a_T]$ of $\nu_h$-length at least $2^{-m+mv+2}$ with $\delta = 2^{-2m/d_\gamma}$, with $u=d_\gamma v/2$ and with $M=\frac{c}{16}$, since we can find a subinterval of $[a,b]$ of the form $[a_{j2^{-m+mv}},a_{k2^{-m+mv}}]$ whose $\nu_h$-measure is at least $1/4$ that of $[a,b]$. This gives an overall constant (i.e., one holding for all $\delta \in (0,1)$) of $M = c2^{-m_0}/16$ (using the right-hand side of (\ref{eq:volbound}) with $\delta = 2^{-2m_0/d_\gamma}$ as the lower bound for all larger $\delta$) and holds with probability bounded by $T\sum_{m\ge m_0} 2^{m(1-v)}e^{-\frac{1}{2}C_0(p_c)2^{mv}}$. Since $\nu_h(I)$ has a finite first moment, for any $\alpha > 0$ we can set $T = 2^{m_0\alpha} $ and observe that the probability that we can take $M = c2^{-m_0}/16$ is bounded by $2^{m_0\alpha}\sum_{m\ge m_0} 2^{m(1-v)}e^{-\frac{1}{2}C_0(p_c)2^{mv}}$ plus the probability that $\nu_h(I)$ is greater than $T$, which is $O(2^{-m_0\alpha})$. Since $\alpha$ is arbitrary, this gives superpolynomial decay of the constant $M$.
	\end{proof}
	\section{Proofs of main results}\label{section:pfs}
	\subsection{Proof of Theorem~\ref{thm:main}}
	We now prove Theorem~\ref{thm:main}, the extension of~\cite[Thm~1.5]{metglu} to the $\gamma$-LQG metric for all $\gamma \in (0,2)$. Suppose we are in the setup of Theorem~\ref{thm:main}. That is, fix $\gamma \in (0,2)$ and $\mathfrak{w}^-$, $\mathfrak{w}^+>0$, and let $(\h,h,0,\infty)$ be a quantum wedge of weight $\mathfrak{w}:=\mathfrak{w}^-+\mathfrak{w}^+$ if $\mathfrak{w}\ge\gamma^2/2$ (so that the wedge is thick), or a single bead of a wedge of weight $\mathfrak{w}$, with specified $\gamma$-LQG area $\mathfrak{a}$ and $\gamma$-LQG boundary lengths $\mathfrak{l}^-$, $\mathfrak{l}^+ > 0$, if $\mathfrak{w} < \gamma^2/2$ (corresponding to a thin wedge). Let~$\eta$ be an independent $\mathrm{SLE}_{\gamma^2}(\mathfrak{w}^--2;\mathfrak{w}^+-2)$ from 0 to $\infty$ which we will parametrize by $\nu_h$-length as measured on either side of the curve (recall that these two boundary length measures agree by~\cite[Thm~1.4]{dms}). As in~\cite{metglu}, we define $V_\rho = \{ z \in \mathbb{C}: |z| < \rho, \mathrm{Im}\, z > \rho^{-1} \}$ for $\rho > 1$. For $z\in\h$ and $r>0$, write $B_r(z;\fd_h)$ for the open $\fd_h$-metric ball of radius~$r$ centred at $z$.
	
	We will replicate the argument of~\cite[\S4]{metglu}, establishing analogues of the lemmas in that section, beginning with an analogue of~\cite[Lemma~4.1]{metglu}:
	\begin{lemma}\label{lem:41}
		In the setting of Theorem~\ref{thm:main}, let $R>1$ and let $z_1$ and $z_2$ be independent samples from $\mu_h|_{V_R}$, normalized to be a probability measure. Almost surely, there exists a $\fd_h$-geodesic from $z_1$ to $z_2$ that does not hit 0 or $\infty$.
	\end{lemma}
	\begin{proof}
	First note that \cite[Lemmas~4.2,~4.3]{metglu} hold for general $\gamma$ just as in the $\gamma = \sqrt{8/3}$ case, since their proofs just rely on the locality and Weyl scaling properties of the $\gamma$-LQG metric, along with (in the case of \cite[Lemma~4.3]{metglu}) calculations for the Gaussian free field (\cite[Lemma~5.4]{lqg2}) that do not depend on $\gamma$. This establishes that no $\fd_h$-geodesic between points of $\h$ hits 0. For our analogue of~\cite[Lemma~4.1]{metglu}, we also need to know that for quantum typical points $z_1$, $z_2$ (i.e.\ if~$z_1$ and~$z_2$ are sampled independently according to $\mu_h$) there almost surely exists a $\fd_h$-geodesic. For a thick wedge, existence of geodesics that do not intersect $\infty$ follows from Prop.~\ref{prop:geodesic-existence} plus absolute continuity with the free-boundary GFF plus a log singularity. For beads of thin wedges, since $(\h,h,\infty,0) \overset{(d)}{=} (\h,h,0,\infty)$, the analogue of~\cite[Lemma~4.3]{metglu} gives that paths of near-minimal $\fd_h$-length between $z_1$ and $z_2$ must stay in a set that is Euclidean-bounded and thus $\fd_h$-compact (since $\fd_h$ still induces the Euclidean topology away from 0, by absolute continuity w.r.t.\ the free-boundary GFF away from 0), and thus we can still deduce the existence of a geodesic between $z_1$ and $z_2$ by the argument of~\cite[Cor.~2.5.20]{bbi}.
	\end{proof}
	We will not address the question of whether geodesics are unique here, since we do not need uniqueness for our results. 
	
	We now proceed to state and prove analogues of~\cite[Lemmas~4.5--4.9]{metglu}. We begin by using the estimates established in the previous sections to prove that a global regularity event $G_C$ holds with high probability, which is analogous to~\cite[Lemma~4.5]{metglu}. The remaining lemmas (the analogues of~\cite[Lemmas~4.6--4.9]{metglu}) will follow from this one in essentially the same way as in~\cite[\S4]{metglu}, though we will give the proofs here since there are minor differences, since conditions (iii) and (iv) in Lemma~\ref{lemma:45} are slightly weaker than those in~\cite[Lemma~4.5]{metglu}, and we have not ruled out the possibility of geodesics hitting the boundary. Given these lemmas, the remainder of the proof of Theorem~\ref{thm:main} will be identical to the argument in~\cite[Thm~1.5]{metglu}.
	\begin{lemma}\label{lemma:45}
		In the setting of Theorem~\ref{thm:main}, there exists $\beta>0$ such that, for all $u \in (0,1)$, $\rho>2$, $p\in(0,1)$, there is $C>\rho$ such that $\p[G_C]\ge 1-p$, where $G_C$ is the event that all the following hold:
		\begin{enumerate}[(i)]
			\item For each $z \in V_\rho$ and $0 < \delta \le 1$ such that $B_\delta(z;\fd_h) \cap \R = \varnothing$, we have $\mu_h(B_\delta(z;\fd_h)) \le C\delta^{d_\gamma-u}$.
			\item For each $U \in \mathcal{U}^- \cup \mathcal{U}^+$ with $U \cap V_\rho \neq \varnothing$, each $z \in \overline{U} \cap V_\rho$, and each $0 < \delta \le 1$, we have $\mu_h(B_\delta(z;\fd_{h|_U})) \ge C^{-1}\delta^{d_\gamma+u}$.
			\item For each $U \in \mathcal{U}^-\cup\mathcal{U}^+$ with $U \cap V_\rho \neq \varnothing$, and each $x,y \in \partial U \cap V_\rho$, we have
			\begin{equation*}
				\fd_{h|_U}(x,y) \le C \nu_h([x,y]_{\partial U})^{(2/d_\gamma)-u}.
			\end{equation*}
			\item For each $U \in \mathcal{U}^-\cup\mathcal{U}^+$ with $U \cap V_\rho \neq \varnothing$, each $0 < \delta < 1$, and each $x,y \in \partial U \cap V_\rho$ with $\nu_h([x,y]_{\partial U}) \ge 4\delta^{d_\gamma/2-u}$, we have
			\begin{equation*}
				\mu_h(B_\delta([x,y]_{\partial U};\fd_{h|_U})) \ge C^{-1} \delta^{(d_\gamma/2)} \nu_h([x,y]_{\partial U}).
			\end{equation*}
			\item For each $z \in {V_\rho}$ and $0 < \delta \le 1$, we have $B_\delta(z;\fd_h) \subseteq B(z,C\delta^\beta)$.
			\item For each $t>s>0$ such that $\eta(s)\in V_{\rho/2}$ and $|t-s|\le C^{-1}$, we have $\eta(t)\in V_\rho$.
		\end{enumerate}	
	\end{lemma}
	\begin{proof}
		Note first that it suffices to show that for each item, there almost surely exists some $C\in(\rho,\infty)$ for which that item holds, since this forces $\p[G_C]\to 1$ as $C\to\infty$.
		
		With this in mind, item (i) follows from~\cite[Thm~1.1]{afs}. Indeed, that result gives us that, for $h^{\mathrm{wp}}$ a whole-plane GFF normalized so that the circle average $h^{\mathrm{wp}}_1(0)=0$, $K$ a compact set and $\ve > 0$, we almost surely have
		\begin{equation}\label{eq:afs}
			\sup_{s\in(0,1)} \sup_{z\in K} \frac{\mu_{h^{\mathrm{wp}}} (B_s(z;\fd_{h^{\mathrm{wp}}}))}{s^{d_\gamma - \ve}} < \infty \quad \mathrm{and} \quad \inf_{s\in(0,1)} \inf_{z\in K} \frac{\mu_{h^{\mathrm{wp}}} (B_s(z;\fd_{h^{\mathrm{wp}}}))}{s^{d_\gamma + \ve}} > 0.
		\end{equation}
		Recall that we can couple a free-boundary GFF $h^{\mathrm{F}}$ on $\h$, normalized so that the semicircle average $h_1(0)$ is zero, with $h^{\mathrm{wp}}$, so that $\mathfrak{h}=h-h^{\mathrm{wp}}$ is a random harmonic function. We thus find that (\ref{eq:afs}) holds with $h^{\mathrm{F}}$ in place of $h^{\mathrm{wp}}$ provided $K \subset \h$ is at positive Euclidean distance from $\R$ (and thus positive $\fd_{h^{\mathrm{F}}}$-distance, so that we need only consider $s < \fd_{h^{\mathrm{F}}}(K,\R)$), and we can then deduce the same for $h$ either a thick quantum wedge or a bead of a thin quantum wedge (in the latter case with specified area and boundary lengths) by local absolute continuity, which implies that there almost surely exists $C<\infty$ for which item (i) holds.
		
		Item (v) follows from Prop.~\ref{prop:biholder} (for a free-boundary GFF, then for a wedge or bead thereof by absolute continuity). Just as in~\cite{metglu}, item (vi) follows from the continuity and transience of SLE from 0 to $\infty$ with force points, proved in~\cite[Thm~1.3]{ig1} (although the parametrization by quantum length depends on $\gamma$, observe that if (vi) holds for one parametrization then it holds for any other parametrization, though not necessarily with the same $C$). We now turn to items (ii)--(iv), which are required to hold for each of the surfaces $U$ cut out by $\eta$ that intersect $V_\rho$. We can reduce to considering finitely many such surfaces: exactly as explained in the first part of the proof of~\cite[Lemma~4.5]{metglu}, it suffices to show that for each $U \in \mathcal{U}^-\cup\mathcal{U}^+$ intersecting $V_\rho$, there almost surely exists $C\in(\rho,\infty)$ such that items (ii)--(iv) hold for $U$. We will map to~$\h$ and use absolute continuity arguments; in particular for each $U$ we will consider the surface $\phi_U(U)$, where we define~$x_U$ (resp.~$y_U$) as the first (resp.~last) point on $\partial U$ to be hit by $\eta$ (with $y_U = \infty$ when $U$ is a thick wedge) and set $\phi_U$ to be the unique conformal map $U\to \mathbb{H}$ sending $x_U$ to 0 and $y_U$ to~$\infty$ with the property that the covariantly transformed field $h_U:=h\circ\phi_U^{-1} + Q\log|(\phi_U^{-1})'|$ satisfies
		\begin{equation*}
			\mu_{h_U}(\mathbb{D}\cap\h) =\begin{cases}
				1 & \mu_h(U)=\infty \\
				\mu_h(U)/2 & \mu_h(U)<\infty
			\end{cases}.
		\end{equation*}
		As in the proof of~\cite[Lemma~4.5]{metglu}, we can find $\widetilde{\rho}$ such that $\phi_U(U\cap{V_{2\rho}})\subset {V_{\widetilde{\rho}}}$ with high probability (this is because the marked points $x_U$ and $y_U$ must be in $\mathbb{R}\cup\infty$). For the free-boundary GFF, item (ii) follows since (\ref{eq:afs}) holds with $K = \phi_U(\overline{U \cap V_\rho})$, whereas items (iii) and (iv) follow from Prop.~\ref{prop:dnu} and Prop.~\ref{prop:volbound} respectively, so it suffices to observe that the restriction of $h_U$ to ${V_{\widetilde{\rho}}}$ is absolutely continuous w.r.t.~the corresponding restriction of $h^F$.
	\end{proof}
	We now proceed as in~\cite{metglu}. Our version of~\cite[Lemma~4.6]{metglu} is as follows:
	\begin{lemma}\label{lemma:46}
		In the setting of Theorem~\ref{thm:main}, for each $v \in (0,1)$ there exists $u_0 = u_0(v)\in(0,1)$ such that whenever $0 < u \le u_0$, $\rho > 2$, $C > 1$, and $G_C = G_C(u,\rho)$ is the event of Lemma~\ref{lemma:45}, there exists $\ve_0 = \ve_0(C,u,v,\rho) > 0$ such that the following holds almost surely on $G_C$. If $0 < a < b \le a+\ve_0 < \infty$ and $\eta([a,b]) \cap V_{\rho/2} \neq \varnothing$, then we have
		\begin{equation*}
			\mathrm{diam}(\eta([a,b]);\fd_h) \ge 7(b-a)^{2(1+v)/d_\gamma}.
		\end{equation*}
	\end{lemma}
	\begin{proof}
		The proof is essentially the same as that of~\cite[Lemma~4.6]{metglu}. Fixing $v$, $C$, $u$, $\rho$, by condition (v) in Lemma~\ref{lemma:45} we can choose $\ve_0 \in (0,1)$ such that whenever $z \in V_{\rho/2}$, we have $B_{8\ve_0^{2(1+v)/d_\gamma}}(z;\fd_h) \subseteq V_\rho$. In particular this ball does not intersect $\R$.
		
		Now suppose $G_C$ occurs and fix $0 < a < b < a+\ve_0 < \infty$ and $z \in \eta([a,b]) \cap V_{\rho/2}$. Setting $\delta = (b-a)^{2(1+v)/d_\gamma}$, if we assume the statement of the lemma is false we have $\eta([a,b]) \subseteq B_{7\delta} (z;\fd_h) \subseteq V_\rho$. Noting that $V_\rho$ does not intersect~$\R$, we can find $U \in \mathcal{U}^-$ such that $\eta([a,b]) \subseteq \partial U$. Now, since $b-a \le \ve_0 < 1$ and $d_\gamma > 2$, we have $b-a= \delta^{d_\gamma/(2+2v)} \ge 4\delta^{d_\gamma/2-u}$ provided $u < \frac{d_\gamma}{2}(1-(1+v)^{-1})$ (i.e.~provided $u$ is sufficiently small depending on $v$) and $\ve_0$ is sufficiently small depending on $u$ and $v$, so by condition (iv) in Lemma~\ref{lemma:45}, we have
		\begin{equation*}
			\mu_h(B_\delta(\eta([a,b]);\fd_h)) \ge \mu_h(B_\delta(\eta([a,b]);\fd_{h|_U})) \ge C^{-1}(b-a)^{2+v}.
		\end{equation*}
		Condition (i) in Lemma~\ref{lemma:45} gives us
		\begin{equation*}
			\mu_h(B_{8\delta}(z;\fd_h)) \le 8^{d_\gamma-u} C (b-a)^{2(1+v)(1-u/d_\gamma)}.
		\end{equation*}
		If $u$ is sufficiently small (depending only on $v$), then if $\ve_0$ is small enough depending on $u$ and $C$ we can ensure that, whenever $b-a \le \ve_0$,
		\begin{equation*}
			8^{d_\gamma-u} C (b-a)^{2(1+v)(1-u/d_\gamma)} < C^{-1}(b-a)^{2+v}.
		\end{equation*}
		Thus $B_\delta(\eta([a,b]);\fd_h) \nsubseteq B_{8\delta}(z;\fd_h)$, so $\eta([a,b])$ cannot be contained in $B_{7\delta}(z;\fd_h)$.
	\end{proof}
	Next we give a version of~\cite[Lemma~4.7]{metglu}, which bounds the number of segments of $\eta$ of a fixed quantum length that can intersect a $\fd_h$-metric ball.
	\begin{lemma}\label{lemma:47}
		In the setting of Theorem~\ref{thm:main}, for each $v \in (0,1)$, there exists $u_0 = u_0(v) \in (0,1)$ such that, whenever $\rho > 2$, $C > 1$, $0 < u \le u_0$, there exists $\delta_0 = \delta_0(C,u,v,\rho) > 0$ such that, almost surely on $G_C = G_C(u,\rho)$, for each $z \in V_{\rho/2}$ and $\delta \in (0, \delta_0]$, the number of $k \in \mathbb{N}$ for which $\eta([(k-1)\delta^{d_\gamma/2},k\delta^{d_\gamma/2}])$ intersects $B_{\delta^{1+v}}(z;\fd_h)$ is at most $\delta^{-v}$.
	\end{lemma}
	\begin{proof}
		Assume $G_C$ occurs; then for $\delta$ small enough depending on $C$ and $\rho$ and $z \in V_{\rho/2}$, we have (using condition (v) in Lemma~\ref{lemma:45}) that $B_{3\delta^{1+v}}(z;\fd_h) \subseteq V_\rho$. By Lemma~\ref{lemma:46}, if $u$ is small enough depending on $v$ and $\delta$ is small enough depending on $C$, $u$, $v$, $\rho$, we have that for all $z \in V_{\rho/2}$ and all $k \in \mathbb{N}$, $\eta([(k-1)\delta^{d_\gamma/2},k\delta^{d_\gamma/2}]) \nsubseteq B_{2\delta^{1+v}}(z;\fd_h)$. Assume that $\delta$ and $u$ are chosen so that the above conditions hold. Let $K$ be the set of $k \in \mathbb{N}$ for which $\eta([(k-1)\delta^{d_\gamma/2},k\delta^{d_\gamma/2}])$ intersects $B_{\delta^{1+v}}(z;\fd_h)$; we now know that each $\eta([(k-1)\delta^{d_\gamma/2},k\delta^{d_\gamma/2}])$ also intersects $\h \setminus B_{2\delta^{1+v}}(z;\fd_h)$.
		Let~$\mathcal{V}$ be the set of connected components of $\h \setminus B_{\delta^{1+v}}(z;\fd_h)$, and for each $V \in \mathcal{V}$ let $\mathcal{O}_V$ be the set of those connected components of $V \setminus \eta$ which intersect $V \setminus B_{2\delta^{1+v}}(z;\fd_h)$.
		A topological argument given in Step 1 of the proof of~\cite[Lemma~4.7]{metglu} shows that we have $|K| \le 2+2\sum_{\mathcal{V}} |\mathcal{O}_V|$. This argument only on the facts that $\eta$ is continuous and transient and does not hit itself and that $\fd_h$ induces the Euclidean topology, so it applies here unchanged. 
		
		Fixing $V \in \mathcal{V}$, $O \in \mathcal{O}_V$, by the definition of $\mathcal{O}_V$ and the fact that $B_{3\delta^{1+v}}(z;\fd_h)$ does not intersect~$\R$, there exists $w_O \in \partial O \cap \eta$ satisfying$$\fd_h(w_O,\partial B_{\delta^{1+v}}(z;\fd_h)) = \fd_h(w_O,\partial B_{2\delta^{1+v}}(z;\fd_h)) = \frac{1}{2}\delta^{1+v}.$$Let $U_O$ be the connected component of $\h\setminus\eta$ containing $O$ (so $U_O \in \mathcal{U}^- \cup \mathcal{U}^+$), and let $B_O = B_{\frac{1}{2}\delta^{1+v}}(w_O;\fd_{h|_{U_O}})$. Then by construction, $\eta$ does not cross $B_O$, and
		$B_O \subseteq B_{\frac{1}{2}\delta^{1+v}}(w_O;\fd_h) \subseteq B_{2\delta^{1+v}}(z;\fd_h) \setminus B_{\delta^{1+v}}(z;\fd_h)$. In particular $B_O \subseteq O$, which implies that $B_O$ and $B_{O'}$ are disjoint when $O$ and $O'$ are distinct elements of $\bigcup_{V\in\mathcal{V}} \mathcal{O}_V$. 
		
		Now since $B_{3\delta^{1+v}}(z;\fd_h) \subseteq V_\rho$, each $U_O$ must intersect $V_\rho$, so by condition (ii) in Lemma~\ref{lemma:45}, for each $O \in \bigcup_{V\in\mathcal{V}} \mathcal{O}_V$ we have $\mu_h(B_O) \ge C^{-1}(\frac{1}{2}\delta)^{(d_\gamma+u)(1+v)}$. We thus find that
		\begin{equation*}
			C^{-1}(\delta/2)^{(d_\gamma+u)(1+v)} \sum_{V \in \mathcal{V}} |\mathcal{O}_V| \le \mu_h(B_{2\delta^{1+v}}(z;\fd_h)) \le C(2\delta)^{(d_\gamma-u)(1+v)},
		\end{equation*} 
		where the second inequality is by condition (i) in Lemma~\ref{lemma:45}. This, combined with the earlier fact that $|K| \le 2+2\sum_{\mathcal{V}} |\mathcal{O}_V|$, gives us a bound on $|K|$ of a universal constant times $C^2\delta^{-2u(1+v)}$, which after possibly shrinking $u$ and $\delta$ is enough to prove the lemma. 
	\end{proof}
	Next we adapt~\cite[Lemma~4.8]{metglu}:
	\begin{lemma}\label{lemma:48}
		In the setting of Theorem~\ref{thm:main}, let $v \in (0,1)$ and let $u_0 = u_0(v)$ be as in Lemma~\ref{lemma:47}. Let $u \in (0,u_0]$, $\rho>2$ and $C>1$, and let $G_C = G_C(u,\rho)$. Let $z_1$, $z_2 \in V_\rho$ and let $\gamma_{z_1,z_2}$ be a $\fd_h$-geodesic from $z_1$ to $z_2$ contained in $V_{\rho/2}$, all chosen in a manner that is independent from $\eta$. For $\delta \in (0,1)$ let $K_{z_1,z_2}^\delta$ be the set of $k \in \mathbb{N}$ for which $\gamma_{z_1,z_2}$ intersects $\eta([(k-1)\delta^{d_\gamma/2},k\delta^{d_\gamma/2}])$. Then there is an exponent $\alpha > 0$ depending only on $\gamma$, $\mathfrak{w}^-$, $\mathfrak{w}^+$ and the exponent $\beta$ in Lemma~\ref{lemma:45}, and a deterministic constant $M = M(C,u,v,\rho)$, such that for each $\delta \in (0,1)$ we have
		\begin{equation}\label{eq:48}
			\mathbb{E}\left[|K_{z_1,z_2}^\delta| \cdot \mathbbm{1}_{G_C} \, \big| \, h, \gamma_{z_1,z_2} \right] \le M \delta^{-1-2v+\alpha(1+v)} \fd_h(z_1,z_2).
		\end{equation}
	\end{lemma}
	\begin{proof}
		It suffices to prove (\ref{eq:48}) for $\delta \le \delta_0$ where $\delta_0 = \delta_0(C,u,v,\rho)$ is as in Lemma~\ref{lemma:47}, since it can then be extended to $\delta \in (0,1)$ by (deterministically) increasing $M$. Fixing $z_1$, $z_2$ as in the statement, let $N := \lfloor \delta^{-(1+v)}\fd_h(z_1,z_2)\rfloor + 1$. For $j \in \{1, \ldots, N-1\}$ let $t_j = j\delta^{1+v}$ and let $t_N = \fd_h(z_1,z_2)$. Now define $\mathcal{V}_j = B_{\delta^{1+v}} (\gamma_{z_1,z_2}(t_j);\fd_h)$, where we parametrize the path $\gamma_{z_1,z_2}\colon [0,t_N] \to V_{\rho/2}$ by $\fd_h$-distance, so that the $\mathcal{V}_j$ cover $\gamma_{z_1,z_2}$. Let $J^{\delta}_{z_1,z_2}$ be the number of $j$ in $\{1, \ldots, N\}$ for which $\mathcal{V}_j$ intersects $\eta$. Lemma~\ref{lemma:47} gives that, on $G_C$, if $\delta \in (0, \delta_0]$ then for each $j$ there are at most $\delta^{-v}$ elements of~$K_{z_1,z_2}^\delta$ for which $\eta([(k-1)\delta^{d_\gamma/2},k\delta^{d_\gamma/2}])$ intersects $\mathcal{V}_j$. This shows that
		\begin{equation*}
			|K_{z_1,z_2}^\delta| \le \delta^{-v}|J_{z_1,z_2}^\delta|.
		\end{equation*}
		So it suffices to show that
		\begin{equation*}
			\mathbb{E} \left[ |J_{z_1,z_2}^\delta| \cdot \mathbbm{1}_{G_C} \, \big| \, h , \gamma_{z_1,z_2} \right] \le M \fd_h(z_1,z_2) \delta^{-1-v+\alpha(1+v)},
		\end{equation*}
		for appropriately chosen $\alpha$ and $M$. On $G_C$, condition (v) in Lemma~\ref{lemma:45} ensures that $\mathcal{V}_j$ is contained in the Euclidean ball $\widetilde{\mathcal{V}}_j := B (\gamma_{z_1,z_2}(t_j),C\delta^{\beta(1+v)})$. There exists $\alpha_0$ depending only on $\gamma$, $\mathfrak{w}^-$ and~$\mathfrak{w}^+$ such that for each $w \in V_{\rho/2}$ and $\ve>0$ we have
		\begin{equation*}
			\p [\eta \cap B(w,\ve) \neq \varnothing] \le c(\rho,C,\gamma,\mathfrak{w}^-,\mathfrak{w}^+) \ve^{\alpha_0}
		\end{equation*}
		(this is~\cite[Lemma~B.1]{metglu}), and since $(h,\gamma_{z_1,z_2})$ is independent of (the trace of) $\eta$ but determines~$\widetilde{\mathcal{V}}_j$, this probability bound applies here to give
		\begin{equation*}
			\p [\eta \cap \widetilde{\mathcal{V}}_j \neq \varnothing] \le c(\rho,C,\gamma,\mathfrak{w}^-,\mathfrak{w}^+) \delta^{\alpha_0\beta(1+v)}.
		\end{equation*} 
		Summing over $1 \le j \le N$ we get the result with $\alpha = \alpha_0 \beta$ and $M = c(\rho,C,\gamma,\mathfrak{w}^-,\mathfrak{w}^+)$.
	\end{proof}
	We now adapt~\cite[Lemma~4.9]{metglu}, which states that $\fd_h$ is equal to the metric gluing at quantum typical points.
	\begin{lemma}\label{lemma:49}
		In the setting of Theorem~\ref{thm:main}, let $\widetilde{d}_h$ be the quotient metric on $\h$ obtained by the metric gluing of $(U,\fd_{h|_U})$. Fix $R > 1$ and sample $z_1,z_2$ independently from the probability measure obtained by normalizing $\mu_h|_{V_R}$. Then almost surely we have $\fd_h(z_1,z_2) = \widetilde{d}_h(z_1,z_2)$.
	\end{lemma}
	\begin{proof}
		Let $v \in (0,\alpha/100)$ where $\alpha$ is as in Lemma~\ref{lemma:48} and let $u \in (0,u_0]$ where $u_0$ is as in Lemma~\ref{lemma:47}. Also fix $p \in (0,1)$ and $\ve > 0$. Choose $\rho > 2$ such that the event
		\begin{equation*}
			E_\rho := \{ \fd_h(z_1,z_2;V_{\rho/2}) - \ve < \fd_h(z_1,z_2) \le \rho\}
		\end{equation*}
		has probability at least $1-p/5$. We can do this because, by definition, $\fd_h(z_1,z_2)$ is the infimum of the $\fd_h$-lengths of paths between them that only intersect $\R$ finitely often, and by Remark~\ref{rem:avoid-r} we can replace a small segment of such a path near each intersection point with a path that stays in $\h$ with arbitrarily close $\fd_h$-length. Since we also know by Lemma~\ref{lem:41} that near-minimal paths from $z_1$ to $z_2$ cannot hit $\infty$, it follows that $\p[E_\rho] \to 1$ as $\rho \to \infty$.  
		
		Having chosen $\rho$ and $u$, choose $C = C(\rho,u)$ so that $G_C = G_C(u,\rho)$ has probability at least $1-p/5$. Work now on the event $E_\rho\cap G_C$. By~\cite[Thm~1.7]{gwynne_networks} there can only be finitely many geodesics from $z_1$ to $z_2$ w.r.t.\ the internal metric $\fd_h(\cdot,\cdot;V_{\rho/2})$ (which must also be geodesics from~$z_1$ to~$z_2$ w.r.t.~$\fd_h$); let $\gamma_{z_1,z_2}$ be the leftmost of these (i.e., when started from $z_1$, $\gamma_{z_1,z_2}$ stays to the left of all other $\fd_h(\cdot,\cdot;V_{\rho/2})$-geodesics from $z_1$ to $z_2$). By Lemma~\ref{lemma:48} we have
		\begin{equation*}
			\mathbb{E}\left[|K_{z_1,z_2}^\delta| \cdot \mathbbm{1}_{G_C \cap E_\rho} \, \big| \, h, \gamma_{z_1,z_2} \right] \le M \delta^{-1-2v+\alpha(1+v)} \rho,
		\end{equation*}
		and by taking a further expectation this bound also holds for $\mathbb{E}\left[|K_{z_1,z_2}^\delta| \cdot \mathbbm{1}_{G_C \cap E_\rho}\right]$. So by Markov's inequality, there exists $\delta_0 = \delta_0(u,v,C,\rho) > 0$ such that when $\delta \le \delta_0$, it holds with probability at least $1-p/2$ that $E_\rho \cap G_C$ occurs and 
		\begin{equation}\label{eq:kbound}
			|K_{z_1,z_2}^\delta| \le \delta^{-1-3v+\alpha(1+v)} \le \delta^{-1+\alpha/2}
		\end{equation}
		(the second inequality holds because $v < \alpha/100$). Now fix $\delta \in (0,\delta_0]$ and assume that $E_\rho \cap G_C$ occurs and (\ref{eq:kbound}) holds. We need to show that $\widetilde{d}_h(z_1,z_2) \le \fd_h(z_1,z_2)$ (note that the reverse inequality is clear by locality of the LQG metric, as pointed out in the discussion after the statement of Theorem~\ref{thm:main}). To this end we construct a path from $z_1$ to $z_2$ by concatenating finitely many paths each of which is contained in some $\overline{U}$, for $U \in \mathcal{U}^- \cup \mathcal{U}^+$.
		
		By condition (vi) in Lemma~\ref{lemma:45}, as long as $\delta \le C^{-2/d_\gamma}$ (which we can guarantee by possibly shrinking $\delta_0$), we have $\eta([(k-1)\delta^{d_\gamma/2},k\delta^{d_\gamma/2}]) \subseteq V_\rho$ for each $k \in K_{z_1,z_2}^\delta$, and thus these segments are disjoint from $\R$, so that we may choose $U_k \in \mathcal{U}^-$ such that $U_k$ intersects $V_\rho$ and such that $\eta([(k-1)\delta^{d_\gamma/2},k\delta^{d_\gamma/2}]) \subseteq \partial U_k$. Let $\tau_k$ and $\sigma_k$ be respectively the first and last times $\gamma_{z_1,z_2}$ hits $\eta([(k-1)\delta^{d_\gamma/2},k\delta^{d_\gamma/2}])$. Let $\widetilde{\gamma}_k$ be a $\fd_{h|_{U_k}}$-geodesic from $\gamma_{z_1,z_2}(\tau_k)$ to $\gamma_{z_1,z_2}(\sigma_k)$. By condition~(iii) in Lemma~\ref{lemma:45}, almost surely on $G_C$ we have $\mathrm{diam}(\eta([(k-1)\delta^{d_\gamma/2},k\delta^{d_\gamma/2}]);\fd_{h|_{U_k}}) \le C\delta^{1-ud_\gamma/2}$, and thus
		\begin{equation}\label{eq:lenbound}
			\mathrm{length}(\widetilde{\gamma}_k;\fd_{h|_{U_k}}) \le C\delta^{1-ud_\gamma/2} \quad \forall k \in K_{z_1,z_2}^\delta.
		\end{equation}
		Pick $k_1 \in K_{z_1,z_2}^\delta$ with $\tau_{k_1}$ minimal, and inductively define $k_2, \ldots, k_{|K_{z_1,z_2}^\delta|}$ by picking $k_j \in K_{z_1,z_2}^\delta$ such that $\tau_{k_j}$ is the smallest $\tau_k$ for $k \in K_{z_1,z_2}^\delta$ for which $\tau_k \ge \sigma_{k_{j-1}}$, if this exists; if there is no such $\tau_k$ let $k_j = \infty$. Let $J$ be the smallest $j \in \mathbb{N}$ for which $k_j =\infty$. Let $\mathring{\gamma}_1 = \gamma_{z_1,z_2}|_{[0,\tau_{k_1}]}$, let $\mathring{\gamma}_J = \gamma_{z_1,z_2}|_{[\sigma_J,\fd_h(z_1,z_2)]}$ and let $\mathring{\gamma}_j = \gamma_{z_1,z_2}|_{[\sigma_{j-1},\tau_j]}$ for $2 \le j \le J-1$. Then for $1 \le j \le J$, the curve~$\mathring{\gamma}_j$ does not hit $\eta$ except at its endpoints, so that we can find $\mathring{U}_j \in \mathcal{U}^- \cup \mathcal{U}^+$ such that $\mathring{\gamma}_j \subseteq \overline{\mathring{U}_j}$, and by locality we have $\mathrm{length}(\mathring{\gamma}_j;\fd_{h|_{\mathring{U}_j}}) = \mathrm{length}(\mathring{\gamma}_j;\fd_h)$ for each $j$.
		We now concatenate the curves $\mathring{\gamma}_1, \widetilde{\gamma}_{k_1}, \mathring{\gamma}_2, \widetilde{\gamma}_{k_2}, \ldots, \mathring{\gamma}_{J-1}, \widetilde{\gamma}_{k_{J-1}}, \mathring{\gamma}_J$, to get a path $\widetilde{\gamma}$ from $z_1$ to $z_2$ such that
		\begin{align*}
			\nonumber \widetilde{d}_h(z_1;z_2) &\le \sum_{j=1}^{J-1} \mathrm{length} (\widetilde{\gamma}_{k_j};\fd_{h|_{U_{k_j}}}) + \sum_{j=1}^J \mathrm{length}(\mathring{\gamma}_j;\fd_{h|_{\mathring{U}_j}}) \\
			\nonumber &\le \sum_{j=1}^{J-1} \mathrm{length} (\widetilde{\gamma}_{k_j};\fd_{h|_{U_{k_j}}}) + \fd_h(z_1,z_2)+\ve \\
			&\le C\delta^{\alpha/2-ud_\gamma/2} + \fd_h(z_1,z_2)+\ve,
		\end{align*}
		where the last inequality comes from (\ref{eq:kbound}) to bound $J$ by $\delta^{-1+\alpha/2}$ and (\ref{eq:lenbound}) to bound the length of each $\widetilde{\gamma}_{k_j}$. By possibly shrinking $u_0$ we can ensure that $ud_\gamma < \alpha$, so that sending $\delta \to 0$ gives $\widetilde{d}_h(z_1,z_2) \le \fd_h(z_1,z_2)+\ve$ as required. Since $p$ and $\ve$ can be made arbitrarily small, we are done.
	\end{proof}
	The last step to prove Theorem~\ref{thm:main} is the same as in the proof of~\cite[Thm~1.5]{metglu}. (Essentially, we now have that $\widetilde{d}_h$ and $\fd_h$ agree on a set of $\mu_h$-full measure, which is dense since open sets have positive $\mu_h$-measure, so we can conclude quickly by an approximation argument.)
	\subsection{Proofs of Theorems~\ref{thm:16},~\ref{thm:17} and~\ref{thm:18}}
	We now turn to the proofs of Theorems~\ref{thm:16},~\ref{thm:17} and~\ref{thm:18}. In fact, as with~\cite[Thm~1.6]{metglu}, in the case that $\mathfrak{w} \ge \gamma^2/2$ (so that $(U,h|_U)$ is a thick wedge), the proof of Theorem~\ref{thm:16} is essentially the same as that of the previous theorem, so we just need to treat the case where cutting along $\eta$ gives a thin wedge. The reason this case is more difficult is that we have to approximate geodesics with paths that avoid the points at which $\eta$ intersects itself. However, we can still deduce this case from the previous results. 
	\begin{proof}[Proof of Theorem \ref{thm:16} in the case $\mathfrak{w} \in (0,\gamma^2/2)$]
		Fix $z \in \mathbb{C}\setminus\{0\}$ and $0<r<s<s'<|z|$. Let~$\tau_1$ be the first time that~$\eta$ hits $\partial B(z,r)$ and let $\sigma_1$ be the first time after $\tau_1$ that $\eta$ hits $\partial B(z,s)$. Having defined $\tau_j$, $\sigma_j$, let $\tau_{j+1}$ be the first time after $\sigma_j$ that $\eta$ hits $\partial B(z,r)$ and let $\sigma_{j+1}$ be the first time after $\tau_{j+1}$ that $\eta$ hits $\partial B(z,s)$. We will show that for each~$j$ it is almost surely the case that the internal metric $\fd_h(\cdot,\cdot;B(z,r))$ agrees with the metric gluing of the components of $B(z,r)\setminus\eta|_{[0,\sigma_j]}$ along $\eta|_{[0,\sigma_j]}$. This suffices to prove the theorem, since then the result almost surely holds for all~$j$, all $z\in\mathbb{Q}^2\setminus\{0\}$ and all $0<r<s<|z|$ rational, so that we can split any path not hitting 0 into finitely many pieces each contained in a ball $B(z,r)$ for which the result holds. Then the length of each such piece is the same according to $\fd_h(\cdot,\cdot;B(z,r))$ and the metric gluing across $\eta$ (which, since~$\eta$ is transient by~\cite[Thm~1.12]{ig4}, is the same as the metric gluing along $\eta|_{[0,\sigma_j]}$ for $j$ sufficiently large).
		
		\begin{figure}\label{fig:s6}
			\includegraphics[width=0.9\textwidth]{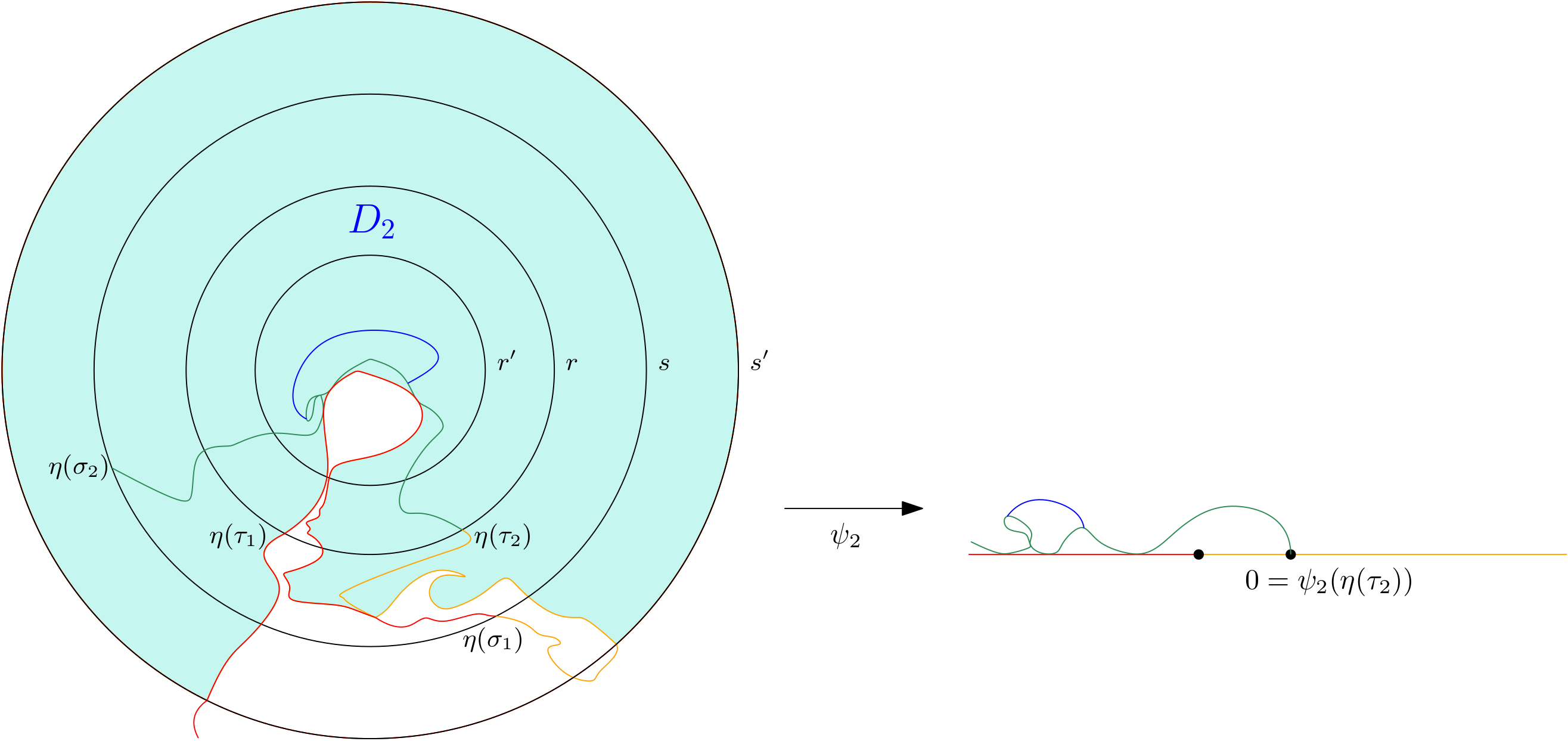}
			\caption{In order to deduce Theorem~\ref{thm:16} for thin wedges from Theorem~\ref{thm:main}, we draw the whole-plane $\SLE_{\gamma^2}(\mathfrak{w}-2)$ curve $\eta$ up to a stopping time then map a domain bounded by $\eta$ and a circular arc to $\h$. The law of the image of the remaining part of $\eta$ up to a later stopping time is absolutely continuous w.r.t.\ that of a chordal $\SLE_\kappa(\mathfrak{w}-2)$, so this puts us in the setting of Theorem~\ref{thm:main}.}
		\end{figure}
		
		We proceed by induction on $j$; first we consider the case $j=1$. The conditional law of $\eta|_{[\tau_1,\sigma_1]}$ given $\eta|_{[0,\tau_1]}$ is that of a radial $\mathrm{SLE}_{\gamma^2}(\mathfrak{w}-2)$ in the unbounded component $\widetilde{D}_1$ of $\mathbb{C}\setminus\eta|_{[0,\tau_1]}$, started from $\eta(\tau_1)$, targeted at $\infty$ and stopped at time $\sigma_1$, and thus has the same law (up to time change) as a chordal $\mathrm{SLE}_{\gamma^2}(\mathfrak{w}-2)$ in $\widetilde{D}_1$ from $\eta(\tau_1)$ targeted at $\infty$ and stopped upon hitting $\partial B(z,s)$~\cite[Thm~3]{sw}. Moreover, if we define the domain $D_1$ to be the component of $B(z,s')\setminus\eta|_{[0,\tau_1]}$ containing $\eta([\tau_1,\sigma_1])$ (note that this component is determined by $\eta|_{[0,\tau_1]}$), this latter law is mutually absolutely continuous with that of a chordal $\mathrm{SLE}_{\gamma^2}(\mathfrak{w}-2)$ in $\widetilde{D}_1$ from $\eta(\tau_1)$ targeted at~$\infty$ and stopped upon hitting $\partial B(z,s)$, and indeed the Radon--Nikodym derivatives between the two laws are bounded by~\cite[Lemma~2.8]{intersections}. Therefore, if we now fix (in some way which is measurable w.r.t.~$\eta|_{[0,\tau_1]}$) a conformal map $\psi_1\colon D_1\to\h$ such that $\psi_1(\eta(\tau_1))=0$, then the law of $\psi_1\circ\eta([\tau_1,\sigma_1])$ is absolutely continuous up to time change w.r.t.\ the law of a chordal $\mathrm{SLE}_{\gamma^2}(\mathfrak{w}-2)$ from~0 to~$\infty$ in $\h$ stopped upon exiting $\psi_1(B(z,s))$, and the Radon--Nikodym derivatives between the two laws are bounded independently of the choice of $\psi_1$.
		
		Letting $h^{\psi_1} = h\circ\psi_1^{-1} + Q\log{|(\psi_1^{-1})'|}$ be the covariantly transformed field on $\h$, we argue that if $0<r'<r$, then the law of the pair $(\psi_1(B(z,r')),h^{\psi_1}|_{(\psi_1(B(z,r'))})$ is absolutely continuous w.r.t.\ that of $(\psi_1(B(z,r')),h^F|_{(\psi_1(B(z,r'))})$ where $h^F$ is a free-boundary GFF on $\h$ (say, normalized so that $h^F_1(0)=0$). This follows because $\psi_1(B(z,r'))$ is at positive Euclidean distance from $\partial\h$ and the laws of the two GFF variants are mutually absolutely continuous away from the boundary (which can be seen by coupling them so that their difference is a random harmonic function and using the Girsanov theorem to express the Radon--Nikodym derivative in terms of this harmonic function). We can thus apply the proof of Theorem~\ref{thm:main} to $(\h,h^F,0,\infty)$ (nothing changes, since the required GFF estimates in Lemma~\ref{lemma:45} are proved for $h^F$ anyway) and, by absolute continuity, deduce the conclusion of that theorem for $h^{\psi_1}$. That is to say, almost surely, for each rational $r'\in(0,r)$, the length of any path in $B(z,r')$ is the same w.r.t.\ $\fd_h(\cdot,\cdot;B(z,r))$ and the metric gluing along~$\eta|_{[0,\sigma_1]}$. This completes the base case.
		
		Suppose that the result holds for $j\ge 1$; we prove that it holds also for $j+1$. By the induction hypothesis, it holds almost surely that if $w_1$, $w_2$ are any two distinct points in $B(z,r)$, then for each $\ve>0$ there is a path $P$ in $B(z,r)$ which crosses $\eta|_{[0,\sigma_j]}$ only finitely many times whose $\fd_h(\cdot,\cdot;B(z,r))$-length is at most $\fd_h(w_1,w_2;B(z,r))+\ve$. We thus aim to show that it is almost surely the case that each path $\widetilde{P}$ in $B(z,r)$ which does not intersect $\eta|_{[0,\sigma_j]}$ except possibly at the endpoints of $\widetilde{P}$ has the same length w.r.t.~$\fd_h(\cdot,\cdot;B(z,r))$ as w.r.t.\ the metric gluing along $\eta|_{[0,\sigma_{j+1}]}$. This implies that if $w_1$ and $w_2$ are quantum typical points (i.e., sampled independently from $h|_{B(z,r)}$ normalized to be a probability measure), then $w_1$, $w_2$ have the same distance w.r.t.~$\fd_h(\cdot,\cdot;B(z,r))$ and the gluing along $\eta|_{[0,\sigma_{j+1}]}$ (since we can choose an almost-minimal path $P$ as above between~$w_1$ and $w_2$ and split into subpaths $\widetilde{P}$ with the same length according to each of the two metrics). We can then conclude that these two metrics on $B(z,r)$ are equal using the same argument as at the end of the proof of~\cite[Thm~1.5]{gm}.
		
		Analogously to the base case, let $D_{j+1}$ be the component of $B(z,s')\setminus\eta|_{[0,\tau_{j+1}]}$ containing $\eta([\tau_{j+1},\sigma_{j+1}])$ and, in some way which is measurable w.r.t.~$\eta|_{[0,\tau_{j+1}]}$, fix a conformal map $\psi_{j+1}:D_{j+1}\to\h$ such that $\psi_{j+1}(\eta(\tau_{j+1}))=0$ and let $h^{\psi_{j+1}} = h\circ\psi_{j+1}^{-1} + Q\log{|(\psi_{j+1}^{-1})'|}$ be the covariantly transformed field on $\h$.
		
		As before, the conditional law of $\eta|_{[\tau_{j+1},\sigma_{j+1}]}$ given $\eta|_{[0,\tau_{j+1}]}$ is (up to time change) that of a chordal $\mathrm{SLE}_{\gamma^2}(\mathfrak{w}-2)$ in the unbounded component $\widetilde{D}_{j+1}$  of $\mathbb{C}\setminus\eta|_{[0,\tau_{j+1}]}$, started from $\eta(\tau_{j+1})$ and stopped upon hitting $\partial B(z,s)$, and thus the law of $\psi_{j+1}\circ\eta([\tau_{j+1},\sigma_{j+1}])$ is absolutely continuous up to time change w.r.t.\ that of a chordal $\SLE_{\gamma^2}(\mathfrak{w}-2)$ from~0 to~$\infty$ in $\h$ stopped upon exiting $\psi_{j+1}(B(z,s))$. 
		
		Moreover, for each $\delta > 0$, $r' \in (0,r)$, the law of the pair$$(\psi_{j+1}(B(z,r')\setminus B(\eta([0,\sigma_j]),\delta)),h^{\psi_{j+1}}|_{\psi_{j+1}(B(z,r'))})$$is absolutely continuous w.r.t.\ that of $(\psi_{j+1}(B(z,r')\setminus B(\eta([0,\sigma_j]),\delta)),h^F|_{(\psi_{j+1}(B(z,r'))})$, since the set $\psi_{j+1}(B(z,r')\setminus B(\eta([0,\sigma_j]),\delta))$ will have positive Euclidean distance from $\partial\h$. We can thus argue as in the case $j=1$ that, almost surely, for any $\delta > 0$ and $r' \in (0,r)$, any path in $B(z,r')\setminus B(\eta([0,\sigma_j]),\delta)$ has the same length w.r.t.~$\fd_h(\cdot,\cdot;B(z,r))$ and the gluing along $\eta|_{[0,\sigma_{j+1}]}$. This suffices to complete the inductive step, since w.r.t.\ either metric we can find the length of any path in $B(z,r')$ intersecting $\eta|_{[0,\sigma_{j+1}]}$ only at its endpoints by considering the amount of length it accumulates in $B(z,r')\setminus B(\eta([0,\sigma_j]),\delta)$ and sending $\delta$ to 0.
	\end{proof}
	Theorem~\ref{thm:17} follows by the same method as in~\cite{metglu}. The left boundary $\eta_L$ of $\eta'((-\infty,0])$ is an $\mathrm{SLE}_{\gamma^2}(2-\gamma^2)$ by~\cite[Footnote 4]{dms} and~\cite[Thm~1.1]{ig4}, and then~\cite[Thm~1.5]{dms} gives that $(\mathbb{C}\setminus \eta_L,h|_{\mathbb{C}\setminus \eta_L},0,\infty)$ is a wedge of weight $4-\gamma^2$. We now apply Theorem~\ref{thm:16}. 
	By~\cite[Thm~1.11]{ig4}, the conditional law of the right boundary $\eta_R$ of $\eta'((-\infty,0])$ given $\eta_L$ is a $\mathrm{SLE}_{\gamma^2} (-\gamma^2/2;-\gamma^2/2)$. Thus $\eta_R$ cuts the wedge into two wedges of weight $2-\gamma^2/2$ and now we deduce Theorem~\ref{thm:17} by applying Theorem~\ref{thm:main}. 
	Finally, Theorem~\ref{thm:18} follows by the same absolute continuity argument (between quantum spheres and $\gamma$-quantum cones) as in~\cite{metglu}.
	\bibliographystyle{alpha}
	\bibliography{newbib}
\end{document}